\documentclass[11pt,a4paper]{article}
\usepackage{epsf, graphicx}
\usepackage{latexsym,amsfonts,amsbsy,amssymb}
\usepackage{amsmath,amsthm,mathrsfs}
\usepackage{cite}
\usepackage{cleveref}
\usepackage{enumerate}
\usepackage{enumerate}
\usepackage{enumitem}
\usepackage{color}
\usepackage{tikz}\usetikzlibrary{positioning} 

\makeatletter

\@addtoreset{equation}{section} \makeatother
\makeatletter \@addtoreset{equation}{section} \makeatother
\setlist[enumerate,1]{label=(\arabic*).,font=\textup,
leftmargin=7mm,labelsep=1.5mm,topsep=0mm,itemsep=-0.8mm}
\setlist[enumerate,2]{label=(\alph*).,font=\textup,
leftmargin=7mm,labelsep=1.5mm,topsep=-0.8mm,itemsep=-0.8mm}

\newtheorem{thmA}{Theorem}
\newtheorem{coroA}[thmA]{Corollary}
\newtheorem{thm}{Theorem}[section]
\newtheorem{prop}[thm]{Proposition}

\newtheorem{lem}[thm]{Lemma}
\newtheorem{coro}[thm]{Corollary}
\newtheorem*{cconj}{Conjecture}

\theoremstyle{definition}
\newtheorem{defi}[thm]{Definition}
\newtheorem{remark}[thm]{Remark}

\label{Abbreviation_order}
\newcommand{\semi}{\rtimes}
\newcommand{\zg}{\unlhd}
\newcommand{\ra}{\rightarrow}
\newcommand{\hH}{\mathcal{H}}

\newcommand{\Pj}{\mathcal{P}}
\newcommand{\PQ}{\mathcal{Q}}

\newcommand{\Q}{\mathbb{Q}}
\newcommand{\ZZ}{{\bf Z}}
\newcommand{\C}{{\bf C}}
\newcommand{\N}{{\bf N}}
\newcommand{\E}{{\bf E}}
\newcommand{\bO}{{\bf O}}
\newcommand{\bS}{{\bf S}}

\newcommand{\Qp}{\mathbb{Q}_p}
\newcommand{\Gal}{{\rm Gal}}
\newcommand{\Rad}{{\rm Rad}}
\newcommand{\Irr}{{\rm Irr}}
\newcommand{\IBr}{{\rm IBr}}
\newcommand{\IBrd}{{\rm IBrd}}
\newcommand{\soc}{{\rm soc}}
\newcommand{\dz}{{\rm dz}}
\newcommand{\GL}{{\rm GL}}
\newcommand{\Aut}{{\rm Aut}}
\newcommand{\Inn}{{\rm Inn}}
\newcommand{\DGN}{{\rm DGN}}

\newcommand{\W}{{\mathcal{W}}}
\newcommand{\rw}{{\mathcal{RW}}}
\newcommand{\Bl}{{\rm Bl}}
\newcommand{\bl}{{\rm bl}}

\newcommand{\which}{{\,|\,}}
\newcommand{\bwhich}{{\,\big|\,}}
\newcommand{\zs}[1]{\hspace{0.1mm}^{#1}\hspace{-0.4mm}}
\newcommand{\cW}{\mathscr{W}}
\newcommand{\cE}{\mathcal{E}}

\begin{document}

\title{\bf A reduction theorem for the Navarro Alperin weight conjecture}
\date{}
 \author{
  Zhicheng Feng \footnote{Shenzhen International Center for Mathematics, Southern University of Science and Technology, Shenzhen 518055, China. E-mail: fengzc@sustech.edu.cn}\quad ,
  Qulei Fu \footnote{Corresponding author, Shenzhen International Center for Mathematics, Southern University of Science and Technology, Shenzhen 518055, China. E-mail: quleifu@outlook.com}\quad ,
 Yuanyang Zhou  \footnote{School of Mathematics and Statistics, Cental China Normal University, Wuhan 430079, China. E-mail: zhouyuanyang@ccnu.edu.cn}
 }
 \maketitle

 
\begin{abstract}
The Alperin weight conjecture has been reduced to simple groups by Navarro and Tiep. In this paper, we investigate the Navarro Alperin weight conjecture, which includes Galois automorphisms and group automorphisms in comparison with the original version, and give a reduction to simple groups. As an application, we prove the conjecture for the finite groups with abelian Sylow $2$-subgroups.
\end{abstract}


\section{Introduction}\label{section-intro}
The Alperin weight (AW) conjecture, announced in 1986 by Alperin \cite{Alp-conj},  is one of the main problems in the representation theory of finite groups. 
Fix a prime $p$.
Let $G$ be a finite group. A ($p$-){\bf weight} of $G$ is a pair $(Q,\delta)$, where $Q$ is a $p$-subgroup of $G$ and $\delta$ is an irreducible complex character of $\N_G(Q)/Q$ such that $\delta(1)_p=|\N_G(Q)/Q|_p$ (here $n_p$ denotes the largest power of $p$ dividing the integer $n$). Denote by $\W(G)$ the set of weights of $G$. The group $G$ acts on the set $\W(G)$ by conjugation, and we denote by $\W(G)/\sim_G$ the set of $G$-orbits of $\W(G)$. As usual, $\IBr(G)$ denotes the set of irreducible ($p$-)Brauer characters of $G$.

\begin{cconj}[Alperin weight conjecture]
  For any finite group $G$, we have that $$|\IBr(G)|=|\W(G)/\sim_G|.$$
\end{cconj}

At the time being, the most hopeful method is to reduce the problem to simple groups. In 2011, Navarro and Tiep \cite{Nav1} achieved a reduction of the AW conjecture to the so-called inductive AW condition for simple groups (i.e., AWC-good for simple groups); they proved that if every  finite non-abelian  simple group satisfies the inductive AW condition, then the AW conjecture holds for every finite group.
For recent developments concerning the verification of the inductive conditions, we refer to the survey paper \cite{FZ22}.

The McKay conjecture is another long-standing conjecture in the representation theory of finite groups, and it has been reduced to simple groups by Issacs, Malle and Navarro \cite{Isaacs_Reduction_of_McKay_conj}. Using this reduction theorem, the McKay conjecture has finally been proved \cite{MS16,CS23}.
In 2004, Navarro strengthened the McKay Conjecture by considering the role of Galois automorphisms in \cite{Nav-Ann}, where he conjectured that the bijection in the McKay conjecture can be chosen to be equivariant under certain Galois automorphisms action. 
This strengthened conjecture is called the McKay--Navarro conjecture, also referred to as the Galois--McKay conjecture, and is a refinement of the original form.
In 2020, Navarro, Sp\"ath and Vallejo \cite{Nav-GMckay} proved a reduction theorem of the McKay--Navarro conjecture to   finite non-abelian simple groups, affording a promising path for the proof of this conjecture.

 In \cite{Nav-Ann}, Navarro also pointed out that a similar idea can be applied to the Alperin weight conjecture. 
 We call this refinement of the Alperin weight conjecture  the \textbf{Navarro Alperin weight (NAW) conjecture}, or the \textbf{Galois Alperin weight conjecture}.
 Moreover, it is implicit in the reduction of Navarro and Tiep \cite{Nav1} that the AW bijection is equivariant under group automorphisms if the inductive condition holds for all simple groups.
In this paper, we adopt the formulation of the NAW conjecture that incorporates both Galois automorphisms and group automorphisms. 
The existence of such an equivariant bijection is a crucial requirement for establishing the corresponding inductive condition (see Definition \ref{Def:ind-GAW}).
 
To be more precise with the NAW conjecture, we introduce some notation.
Let $\Qp$ be the field of $p$-adic numbers, and let $\mathcal{K}=\Qp(w_n)$ be the field extension of $\Qp$ by a primitive $n$-th root of unity $w_n,$ where the integer $n$ is  big enough such that $\mathcal{K}$ splits for every considered finite group.
 Noticing that $\mathcal{K}$ is a complete discrete valuation field with the $p$-adic valuation, we let $\mathcal{O}$ and $F$ be the valuation ring and residue field of $\mathcal{K},$ respectively.
 We also assume that $F$ splits for every considered finite group.
 Note that $F$ is a finite field.
 Let $^{\ast}:\mathcal{O} \ra F$ be the canonical surjection.
 Let $\mathbb{F}_p$ be the prime subfield of $F.$
We assume that the maps in $\Gal(\mathcal{K}/\Q_p)$ and $\Gal(F/\mathbb{F}_p)$ are composed from the left, and we take exponential notation, writing $x^{\sigma}$ for the image of $x\in\mathcal{K}\sqcup F$ under the map $\sigma\in \Gal(\mathcal{K}/\Q_p) \sqcup \Gal(F/\mathbb{F}_p),$ respectively.
  Let $\hH=\Gal(\mathcal{K}/\Q_p),$  and we  mention that $\hH$ is a finite abelian group.
 Note that the $p$-modular system $(\mathcal{K},\mathcal{O},F)$ is $\hH$-fixed, and that 
the natural group homomorphism $f:\Gal(\mathcal{K}/\Qp)\ra \Gal(F/\mathbb{F}_p)$ induced by \[(x^{\ast})^{f(\sigma)}=(x^{\sigma})^{\ast},\forall \sigma\in\hH,\forall x\in \mathcal{O},\] is surjective.
 Through the group homomorphism $f,$ we also regard $\hH$ as a group of automorphisms of $F,$ and it contains every element in $\Gal(F/\mathbb{F}_p).$  
 Let $K$ be the subfield of $\mathcal{K}$ that is an extension of $\mathbb{Q}$ by all the $p'$-th roots of unity in $\mathcal{K},$ where $\mathbb{Q}$ is the field of rational numbers. Then the values of Brauer characters are taken over the field $K.$ Let $\sigma_p$ be the unique element in $\Gal(K/\mathbb{Q})$ such that $\epsilon^{\sigma_p}=\epsilon^p$ for every $p'$-th root of unity $\epsilon$ in $K.$ Then it is well known that the restriction of $\Gal(\mathcal{K}/\Qp)$ to $\Gal(K/\mathbb{Q})$ yields exactly the subgroup $\langle\sigma_p\rangle$ of $\Gal(K/\mathbb{Q}).$ See \cite[Section 2]{Turull_The_strengthened_Alperin_Weight     Conjecture_for.} for more details.

\begin{cconj}[NAW conjecture]
For any finite group $G,$ and let $\hH$ be described as above. Then there exists an $\hH\times\Aut(G)$-equivariant bijection from $ \IBr(G)$ to $\W(G)/\sim_G.$ 
\end{cconj}

In this paper, we define the inductive {NAW} condition for finite non-abelian simple groups (see Definition \ref{Def:ind-GAW}), and establish the following reduction theorem. 
We say that a group $X$ is {involved} in a finite group $G$ if there are two subgroups $N, H$ of $G$ such that $N$ is a normal in $H$ and $H/N\cong X.$

\begin{thmA}\label{thm-reduction}
  For any finite group $G.$  If the inductive NAW condition (see Definition \ref{Def:ind-GAW}) holds for every finite non-abelian simple group involved in $G$ and of order divisible by $p$,  then the NAW conjecture is true for $G$ (at the prime $p$).
\end{thmA}

The proof of Theorem A refines the strategy of Navarro--Tiep \cite{Nav1}. At the time being, it seems that the language of character triples,  together with its ordering relation, is best suited in describing the inductive conditions. In order to consider Galois automorphisms, character triples are needed to be replaced by $\hH$-triples, which are introduced in Navarro--Sp\"ath--Vallejo \cite{Nav-GMckay}. 
It is well known that any (Brauer) character triple is isomorphic to one in which the normal subgroup is  central. However, incorporating Galois automorphisms into this framework presents significant difficulties, and this constitutes the main obstacle in the reduction theorem for the McKay--Navarro conjecture. In our reduction of the NAW conjecture, we encounter the same problem, which we overcome by restricting our attention to representations over the finite field \(F\), where the Schur indices of group algebras are always trivial.
We also need to consider the Galois automorphisms and group automorphisms on irreducible Brauer characters above the Dade--Glauberman--Nagao correspondence, accomplished by the second author in \cite{Fu}. 

Recently, Mart\'inez, Rizo and Rossi \cite{MRR23} proved that the inductive blockwise Alperin weight condition holds for all finite groups provided it holds for all simple groups; this result also establishes the group automorphism equivariance required in the above theorem.

Since there doesn't exists any  finite non-abelian simple group of order divisible by $p$ involved in a $p$-solvable finite group,  the following result is clear from Theorem~\ref{thm-reduction}. This provides a new proof of certain results of Turull \cite{Turull_The_strengthened_Alperin_Weight Conjecture_for.}.

\begin{coroA}
  Let $G$ be a $p$-solvable finite group. Then the NAW conjecture is true for $G$ at the prime $p$.
\end{coroA}

In the original paper of Navarro--Tiep \cite{Nav1}, the inductive AW condition has been verified to hold for finite simple groups of Lie type at their defining characteristic, enhancing the understanding of these conjectures in this case and providing essential examples.
 For the same reason, we prove the following theorem.

\begin{thmA}\label{thm:def-char}
  The inductive NAW condition holds for every finite non-abelian simple group of Lie type  at their defining characteristic.
\end{thmA}

As a consequence, we  prove the following result.

  \begin{thmA}
Suppose that $G$ has an abelian Sylow 2-subgroup. Then the NAW conjecture is true for $G$ at every prime.
  \end{thmA}
  
Though certain simple groups are proved to satisfy the inductive NAW condition in this paper, the verification of the inductive NAW condition for finite simple groups remains a serious challenge since it requires a vast knowledge of the character values of simple groups which is still not completely comprehended.
  
 In Section \ref{sec_Preli} we introduce some basic notation and facts. Then in Section \ref{sec_centralize},
  we enhance the well-known fact that any Brauer character triple can always be reduced to one with a central $p'$-subgroup by considering Galois automorphisms and group automorphisms. Afterwards in Section \ref{sec_DGN-cor}, we give  an equivariant bijection of irreducible Brauer characters above the Dade--Glauberman--Nagao correspondence. 
  We introduce the ordering relation of Brauer $\hH$-triples and related properties in Sections \ref{sec_Character_pairs} and \ref{sec_new-H-tri}, and define the inductive NAW condition for simple groups  in Section \ref{sec_ind_GAW}.  These preparations lead to the final reduction of the NAW conjecture in Section \ref{Sec_red}. In Section \ref{sec_defining_char}, we consider groups of Lie type at their defining characteristic and prove Theorem C, while we prove Theorem D in Section \ref{sec_abelian2}.


\section{Preliminaries}\label{sec_Preli}
Throughout this paper, $p$ is a fixed prime, and the $p$-modular system $(\mathcal{K},\mathcal{O},F)$ together with the Galois group $\hH$ are as described in Section \ref{section-intro}.
All groups considered are finite; all modules are right modules and are of finite dimension.
For ordinary and Brauer characters our notation follows \cite{Nagao}.

Let $\mathscr{A}$ be a group acting (from the right) on the sets $X_1$ and $X_2$. Then a map $f:X_1\ra X_2$ is called $\mathscr{A}$-equivariant if $f(x)^a=f(x^a)$ for any $x\in X_1$ and $a\in \mathscr{A}.$

Let $G$ be a finite group, we denote by $G_{p'}$ the set of $p$-regular elements of $G.$ 
An ordinary character $\chi\in\Irr(G)$ is said to have ($p$-)defect zero if $\chi(1)_p=|G|_p.$ We write $\dz(G)$ for the set of irreducible defect zero characters of $G.$ If $\chi\in\dz(G),$ then $\chi^{\circ}:=\chi_{G_{p'}}$ is an irreducible Brauer character of $G.$ 
In fact, the map $\chi \mapsto \chi^{\circ}$ induces a bijection between the set of irreducible defect zero characters of $G$ and the set of irreducible Brauer characters of $G$ whose corresponding irreducible $FG$-modules have trivial vertices (for the definition of vertices of indecomposable $FG$-modules, see \cite[Chapter IV, Section 3]{Nagao}).

Let $N$ be a subgroup of  $G.$ Let $\chi\in\Irr(G)\cup\IBr(G)$ and $\theta\in\Irr(N)\cup\IBr(N).$ Then we denote by $\chi_N$ the restriction of $\chi$ to $N,$ and $\theta^G$ the induction of $\theta$ to $G.$
We denote by $\IBr(G \which \theta)$ and $\Irr(G \which \theta)$ the set of irreducible Brauer and ordinary  characters of $G$ lying over $\theta$, respectively. 
Assume that $\theta\in\IBr(N)$ and  $m$ is a positive integer that is a power of $p$.
We denote by $\IBr(G \which\theta, m)$ (resp. $\IBr(G \which m)$)  the subset of $\IBr(G \which \theta)$ (resp. $\IBr(G)$) consisting of characters whose corresponding modules have vertices of order $m$.

Let $\Aut(G)$ be the automorphism group of $G,$  and like the group $\hH$ of Galois automorphisms,  we require that  elements in $\Aut(G)$ are composed from the left.
For $x\in \mathcal{O} \cup F,g\in G,\sigma\in\hH,\phi\in\Aut(G)$,  we write $x^{\sigma}$ and $g^{\phi}$ for the images of $x$ and $g$ under the map $\sigma$ and $\phi,$ respectively.
Let $a=(\sigma,\phi)\in\hH\times\Aut(G).$ 
If $H$ is a subgroup of $G$ and $\chi\in\IBr(H)\cup\Irr(H)$, then we define $H^a=H^{\phi}$ and $\chi^a\in\IBr(H^a)\cup\Irr(H^a)$
such that $\chi^a(x)=\chi(x^{\phi^{-1}})^{\sigma}.$
For $g\in G,$ we define $\chi^g$ to be the character of $H^g$ such that $\chi^g(x)=
\chi(gxg^{-1}).$
We can similarly define the action of $\hH\times\Aut(G)$ on group representations, as will be described in Section \ref{sec_Character_pairs}.  Actions on modules will be discussed in Section \ref{sec_centralize}.  

A $p$\textbf{-radical} subgroup of $G$ is a $p$-subgroup $Q$ of $G$ such that $\bO_p\big(\N_G(Q)\big)=Q$ \cite{Nav1}, where $\bO_p\big(\N_G(Q)\big)$ is the maximal normal $p$-subgroup of $\N_G(Q)$. We denote by $\Rad(G)$ the set of $p$-radical subgroups of $G$.
It's known that if  $\dz(G)$ is nonempty, then $\bO_p(G)=1$ (see \cite[Section 2]{Nav1}). 
Thus if $(Q,\delta)$ is a $p$-weight of $G$, then $Q\in\Rad(G).$

\begin{lem}\label{thm_6.02}
  Let $Z$ be a normal subgroup of a finite group $G.$ Then we have that
  \begin{enumerate}
    \item If $Z$ is a central $p'$-subgroup of $G,$ then there is a natural bijection $$\Rad(G)\ra\Rad(G/Z), S\mapsto SZ/Z.$$ 
    \item Let $Z$ be a normal $p$-subgroup of $G,$  and denote by $\bar{H}=HZ/Z$ for any subgroup $H$ of $G.$ 
        Then $Z$ is contained in any $p$-radical subgroup of $G$ and there exists a natural bijection $\Rad(G)\ra\Rad(\bar{G}),S\mapsto\bar{S}.$ 
        For any $\varphi\in\IBr(G),$  let $\bar{\varphi}\in\IBr(\bar{G})$ be the deflation of $\varphi$. If $Q$ is a vertex of $\varphi,$ then $Z\leqslant Q$ and $\bar{Q}$ is a vertex of $\bar{\varphi}.$ 
   \end{enumerate}
\end{lem}
\begin{proof}
  The condition (1) and the first part of the condition (2) follow from \cite[Lemma 2.3]{Nav1}. The second part of Condition (2) follows from \cite[Chapter IV, Theorem 2.2(2), Lemma 3.4(ii)]{Nagao}, noticing that $Z$ acts trivially on any irreducible $FG$-module.
\end{proof}

Let $Q\in\Rad(G).$ 
By the last statement of the above lemma, there is a natural bijection 
\begin{equation}\label{equ:dz-bij-vert.}
  \dz(\N_G(Q)/Q)\mapsto \IBr(\N_G(Q)\which |Q|),\ \theta\mapsto \widetilde{\theta^{\circ}},
\end{equation}
where $\widetilde{\theta^{\circ}}$ is the inflation of $\theta^{\circ}$ to $\N_G(Q).$
Through this bijection, the set $\W(G)$ of weights of $G$ can be identified with the set 
\begin{equation}\label{equ:wights}
  \W^{\circ}(G)=\big\{(Q,\delta)\bwhich Q\in\Rad(G),\delta\in\IBr(\N_G(Q)\which |Q|)\big\}.
\end{equation}

The group $\hH \times \Aut(G)$ acts on the set $\W^{\circ}(G)$ by $(Q,\delta)^a := (Q^a, \delta^a)$ for all $(Q,\delta) \in \W^{\circ}(G)$ and $a \in \hH \times \Aut(G)$. This induces an action of $\hH \times \Aut(G)$ on the set $\W^{\circ}(G)/\sim_G$ of $G$-orbits of $\W^{\circ}(G)$.
The NAW conjecture in fact claims that there is an $\hH \times \Aut(G)$-equivariant bijection from $\IBr(G)$ to $\W^{\circ}(G)/\sim_G$.


\section{Modular character triples with defect zero characters}\label{sec_centralize}
In this section, we prove that any modular character triple with a defect zero character  is isomorphic to a triple with a \(p'\)-central subgroup which respects the action of Galois automorphisms, group automorphisms, and vertices (see Theorem \ref{thm:centralize}).

Let $G$ be a finite group. Now we want to define an $\hH\times\Aut(G)$-action on the set of isomorphism classes of $FG$-modules, but before this we introduce some more general notation. Let $A$ be an $F$-algebra and $V$ be a (right) $A$-module. Suppose that $a\in\Aut(A)$ is an automorphism of $A$ as a ring. We define $V^a$  to be the $A$-module such that
\begin{enumerate}
  \item As an abelian group $V^a$ is isomorphic to $V.$ We  write $V^a=\{v^a\which v\in V\},$ that is, the map $V\ra V^a,v\mapsto v^a$ is an isomorphism of abelian groups.
  \item For any $v\in V$ and $x\in A,$ let $v^a\cdot x^a=(vx)^a.$
\end{enumerate} 
Then $V^a$ is another $A$-module which may not be isomorphic to $V$ as $A$-modules. 
Now let $A=FG$ be a group algebra, and let $a=(\sigma,\phi)\in\hH\times \Aut(G)$. Then $a$ can be regarded as a ring automorphism of $FG$  that is defined as $a:FG\ra FG,\sum_{g\in G}k_gg\mapsto \sum_{g\in G}k_g^{\sigma}g^{\phi},$ where $k_g\in F.$ 
Thus we  define an $\hH\times\Aut(G)$-action on the set of isomorphism classes of $FG$-modules. 

\begin{lem}
Let $G$ be a finite group and $a=(\sigma,\phi)\in\hH\times\Aut(G).$ Suppose that $V$ is an  $FG$-module that affords the Brauer character $\varphi,$ then $V^a$ affords the Brauer character $\varphi^a.$ 
\end{lem}
\begin{proof}
Choose $\{v_1,\cdots,v_s\}$ to be an $F$-basis of $V.$ 
  Let $X:G\ra\GL_s(F)$ be the matrix representation of $V$ with respect to the basis $\{v_1,\cdots,v_s\},$ that is, $\forall g\in G,$ $X(g)$ is defined by 
  $$\begin{pmatrix}
   v_1 \\
  \vdots \\
   v_s                                 
  \end{pmatrix}g=
  \begin{pmatrix}
   v_1 g\\
  \vdots \\
   v_s g                                
  \end{pmatrix}=
  X(g)\begin{pmatrix}
   v_1 \\
  \vdots \\
   v_s                                
  \end{pmatrix}.$$
Note that $(kv)^a=k^{\sigma}v^a$ for any $k\in F$ and $v\in V.$  
Since every element $v\in V$ can be uniquely expressed in the form $v=k_1v_1+\cdots+k_sv_s,$ where $k_1,\cdots,k_s\in F,$ we have that $v^a=k_1^{\sigma}v_1^a+\cdots+k_s^{\sigma}v_s^a,$ and this expression is unique. Thus $\{v_1^{a},\cdots, v_s^a\}$ is an $F$-basis of $V^a.$ As
$$
\begin{pmatrix}
    v_1^a \\
    \vdots \\
    v_s^a 
\end{pmatrix}\cdot g=
\begin{pmatrix}
   v_1^a\cdot  g\\
  \vdots \\
   v_s^a \cdot g                                
  \end{pmatrix}=
  \begin{pmatrix}
   v_1^a\cdot  (g^{\phi^{-1}})^a\\
  \vdots \\
   v_s^a \cdot (g^{\phi^{-1}})^a                                
  \end{pmatrix}=
  \begin{pmatrix}
  ( v_1\cdot  g^{\phi^{-1}})^a\\
  \vdots \\
   (v_s\cdot g^{\phi^{-1}})^a                                
  \end{pmatrix}=
     X(g^{\phi^{-1}})^{\sigma}\begin{pmatrix}
   v_1^a \\
  \vdots \\
   v_s^a                                
  \end{pmatrix},
$$ we see that the matrix representation of the $FG$-module $V^a$ with respect to the basis $\{v_1^{a},\cdots, v_s^a\}$ is $Y:G\ra \GL_s(F),g\mapsto X(g^{\phi^{-1}})^{\sigma},$ which certainly affords the Brauer character $\varphi^a.$
\end{proof}

The following theorem is a strengthened version of \cite[Chapter V, Theorem 7.4]{Nagao}, incorporating Galois automorphisms and group automorphisms at its foundation.
Let $G$ be a finite group and $\varphi\in\IBr(G).$ Let $\mathbb{F}_p$ be the prime subfield of $F$, which is the smallest subfield of $F.$ We denote by $\mathbb{F}_p(\varphi)$ the subfield of $F$ that is an extension of $\mathbb{F}_p$ by adjoining the elements $\{\varphi(g)^{\ast} \which g\in G_{p'}\},$ where $\varphi(g)^{\ast}$ is the image of $\varphi(g)$ in $F$ under the natural surjection. For a subgroup $N$ of $G$, we denote by $\mathcal{S}(G,N)$ the set of subgroups of $G$ containing $N.$

\begin{thm}\label{thm:centralize}
  Let $Z$ be a normal subgroup of a finite group $G$. 
  Let $\hat\lambda\in\dz(Z)$, and let $\Inn(G)\leqslant \mathscr{A}\leqslant \hH\times\Aut(G)_Z.$ Assume that $\lambda:=\hat\lambda^{\circ}$ is $\mathscr{A}$-invariant. 
  Then there exist a finite group $G_1$, a $p'$-cyclic central subgroup $Z_1\leqslant G_1,$ a linear faithful character $\lambda_1\in\IBr(Z_1),$ and a group homomorphism $\omega:\mathscr{A}\ra \hH\times \Aut(G_1)_{Z_1}$ such that $\lambda_1$ is $\omega(\mathscr{A})$-invariant and  the following conditions hold:
  \begin{enumerate}
    \item There is a group isomorphism $\epsilon:G/Z\cong G_1/Z_1$.  Via this isomorphism, we identify $G/Z$ with $G_1/Z_1.$ Moreover, we have $|Z_1|= |\mathbb{F}_p(\lambda)^{\times}|.$
    \item \(
\omega(\operatorname{Inn}(G)) = \operatorname{Inn}(G_1).
\)
Moreover, for \( a = (\sigma, \phi) \in \mathscr{A} \) and \( \omega(a) = (\sigma', \phi') \), we have 
\begin{enumerate}
    \item \(\sigma = \sigma'\), and
    \item \(\phi = \phi'\) when they are viewed as automorphisms of \( G/Z = G_1/Z_1 \).
\end{enumerate}
    \item      For any $J\in\mathcal{S}(G,Z)$,  there is a bijection 
    $\nu_J:\IBr(J\which\lambda)\ra \IBr(J_1\which\lambda_1),$ where $J_1/Z_1=\epsilon(J/Z)$,  such that the following conditions hold:
    \begin{enumerate}
     \item For any $\varphi\in\IBr(J\which \lambda)$ and $a\in \mathscr{A} ,$ we have $\nu_J(\varphi)^{\omega(a)}=\nu_{J^a}(\varphi^a).$
     \item For any $H\in\mathcal{S}(J,N),$ we have $\nu_J(\varphi)_{H_1}=\nu_H(\varphi_H).$ Here, we need to span $\nu_H$ $\mathbb{N}$-linearly, where $\mathbb{N}$ denotes the set of non-negative integers, as $\varphi_H$ may be reducible.
     \item Let $\varphi\in\IBr(J\which\lambda)$ and $Y$ be a vertex of $\varphi.$ Then there exits a vertex $Y_t$ of $\nu_J(\varphi)$ such that $Y_tZ_1/Z_1=\epsilon(YZ/Z).$
    \end{enumerate}
  \end{enumerate}
\end{thm}
\begin{proof}
Let \( e_{\lambda} \) be the block idempotent associated with \( \lambda \), and set \( A = e_{\lambda}FG \) and \( A_1 = e_{\lambda}FZ \).  
Note that \( A \) is an \( F \)-algebra with identity element \( e_{\lambda} \). If \( x \) is an invertible element in \( A \), we denote by \( x' \) the inverse of \( x \) in \( A \); that is, \( x' \in A \) satisfies \( xx' = x'x = e_{\lambda} \). 
Note that \( A_1 \) is a subalgebra of \( A \) containing the identity.
 Let $E$ be the subfield $\mathbb{F}_p(\lambda)$ of $F.$ Then \( EG \) can be viewed as an \( E \)-subalgebra of \( FG \).  
By \cite[Chapter III, Theorem 2.22]{Nagao} and Equation (6.1) on page 231 of \cite{Nagao}, we have \( e_{\lambda} \in EZ \).  
Set \( B = e_{\lambda}EG \) and \( B_1 = e_{\lambda}EZ \).  
Then \( B \) and \( B_1 \) are \( E \)-subalgebras of \( A \) and \( A_1 \), respectively. Moreover,
$A = F \otimes_E B$ and $A_1 = F \otimes_E B_1.$

  Since $A_1$ is a full matrix algebra over the field $F,$ for any $g\in G,$ there is an element $a_g\in A_1^{\times}$ such that $a^g=a_g'aa_g$ for any $a\in A_1.$ 
  In fact, we can choose $a_g\in B_1^{\times},$ as $B_1$ is a full matrix algebra over the field $E$ and $A_1=F\otimes_E B_1.$
  Let $\bar{G}=G/Z.$ 
  For any $\bar{g}\in\bar{G},$ choose $g$ to be an inverse image of $\bar{g}$ in $G$ under the canonical surjection (let $g=1$ if $\bar{g}=1$), and choose $a_g$ be as above (let $g_1=e_{\lambda}$), and let $u_{\bar{g}}=a_g'g.$
  Note that $u_{\bar{g}}\in B^{\times}$ for any $\bar{g}\in\bar{G}.$
  By \cite[Chapter V, Theorem 7.2]{Nagao}, $\C_A(A_1)$ (resp. $\C_B(B_1)$)  is a twisted group algebra with an $F$-basis (resp. an $E$-basis) $\{u_{\bar{g}}\which \bar{g}\in\bar{G}\}$ and the multiplication induces an isomorphism $A_1\otimes \C_A(A_1)\ra A$ (resp. $B_1\otimes_E \C_B(B_1)\ra B$).
  The tensor product is taken over the field $F$ if there is no script. We identify $A$ with $A_1\otimes \C_A(A_1)$ through this isomorphism. Let $\alpha\in\ZZ^2(\bar{G},F^{\times})$ be the factor set such that $u_{\bar{g}}u_{\bar{h}}=\alpha(\bar{g,}\bar{h})
  u_{\bar{g}\bar{h}}.$ Then we have  $\alpha(\bar{g},\bar{h})\in E^{\times}$ for any $\bar{g},\bar{h}\in\bar{G}.$
  Since $e_{\lambda}$ is $ \mathscr{A} $-invariant, we have that $ \mathscr{A} $ acts on $A$ and $A_1,$ and hence on $\C_A(A_1).$
  Let $G_1=\{ku_{\bar{g}}\which k\in E^{\times}, \bar{g}\in\bar{G}\}.$ 
  We want to prove that $G_1$ is an $ \mathscr{A} $-invariant finite subgroup of $\C_A(A_1)^{\times}$ and has order $|E^{\times}||\bar{G}|$. We only need to show  that $G_1$ is $ \mathscr{A} $-invariant. Notice that $\C_B(B_1)$ is $ \mathscr{A} $-invariant and $G_1$ is exactly the invertible elements in $\C_B(B_1)$ that has the form $ku_{\bar{g}}$ for some $k\in E^{\times}$ and $\bar{g}\in\bar{G},$ we are done.
 Let $Z_1=\{ku_1\which k\in E^{\times}\},$ then $Z_1$ is a $p'$-cyclic central subgroup of $G_1$ and has order $|E^{\times}|.$ 
 We  identify $G_1/Z_1$ with $\bar{G}$ naturally, and  denote this isomorphism by $\epsilon:G/Z\ra G_1/Z_1.$
 Define a faithful representation $\lambda_1:Z_1\ra F^{\times},ku_1\mapsto k$ and let  $e_{\lambda_1}$ be the block idempotent of $FZ_1$ associated with $\lambda_1.$ 
 Note that $$e_{\lambda_1}=\frac{1}{|Z_1|}\sum_{z\in Z_1}\lambda_1(z)^{-1}z.$$ It is not difficult to prove that there is an isomorphism $e_{\lambda_1}FG_1\ra \C_A(A_1)$ of $F$-algebras induced by the inclusion map $e_{\lambda_1}g\mapsto g,$ where $g\in G_1.$
  We identify $\C_A(A_1)$ with $e_{\lambda_1}FG_1$ via this isomorphism.
  
  Let $W$ be the unique simple $A_1$-module. By \cite[Chapter V, Lemma 7.3]{Nagao}, the map $V\mapsto W\otimes V$ defines a bijection from the set of isomorphism classes of irreducible $e_{\lambda_1}FG_1$-modules to the set of isomorphism classes of irreducible $A$-modules. 
  Note that $\mathscr{ A} $ acts on the group $G_1$ as group automorphisms. To avoid confusion, for $a\in \mathscr{A} ,$ we  write $\widetilde{a}\in\Aut(G_1)$ for this group automorphism. 
  Fix an $a=(\sigma,\phi)\in \mathscr{A}.$ 
  As $Z_1=\{ku_1\which k\in E^{\times}\}$ and $(ku_1)^{\widetilde{a}}=k^{\sigma}u_1,$  it is routinely to check that $e_{\lambda_1}^{(\sigma, \widetilde{a})}=e_{\lambda_1}.$ 
  For a simple $e_{\lambda_1}FG_1$-module $V$, we need to prove that $W\otimes V^{(\sigma,\widetilde{a})}$ is isomorphic to $(W\otimes V)^a$ as $A$-modules. 
  Note that $e_{\lambda_1}FG_1$ is naturally isomorphic to $\C_A(A_1)$ and the $(\sigma,\widetilde{a})$-action on $e_{\lambda_1}FG_1$ corresponds to the $a$-action on $\C_A(A_1)$ under this  isomorphism. Consequently, we may regard $V$ as a $\C_A(A_1)$-module, and $V^{(\sigma,\widetilde{a})}$ is precisely the $\C_A(A_1)$-module $V^a.$ 
  We are going to construct an isomorphism $\widetilde{f}:W\otimes V^a\ra (W\otimes V)^a$ of $A$-modules. 
  First,  since there is a unique simple \( A_1 \)-module up to isomorphisms, there exists an isomorphism of abelian groups \( f: W \to W \) such that the map \( W \to W^a,\ w \mapsto f(w)^a \) defines an isomorphism of \( A_1 \)-modules. One can verify that \( f(w \cdot x) = f(w) \cdot {}^a x \) for all \( w \in W \) and \( x \in A_1 \), where \( {}^a x = x^{a^{-1}} \).
Second, we define
\[
\widetilde{f}: W \otimes V^a \longrightarrow (W \otimes V)^a,\quad w \otimes v^a \longmapsto (f(w) \otimes v)^a,
\]
with \( w \in W \) and \( v \in V \). We now prove that \( \widetilde{f} \) is an isomorphism of \( A \)-modules. 
By the universal property of the tensor product, one can show that $\widetilde{f}$ is a well-defined isomorphism of abelian groups.
 For any $x\in A_1$ and $y\in\C_A(A_1),$ we have 
  \begin{align*}
    \widetilde{f}(w\otimes v^a)(x\otimes y)&=(f(w)\otimes v)^a\cdot(x\otimes y)  \\
     &=(f(w)\cdot \zs{a}x\otimes v\cdot\zs{a}y)^a \\
     &=(f(wx)\otimes v\cdot\zs{a}y)^a \\
     &=\widetilde{f}(wx\otimes v^{a}\cdot y) \\
     &=\widetilde{f}((w\otimes v^{a})(x\otimes y)),
  \end{align*}
  thus proving that $\widetilde{f}$ is an isomorphism of $A$-modules. We define $\omega: \mathscr{A} \ra \hH\times \Aut(G_1),a=(\sigma,\phi)\mapsto (\sigma,\widetilde{a})$ and $\nu_G:\IBr(G\which\lambda)\ra\IBr(G_1\which\lambda_1)$ such that if the simple $e_{\lambda_1}FG_1$-module $V$ affords the irreducible Brauer character $\nu_G(\varphi),$ then the simple $e_{\lambda}FG$-module $W\otimes V$ affords the irreducible Brauer character $\varphi.$ We have showed  that $\nu_G(\varphi)^{\omega(a)}=\nu_G(\varphi^a).$ 
  
  Let $Y_t$ be a vertex of $V$. We aim to show that there exists a vertex $Y$ of the $G$-module $W\otimes V$ such that
\(
YZ/Z = Y_t Z_1 / Z_1.
\)
Let $\E_k(V)$ denote the endomorphism ring of $V$. Then $\E_k(V)$ can be viewed as the simple quotient of $e_{\lambda_1} F G_1$ corresponding to $V$. The group homomorphism $G_1 \to \E_k(V)^{\times}$ induced by $V$ gives rise to the canonical surjection
\[
p_V : e_{\lambda_1} F G_1 \ra \E_k(V).
\]
Since $e_{\lambda} F G = A_1 \otimes e_{\lambda_1} F G_1$, it follows that $A_1 \otimes \E_k(V)$ is a simple quotient of $e_{\lambda} F G$, which naturally corresponds to the simple $G$-module $W \otimes V$.
As $A_1$ is the endomorphism ring of $W$, we may regard $A_1 \otimes \E_k(V)$ as the endomorphism ring of $W \otimes V$. The corresponding group homomorphism $G \to (A_1 \otimes \E_k(V))^{\times}$ then induces precisely the algebra surjection
\[
\operatorname{id}_{A_1} \otimes p_V : e_{\lambda} F G = A_1 \otimes e_{\lambda_1} F G_1 \ra A_1 \otimes \E_k(V).
\]

Let $g \in G$ and $a_g \in A_1^\times$ be such that $a^g = a_g' a a_g$ for all $a \in A_1$.  
Then $a_g' g = k u_{\bar g}$ for some $k \in E^\times$, and hence can be identified with an element $\tilde{g} \in G_1$ satisfying $\epsilon(gZ) = \tilde{g} Z_1$.  
Consequently, under the canonical isomorphism $e_\lambda FG = A_1 \otimes e_{\lambda_1} F G_1$, we have $e_\lambda g = a_g \otimes e_{\lambda_1} \tilde{g}$.
Since $e_\lambda h = e_\lambda h \otimes 1$ for any $h \in Z$, it follows from \cite[Chapter II, Lemma 4.1(i)]{Nagao} that $(A_1 \otimes \E_k(V))^Z = 1 \otimes \E_k(V)$.  
Moreover, $\E_k(V) = \E_k(V)^{Z_1}$ because $Z_1$ is a central subgroup of $G_1$.  
Thus we obtain an isomorphism
\[
\mathfrak{f}_Z : \E_k(V)^{Z_1} \ra (A_1 \otimes \E_k(V))^{Z}, \quad x \mapsto e_\lambda \otimes x,
\]
which, by direct computation, restricts to an isomorphism
\(
\mathfrak{f}_J : \E_k(V)^{J_1} \ra (A_1 \otimes \E_k(V))^{J}
\)
for any $J \in \mathcal{S}(G,Z)$.
One can also verify that the family of isomorphisms $\{ \mathfrak{f}_J \which J \in \mathcal{S}(G,Z) \}$ commutes with the trace maps.

By \cite[Chapter IV, Theorem 2.2]{Nagao}, the vertex $Y_t$ of $V$ is a minimal subgroup of $G_1$ such that $1\in\E_k(V)_{Y_t}^{G_1}$.  
Then clearly $1\in\E_k(V)_{Z_1Y_t}^{G_1}$.  
Via the isomorphisms $\{\mathfrak{f}_J\}$, we obtain $1\in(A_1\otimes\E_k(V))_{Y'}^G$, where $Y'\in\mathcal{S}(G,Z)$ satisfies $Y'/Z=Y_tZ_1/Z_1$.  
Choose a vertex $Y$ of $W\otimes V$ contained in $Y'$; we claim that $Y' = ZY$.  
Otherwise, if $ZY\subsetneq Y'$, then $1\in(A_1\otimes\E_k(V))_{ZY}^G$.  
Applying the isomorphisms $\{\mathfrak{f}_J\}$ again, we get $1\in\E_k(V)_{(ZY)_1}^{G_1}$.  
This contradicts the minimality of $Y_t$, since $(ZY)_1\subsetneq Y_tZ$.  
Thus we have proved  $YZ/Z = Y_tZ_1/Z_1$.

Conditions (1) and (2) follow directly from the definition. For any $J\in\mathcal{S}(G,N)$, observing that the fixed isomorphism $e_{\lambda}FG\cong A_1\otimes e_{\lambda_1}FG_1$ restricts to an isomorphism $e_{\lambda}FJ\cong A_1\otimes e_{\lambda_1}FJ_1$,  we can define $\nu_J$ analogously to $\nu_G$. 
Condition (3,a) can be proved in the same manner as we did for $\nu_G$ in the preceding paragraphs, while condition (3,b) is evident from the definition. Condition (3,c) has already been discussed for the case $J=G$ in the previous paragraphs; the other cases follow by the same argument.
\end{proof}

\begin{remark}
   Notice that in the above theorem, we have controlled the order of $G_1,$ that is, the equation $|G_1|=|G/Z||\mathbb{F}_p(\lambda)^{\times}|.$ The key point is that $|G_1|$ will not grow as the field $F$ becomes larger. This is because  by our assumption, we have assumed that the field $F$ is big enough for every finite group we consider, which means that it is big enough for $G_1.$
\end{remark}

\section{The Dade--Glauberman--Nagao correspondence}\label{sec_DGN-cor}

Now we  introduce an important theorem that will be used in the reduction. Suppose that $N\zg M$ are finite groups and $M/N$ is a $p$-group. Let $\theta\in\dz(N)$ be $M$-invariant.
 Let $B$ be the unique block of $M$ covering the block $\bl(\theta)$ of $N$ to which $\theta$ belongs. Let $D$ be a defect group of $B.$ Then it follows that $M=ND$ and $N\cap D=1.$ Notice then that $\N_M(D)=D\times \C_N(D).$ 
 The Brauer first main correspondent $B^{\star_D}\in\Bl(\N_M(D))$ of $B$ covers a unique block $b^{\star_D}$ of $\C_N(D).$ 
 The block $b^{\star_D}$ has defect zero, and the unique irreducible complex character $\theta^{\star_D}$ contained in $b^{\star_D}$ is called the \textbf{Dade–Glauberman–Nagao} (\textbf{DGN}) \textbf{correspondent} of $\theta$ with respect to $D$ (see \cite[Chapter V, Theorem 12.1]{Nagao}). Evidently, $\theta^{\star_D} \in \dz(\C_N(D))$.
If $a \in \hH \times \Aut(M)_N$, then $(\theta^{\star_D})^a = (\theta^a)^{\star_{(D^a)}}$.

 Let $N \leqslant G$ be finite groups and $\theta \in \IBr(N)$. Denote by $\theta^{\hH}$ the $\hH$-orbit of $\theta$, and by $\Aut(G)_{N,\theta^{\hH}}$ the stabilizer of $\theta^{\hH}$ in $\Aut(G)_N$. Since $\theta^{\hH}$ consists of characters of $N$, we simply write $\Aut(G)_{\theta^{\hH}}$ for $\Aut(G)_{N,\theta^{\hH}}$. 
We similarly define $\Aut(G)_{\theta}$ and $(\hH \times \Aut(G))_{\theta}$ as the stabilizers of $\theta$ in $\Aut(G)_N$ and $\hH \times \Aut(G)_N$, respectively. Note that $\Aut(G)_{\theta^{\hH}}$ is the image of $(\hH \times \Aut(G))_{\theta}$ under the canonical projection $\hH \times \Aut(G) \to \Aut(G)$. 
 Let $\IBr(G \which \theta^{\hH})$ be the set of irreducible Brauer characters of $G$ whose restriction to $H$ contains an irreducible constituent lying in $\theta^{\hH}$.

\begin{thm}\label{Thm:DGN}
  Suppose that $N\zg G$ are finite groups, and that $M/N$ is a normal $p$-subgroup of $G/N.$ Let $\theta\in\dz(N)$ be $G$-invariant and $D$ be a defect group of the unique block $B$ of $M$ covering $\bl(\theta).$ Let $\theta^{\star_D}$ be the \DGN  correspondent of $\theta$ with respect to $D.$ 
  Then there exists an $\big(\hH\times\Aut(G)\big)_{\theta,D}$-equivariant bijection $$f:\IBr(G\which\theta^{\circ},|D|) \ra\IBr(\N_G(D)\which (\theta^{\star_D})^{\circ},|D|).$$ 
\end{thm}
\begin{proof}
  Let $T=\Aut(G)_{\theta^{\hH},D}$ and $\widetilde{G}= G\semi T$ be the semi-product of $G$ and $T.$ Note that $M$ is normal in $\widetilde{G}$ and $\N_{\widetilde{G}}(D)=\N_G(D)\semi T.$
  Write $\eta=\theta^{\circ}$ and $\eta'=(\theta^{\star_D})^{\circ}.$ 
  By \cite[Theorem A]{Fu}, there exists an $\hH\times \N_{\widetilde{G}}(D)$-equivariant bijection $\varDelta_G:\IBr(G\which \eta^{\hH})\ra\IBr(\N_G(D)\which \eta'^{\hH}).$ Since $\big(\hH\times\Aut(G)\big)_{\theta,D}$ can be regarded as a subgroup of $\hH\times \N_{\widetilde{G}}(D),$ we have that the map $\varDelta_G$ is $\big(\hH\times\Aut(G)\big)_{\theta,D}$-equivariant. By conditions (2) and (3) in \cite[Theorem A]{Fu}, we have that the bijection $\varDelta_G$  restricts to an $\big(\hH\times\Aut(G)\big)_{\theta,D}$-equivariant bijection $f_1:\IBr(G\which\eta) \ra\IBr(\N_G(D)\which\eta').$ 
  For any $\varphi\in\IBr(G\which\eta),$ notice that any vertex of $\varphi$ intersects trivially with $N$  since the block containing $\theta$ has defect zero. The same situation happens with $\IBr(\N_G(D)\which\eta')$ and $\C_N(D).$ By (4) of \cite[Theorem A]{Fu}, we have that $m(\varphi)=m(f_1(\varphi))$ for any $\varphi\in\IBr(G\which\eta),$ where $m(\varphi)$ means the order of any vertex of $\varphi.$ Thus $f_1$  restricts to an $\big(\hH\times\Aut(G)\big)_{\theta,D}$-equivariant bijection $f:\IBr(G\which\eta,|D|) \ra\IBr(\N_G(D)\which\eta',|D|).$ 
\end{proof}


\section{Modular $\hH$-triples}\label{sec_Character_pairs}
The isomorphism theory of character triples (see \cite{Sp17,Sp18} or \cite{Nav-Mckaybook}) and modular character triples (see \cite{SV16}) plays an important role in the reduction of the McKay conjecture  and the  Alperin weight conjecture. It facilitates both the formulation of the inductive conditions and the streamlining of the proofs of these reduction theorems (see, for example, \cite{Sp18} for the reduction of the McKay conjecture).
If Galois automorphisms are considered, then $\hH$-triples (see \cite{Nav-GMckay}) are needed in place of character triples. 
The theory of $\hH$-triples introduced in \cite{Nav-GMckay} is for ordinary characters. In this paper, we will adapt it to Brauer characters.

Let $N\zg G$ be finite groups, and $\eta\in\IBr(N).$ 
Denote by $\eta^{\hH}$ the $\hH$-orbit of $\eta,$ and by $G_{\eta^{\hH}}$ the stabilizer of $\eta^{\hH}$ in $G.$ 
If $G=G_{\eta^{\hH}},$  then we write $(G,N,\eta)_{\hH}$ and call it a \textbf{(modular) $\hH$-triple}.

Let $(G,N,\eta)_{\hH}$ be an $\hH$-triple, and $G_{\eta}$ be the stabilizer of $\eta$ in $G.$ Then we have $G_{\eta}\zg G.$ Let $\Pj$ be projective representation of $G_{\eta}$ associated with $\eta.$ Let $a=(\sigma, g)\in(\hH\times G)_{\eta},$ the stabilizer of $\eta$ in $\hH\times G.$ Then the map $\Pj^a:G_{\eta}\ra\GL_{\eta(1)}(F),h\mapsto \Pj(ghg^{-1})^{\sigma}$ is also a projective representation of $G_{\eta}$ associated with $\eta.$ Thus as described in \cite[Remark 1.3]{Nav-GMckay}, there is a map $\mu_a:G_{\eta}\ra F^{\times}$ uniquely determined by $\Pj$ and $a$ such that $\Pj^a\sim \mu_a\Pj,$ that is, there exists an invertible matrix $T$ satisfying that $\Pj^a(h)=\mu_a(h)T \Pj(h) T^{-1}$ for any $h\in G_{\eta}$. In this case $\mu_a(1)=1$ and the values of $\mu_a$ is constant on any $N$-coset, thus $\mu_a$ can be regarded as a map from $G_{\eta}/N$ to $F^{\times}.$

We now define a partial order relation between $\hH$-triples, analogous to \cite[Definition 1.5]{Nav-GMckay}.

\begin{defi}\label{def:H-trip}
  Suppose that $(G,N,\eta)_{\hH}$ and $(H,M,\eta')_{\hH}$ are $\hH$-triples. Then we write $$(G,N,\eta)_{\hH}\geqslant_c(H,M,\eta')_{\hH}$$ if the following conditions hold:
  \begin{enumerate}
    \item $G=NH,M=N\cap H,\C_{G}(N)\subseteq H.$
    \item $(\hH\times H)_{\eta}=(\hH\times H)_{\eta'}.$ In particular, $H_{\eta}=H_{\eta'}.$
    \item There exist projective  representations $\Pj$ of $G_{\eta}$ and $\Pj'$ of $H_{\eta'}$ associated with $\eta$ and $\eta'$ with factor sets $\alpha$ and $\alpha'$, respectively, such that $\alpha'=\alpha_{H_{\eta}\times H_{\eta}}$ and for any $c\in\C_{G}(N)$, the scalar matrices $\Pj(c)$ and $\Pj'(c)$ are associated with the same scalar.
    \item For any $a\in(\hH\times H)_{\eta},$ the functions $\mu_a$ and $\mu_a'$ determined by $\Pj^a\sim\mu_a\Pj$ and $\Pj'^a=\mu_a'\Pj'$, respectively, agree on $H_{\eta}$. 
  \end{enumerate}
\end{defi}

In the situation described above we say that $(\Pj,\Pj')$ \textbf{gives} $(G,N,\eta)_{\hH}\geqslant_c(H,M,\eta')_{\hH}.$
If $\Pj$ and $\Pj'$ are group representations, then the maps $\mu_a,\mu_a'$ are also group representations. In this case we say that $(\widetilde{\eta},\widetilde{\eta}')$ gives the above ordering of $\hH$-triples, where $\tilde{\eta}$ and $\tilde{\eta}'$ are the Brauer characters corresponding to $\Pj$ and $\Pj',$ respectively.
It follows directly that if $(G,N,\eta)_{\hH} \geqslant_c (H,M,\eta')_{\hH}$, then for any subgroup $N \subseteq J \leqslant G$, we have
$
(J, N, \eta)_{\hH} \geqslant_c (J \cap H, M, \eta')_{\hH}.
$

Let $N$ be a normal subgroup of a finite group $G$, and $H,M$ be subgroups of $G$ such that $G=NH, N\cap H=M.$ 
Let $\eta\in\IBr(N), \eta'\in\IBr(M),$ and assume that $\eta$ is $G$-invariant and $\eta'$ is $H$-invariant.
Suppose that there exist  projective representations $\Pj$ of $G$ and $\Pj'$ of $H$ associated with $\eta$ and $\eta'$ with factor sets $\alpha$ and $\alpha'$, respectively, such that $\alpha'=\alpha_{H\times H}.$ 
Then there is a bijection $$':\IBr(G\which\eta)\ra\IBr(H\which\eta'),$$
such that if $\PQ\otimes\Pj$ affords $\varphi\in\IBr(G\which\eta),$ then $\PQ\otimes\Pj'$ affords $\varphi'$, where $\PQ$ is any projective representation of $G/N\cong H/M$ with the factor set $\alpha^{-1}$ (see \cite[Theorem 10.13]{Nav-Mckaybook}).
We say that this bijection is induced by $(\Pj,\Pj').$

Keep the notation above. Let $a=(\sigma,\phi)\in\big(\hH\times\Aut(G)\big)_{N,H,\eta,\eta'}.$ 
Define the maps $$\Pj^{a}:G\ra\GL_{\eta(1)}(F)\quad\text{and}\quad \Pj'^a:H\ra \GL_{\eta'(1)}(F) $$ by $$\Pj^a(g)=\Pj(g^{\phi^{-1}})^{\sigma},\ \forall g\in G\quad \text{and}\quad \Pj'^a(h)=\Pj'(h^{\phi^{-1}})^{\sigma},\ \forall h\in H.$$  Then $\Pj^a$ and $\Pj'^a$ are projective representations of $G$ and $H$ associated with $\eta$ and $\eta'$, respectively. 
Let $\mu_a,\mu'_a:G/N\cong H/M\ra F^{\times}$ be the maps determined by $\Pj^a\sim\mu_a\Pj$ and $\Pj'^a\sim\mu'_a\Pj',$ respectively. 
The following lemma has been  proved in \cite[Lemma 2.2]{Fu}.

\begin{lem}\label{lem:H-tri1}
  With the previous notation. If $\mu_a=\mu'_a$ through the isomorphism $G/N\cong H/M.$  Then $(\varphi^a)'=(\varphi')^a$ for every $\varphi\in\IBr(G\which\eta).$
\end{lem}

\begin{lem}\label{lem:H-tri2}
  Following the notation above. Further assume that $C:=\C_G(N)\subseteq H$ and that there exists a map $\zeta:C\ra F^{\times}$ such that $\Pj(c)=\zeta(c)I_{\eta(1)}, \Pj'(c)=\zeta(c)I_{\eta'(1)}$ for any $c\in C.$ Then for any $\varphi\in\IBr(G\which\eta),$ we have $\varphi_C=\eta(1)/\eta'(1)\cdot \varphi'_C.$
\end{lem}
\begin{proof}
  This can  be seen from the definition, or  \cite[Lemma 10.16]{Nav-Mckaybook}.
\end{proof}

Ordered pairs of $\hH$-triples yield an equivariant bijection between related character sets.

\begin{thm}\label{thm:H-triB}
  Let $(G,N,\eta)_{\hH}$ and $(H,M,\eta')_{\hH}$ be $\hH$-triples such that $(G,N,\eta)_{\hH}\geqslant_c(H,M,\eta')_{\hH}$. Then for any $N\leqslant J\leqslant G$, there exists an $\big(\hH\times\N_H(J)\big)_{\eta}$-equivariant bijection 
  \begin{equation}\label{equ:add1}
    \varDelta_J:\IBr(J\which\eta)\ra\IBr (J\cap H\which \eta')
  \end{equation} such that  $\varphi_{\C_J(N)}=\eta(1)/\eta'(1)\cdot \varDelta_J(\varphi)_{\C_J(N)}$ for any $\varphi\in\IBr(J\which\eta).$
\end{thm}
\begin{proof}
Let $(G,N,\eta)_{\hH} \geqslant_c (H,M,\eta')_{\hH}$ be given by $(\Pj,\Pj')$ and let the bijection (\ref{equ:add1}) be induced by $(\Pj_J,\Pj'_{J\cap H})$.  
The conditions in the theorem then follow from Lemmas \ref{lem:H-tri1} and \ref{lem:H-tri2}.
\end{proof}
 
The following theorem tells that the ordering of $\hH$-triples only depends on the automorphisms of the normal subgroup induced via conjugation by the overgroup; see \cite[Theorem 2.9]{Nav-GMckay}.

\begin{thm}\label{thm:H-triA}
 Let $(G,N,\eta)_{\hH}$ and $(H,M,\eta')_{\hH}$ be $\hH$-triples such that $(G,N,\eta)_{\hH}\geqslant_c(H,M,\eta')_{\hH}.$ Let $(\widehat{G},N,\eta)_{\hH}$ and $(\widehat{H},M,\eta')_{\hH}$ be $\hH$-triples such that $\widehat{G}=N\widehat{H}, M=N\cap\widehat{H},\C_{\widehat{G}}(N)\subseteq \widehat{H}$ and $(\hH\times\widehat{H})_{\eta}=(\hH\times \widehat{H})_{\eta'}.$ Let $\epsilon:H\ra\Aut(N),\hat{\epsilon}:\widehat{H}\ra\Aut(N)$ be the group homomorphisms induced by conjugation. 
 If $\hat{\epsilon}(\widehat{H})\subseteq \epsilon(H)$ then $$(\widehat{G},N,\eta)_{\hH}\geqslant_c(\widehat{H},M,\eta')_{\hH}.$$
\end{thm} 
\begin{proof}
Let $H_1 \subseteq H$ be the inverse image of $\hat{\epsilon}(\widehat{H})$ under the group homomorphism $\epsilon$. Then, by the properties of the ordering relation of $\hH$-triples, we have $(NH_1, N, \eta)_{\hH} \geqslant_c (H_1, M, \eta')_{\hH}$. Thus, by replacing $G$ with $NH_1$, we may assume that $\hat{\epsilon}(\widehat{H}) = \epsilon(H)$.

Assume that $\hat{\epsilon}(\widehat{H}) = \epsilon(H)$.
Suppose that $(\Pj,\Pj')$ gives $(G,N,\eta)_{\hH}\geqslant_c(H,M,\eta')_{\hH}.$ Let $\alpha$ and $\alpha'$ be the factor sets associated with $\Pj$ and $\Pj',$ respectively. 
Let $\mathcal{T}$ be a complete set of representatives of $M\C_{G}(N)$-cosets in $H_{\eta}$ with $1\in\mathcal{T}.$ For any $t\in \mathcal{T},$ choose $\hat{t}\in\widehat{H}_{\eta}$ be such that $\epsilon(t)=\hat{\epsilon}(\hat{t})$ and with $\hat{1}=1.$ Notice that the set $\widehat{\mathcal{T}}:=\{\hat{t}\which t\in\mathcal{T}\}$ is a complete set of representatives of $M\C_{\widehat{G}}(N)$-cosets in $\widehat{H}_{\eta}.$
Let $\mu_0:\ZZ(N)\ra F^{\times}$ be the linear representation such that $\Pj(z)=\mu_0(z)I_{\eta(1)}$ and $\Pj'(z)=\mu_0(z)I_{\eta'(1)}$ for $z\in\ZZ(N).$ Let $\hat{\mu}:\C_{\widehat{G}}(N)\ra F^{\times}$ be any map such that $\hat{\mu}(cz)=\hat{\mu}(c)\mu_0(z)$ for $c\in \C_{\widehat{G}}(N)$ and $z\in \ZZ(N).$

We define the maps
$$\widehat{\Pj}:\widehat{G}_{\eta}\ra\GL_{\eta(1)}(F)\quad \text{and} \quad \widehat{\Pj}':\widehat{H}_{\eta}\ra\GL_{\eta'(1)}(F)$$ by
$$\widehat{\Pj}(\hat{t}n\hat{c})=\Pj(t)\Pj(n)\hat{\mu}(\hat{c})\quad \text{and}\quad \widehat{\Pj}'(\hat{t}m\hat{c})=\Pj'(t)\Pj'(m)\hat{\mu}(\hat{c}),$$ where $t\in\mathcal{T},n\in N,\hat{c}\in\C_{\widehat{G}}(N)$ and $m\in M.$ 
It can be seen from the definition that for $\hat{c}\in\C_{\widehat{G}}(N),$ the scalar matrices $\widehat{\Pj}(\hat{c})$ and $\widehat{\Pj}'(\hat{c})$ are associated with the same scalar $\hat{\mu}(\hat{c}).$ 
Following the proof of \cite[Theorem 10.18]{Nav-Mckaybook}, it can be shown that $\widehat{\Pj}$ and $\widehat{\Pj}'$ are projective representations associated with $\eta$ and $\eta',$ respectively, and the factor sets $\hat{\alpha}$ associated with $\widehat{\Pj}$ and $\hat{\alpha}'$ associated with $\widehat{\Pj}',$ agree on $\widehat{H}_{\eta}.$ 

Let $\hat{a}\in(\hH\times\widehat{H})_{\eta}.$ Let $\hat{\mu}_{\hat{a}}:\widehat{G}_{\eta}\ra F^{\times}$ and $\hat{\mu}'_{\hat{a}}:\widehat{H}_{\eta}\ra F^{\times}$ be the maps determined by $\widehat{\Pj}^{\hat{a}}\sim \hat{\mu}_{\hat{a}}\widehat{\Pj}$ and $\widehat{\Pj}'^{\hat{a}}\sim \hat{\mu}'_{\hat{a}}\widehat{\Pj}',$ respectively. It remains to prove that $\hat{\mu}_{\hat{a}}$ and $\hat{\mu}'_{\hat{a}}$ agree on $\widehat{H}_{\eta}.$ This has been done in the proof of \cite[Theorem 2.9]{Nav-GMckay} (see Remark \ref{Rem:H-trip,cor}). 
\end{proof}

\begin{remark}\label{Rem:H-trip,cor}
  There are some misprints  in the proof of \cite[Theorem 2.9]{Nav-GMckay}. Keep the notation there, in fact $$\hat{\mu}_{\hat{h}^{-1}\sigma}(\hat{t}m\hat{c})=\mu_{h^{-1} \sigma}(tm)\lambda(c_{h}^{-1})^{\sigma}\hat{\lambda} (\hat{c}_{\hat{h}}(\hat{c})^{\hat{h}})^{\sigma}\hat{\lambda} (\hat{c})^{-1}(\alpha((tm)^{h},c^{-1}_h)^{-1})^{\sigma}, $$ where $\alpha$ is the factor set associated with $\Pj.$ However, this will not affect the correctness of that Theorem.
\end{remark}

\section{Constructing new ordering $\hH$-triples from old ones}\label{sec_new-H-tri}
In this section, we are going to introduce some ways of constructing new ordering relations between $\hH$-triples from given ones. There are  three main cases, all of which can be found in \cite[Section 2]{Nav-GMckay}.  After this, we prove a lifting and descent lemma of ordering relation of $\hH$-triples (see Lemma \ref{lem:H-triUD}).

\begin{lem}\label{lem:H-triA}
Let $(G,N,\eta)_{\hH} \geqslant_c (H,M,\eta')_{\hH}$ be $\hH$-triples and assume that $G$ is a subgroup of a group $A$. Then for any $a = (\sigma, g) \in \hH \times A$, we have $(G^a,N^a,\eta^a)_{\hH}\geqslant_c(H^a,M^a,\eta'^a)_{\hH}.$
\end{lem}
\begin{proof}
By \cite[Lemma 2.3]{Nav-GMckay}, we have $(G,N,\eta^{\sigma})_{\hH}\geqslant_c (H,M,\eta'^{\sigma})_{\hH}.$ Thus we may assume that $a=(1,g)$ with $g\in A.$ 
The lemma then follows by Applying \cite[Lemma 2.1]{Nav-GMckay} to the group isomorphism $f:G\ra G^g,x\mapsto x^g.$ 
\end{proof}

Now, we introduce some notation of \textbf{wreath products}, following \cite[10.1]{Nav-Mckaybook}. If $N$ is a finite group and $m\geqslant 1$ is an integer, we denote by $N^m$ the group $N\times\cdots\times N$ ($m$ times). 
Let $\bS_m$ be the symmetric group of degree $m,$ where the maps are composed from the left. 
For any $(n_1,\cdots,n_m)\in N^m$ and $a_1\cdots a_mf\in\Aut(N)\wr \bS_m:=\Aut(N)^m\rtimes \bS_m,$ where $a_i\in\Aut(N)$ and $f\in \bS_m,$ let
$$(n_1,\cdots,n_m)^{a_1\cdots a_mf}=\Big(\big(n_{f^{-1}(1)}\big)^{a_{f^{-1}(1)}},\cdots,\big(n_{f^{-1}(m)}\big)^{a_{f^{-1}(m)}}\Big).$$
Thus we can view $\Aut(N)\wr \bS_m$ as a subgroup of $\Aut(N^m)$ (if $N$ is nontrivial).
It can be calculated that
$$(\eta_1\times\cdots\times\eta_m)^{a_1\cdots a_mf}= \big(\eta_{f^{-1}(1)}\big)^{a_{f^{-1}(1)}}\times\cdots\times\big(\eta_{f^{-1}(m)}\big)^{a_{f^{-1}(m)}} ,$$ where $\eta_1\times\cdots\times\eta_m$ is a Brauer character of $N^m.$
The Galois automorphisms action on the characters of $N^m$ is much simpler. For any $\sigma\in\hH,$ it is easy to check that
\begin{align*}
  (\eta_1\times\cdots\times\eta_m)^{\sigma}=\eta_1^{\sigma}\times\cdots\times\eta_m^{\sigma}.
\end{align*}

The following lemma is an application of \cite[Theorem 2.7]{Nav-GMckay} to modular $\hH$-triples. The proof follows the same routine, except that all representations here are taken over the field $F$.

\begin{lem}\label{lem:H-triB}
Let $(G,N,\eta)_{\hH}\geqslant_c(H,M,\eta')_{\hH}$ be $\hH$-triples. 
Let $k,l\geqslant 1$ be two integers, and $\sigma_1,\cdots,\sigma_k\in\hH$. Let $\widetilde{\eta}=(\eta^{\sigma_1})^{l}\times\cdots\times
(\eta^{\sigma_k})^{l}\in\IBr(N^m)$ and $\widetilde{\eta}'=(\eta'^{\sigma_1})^{l}\times\cdots
\times(\eta'^{\sigma_k})^{l}\in\IBr(M^m),$ where $m=kl.$
  Assume that
     $\eta^{\sigma_i}$ and $\eta^{\sigma_j}$ are not $G$-conjugate whenever $i\neq j.$ 
Then we have \[((G\wr \bS_m)_{\widetilde{\eta}^{\hH}},N^m,\widetilde{\eta})_{\hH} \geqslant_c((H\wr   \bS_m)_{\widetilde{\eta}'^{\hH}},M^m,\widetilde{\eta}')_{\hH}.\]
\end{lem}

\begin{lem}\label{lem:H-triC}

For any $i=1,\cdots,m,$ let $(G_i,N_i,\eta_i)_{\hH}$ and $(H_i,M_i,\eta_i')_{\hH}$ be $\hH$-triples such that $(G_i,N_i,\eta_i)_{\hH}\geqslant_c(H_i,M_i,\eta'_i)_{\hH}.$
  Let $G=G_1\times\cdots\times G_m, H= H_1\times\cdots\times H_m, N=N_1\times\cdots\times N_m, M=M_1\times\cdots\times M_m, \eta=\eta_1\times\cdots\times\eta_m$ and $\eta'=\eta_1'\times\cdots\times\eta_m'.$
Then we have  \[(G_{\eta^{\hH}},N,\eta)_{\hH}
  \geqslant_c(H_{\eta'^{\hH}}, M,\eta')_{\hH}.\]
  
\end{lem}
\begin{proof}
The case $m=2$ was settled in \cite[Lemma 2.5]{Nav-GMckay}; for the general case, we proceed by induction.
\end{proof}

If $N$ is a normal subgroup of $G,$ we denote by $\epsilon_N:G\ra\Aut(N)$ the group homomorphism induced by conjugation.
\begin{lem}\label{lem:H-triUD}
  Suppose that $(G,N,\eta)_{\hH}$ and $(H,M,\eta')_{\hH}$ are $\hH$-triples such that $G=NH, N\cap H=M$ and $\C_G(N)\subseteq H.$ Let $Q\subseteq M\cap \ker(\eta)\cap\ker(\eta')$ be a normal subgroup $G.$ We denote by $\bar{T}=TQ/Q$ for any subgroup $T$ of $G$ and $\bar{\eta},\bar{\eta}'$ the deflation of the corresponding characters to the related subgroups of $\bar{G}.$ Assume that the group homomorphism $\epsilon_N(G)\ra\epsilon_{\bar N}(\bar G)$ induced by quotient of maps is injective. Then we have that $(G,N,\eta)_{\hH}\geqslant_c(H,M,\eta')_{\hH}$ if and only if $(\bar{G},\bar{N},\bar{\eta})_{\hH}\geqslant_c (\bar{H},\bar{M},\bar{\eta}')_{\hH}.$ 
\end{lem}
\begin{proof}
It is easy to see that $(\bar G,\bar N,\bar{\eta})_{\hH}$ and $(\bar H,\bar M,\bar{\eta}')_{\hH}$ are $\hH$-triples such that $\bar G=\bar N\bar H$ and $\bar  N\cap\bar  H=\bar M.$
Moreover, $(\hH\times H)_{\eta}=(\hH\times H)_{\eta'}$ if and only if $(\hH\times \bar H)_{\bar{\eta}}=(\hH\times \bar H)_{\bar{\eta}'}.$
 By our assumption that the homomorphism $\epsilon_N(G)\ra \epsilon_{\bar N}(\bar G)$ induced by quotient of maps is injective, we have $\overline{\C_G(N)}=\C_{\bar{G}}(\bar{N}).$ 
 Suppose that $(\Pj,\Pj')$ gives the ordering relation of $(G,N,\eta)_{\hH}\geqslant_c(H,M,\eta')_{\hH}.$ 
 Then we can regard $\Pj$ and $\Pj'$ as projective representations of $\bar{G}_{\bar{\eta}}$ and $\bar{H}_{\bar{\eta}'},$ respectively, and check that it gives  $(\bar{G},\bar{N},\bar{\eta})_{\hH}\geqslant_c (\bar{H},\bar{M},\bar{\eta}')_{\hH}.$ 
 Conversely if $(\Pj,\Pj')$ gives  $(\bar{G},\bar{N},\bar{\eta})_{\hH}\geqslant_c (\bar{H},\bar{M},\bar{\eta}')_{\hH}.$ 
 Then we can lift $\Pj$ and $\Pj'$ to projective representations of $G_{\eta}$ and $H_{\eta'},$ respectively, and chenk that it gives $(G,N,\eta)_{\hH}\geqslant_c(H,M,\eta')_{\hH}.$
\end{proof}


\section{The Inductive NAW condition}\label{sec_ind_GAW}
In this section, we  introduce the \textbf{inductive NAW condition} for finite non-abelian simple groups.
Recall that the \textbf{universal $p'$-covering group} of a perfect group $L$ is the maximal perfect central extension of $L$ by an abelian $p'$-group. For a $p'$-group, we mean a finite group of order not divisible by $p.$

\begin{defi}[Inductive NAW condition for $L$ at $p$]\label{Def:ind-GAW}
  Let $L$ be a finite non-abelian  simple group of order divisible by $p$ and let $S$ be the universal $p'$-covering group of $L$. We say that the inductive NAW condition holds for $L$ (at $p$)  if the following conditions hold.
  \begin{enumerate}
    \item There exists an $\hH\times\Aut(S)$-equivariant bijection
    \[\Omega:\IBr(S)\ra\W^{\circ}(S)/\sim_S.\]
    \item Let ${A}=S\semi \Aut(S)$. Then for any $\eta\in\IBr(S)$ and $(Q,\delta)\in\Omega(\eta),$ we have $$({A}_{\eta^{\hH}},S,\eta)_{{\hH}}\geqslant_c
        (\N_{{A}}(Q)_{\delta^{\hH}},\N_S(Q),\delta)_{\hH}.$$
  \end{enumerate} 
\end{defi}

\begin{remark}\label{themethod}
  {\rm
  See (\ref{equ:wights}) for the definition of $\W^{\circ}(S).$ In fact condition (1) in the above definition guarantees the following:
  \begin{enumerate}
    \item $\N_{{A}}(Q)_{\eta^{\hH}} =\N_{{A}}(Q)_{\delta^{\hH}}$ and ${A}_{\eta^{\hH}}=S\N_{{A}}(Q)_{\delta^{\hH}}.$
    \item $\big(\hH\times \N_{{A}}(Q)\big)_{\eta}= \big(\hH\times \N_{{A}}(Q)\big)_{\delta}.$ 
  \end{enumerate}
  }
\end{remark}

The aim of the remaining part of this section is to give a proof of Theorem \ref{To_p'exten}, which is needed in a step in the final reduction Theorem \ref{Redution}.

\begin{lem}\label{To_directproduct}
  Let $m\geqslant 1$ be an integer. Let $L$ be a finite non-abelian  simple group of order divisible by $p,$ for which the inductive NAW condition holds.  Let $S$ be the universal $p'$-covering group of $L$. Then the following conditions hold.
  \begin{enumerate}
    \item There exists an $\hH\times\Aut(S^m)$-equivariant bijection
    \[\widetilde\Omega:\IBr(S^m)\ra\W^{\circ}(S^m)/\sim_{S^m}.\]
    \item Let $\widetilde{A}=S^m\semi \Aut(S^m)$. Then for any $\tilde\eta\in\IBr(S^m)$ and $(\widetilde Q,\tilde\delta)\in\widetilde\Omega(\tilde\eta),$ we have $$(\widetilde{A}_{\tilde\eta^{\hH}},S^m, \tilde\eta)_{{\hH}}\geqslant_c
        (\N_{\widetilde{A}}(\widetilde Q)_{\tilde\delta^{\hH}}, \N_{S^m}(\widetilde Q),\tilde\delta)_{\hH}.$$
  \end{enumerate}
\end{lem}
\begin{proof}
Since the inductive NAW condition holds for $L$, by our assumption we have an $\hH \times \Aut(S)$-equivariant bijection
\[
\Omega: \IBr(S) \longrightarrow \W^{\circ}(S)/\sim_S,
\]
such that, letting $A = S \rtimes \Aut(S)$, for any $\eta \in \IBr(S)$ and $(Q,\delta) \in \Omega(\eta)$ we have
\begin{equation}\label{equ:new2}
 (A_{\eta^{\hH}}, S, \eta)_{\hH} \geqslant_c (\N_A(Q)_{\delta^{\hH}}, \N_S(Q), \delta)_{\hH}.
\end{equation}
Any $\tilde{\eta} \in \IBr(S^m)$ can be written as $\tilde{\eta} = \eta_1 \times \cdots \times \eta_m$, where $\eta_i \in \IBr(S)$ for $1 \leqslant i \leqslant m$. For each $i$, choose $(Q_i, \delta_i) \in \Omega(\eta_i)$ and set
\[
\widetilde Q = Q_1 \times \cdots \times Q_m, \quad \tilde{\delta} = \delta_1 \times \cdots \times \delta_m.
\]
It is straightforward to verify that $(\widetilde{Q}, \tilde{\delta}) \in \W^{\circ}(S^m)$. 
Define $\widetilde{\Omega}(\tilde{\eta})$ to be the $S^m$-conjugacy class containing $(\widetilde{Q}, \tilde{\delta})$, it is independent of the choice of $(Q_i,\delta_i)\in\Omega(\eta_i).$

{\it Step 1.} The map  $\widetilde\Omega:\IBr(S^m)\ra\W^{\circ}(S^m)/\sim_{S^m}$ is an $\hH\times\Aut(S^m)$-equivariant bijection.

The injectivity of the map $\widetilde{\Omega}$ follows directly from the definition.  
By \cite[Lemma 2.3]{Nav1}, every radical subgroup $\widetilde{Q}$ of $S^m$ decomposes as $\widetilde{Q} = Q_1 \times \cdots \times Q_m$, where $Q_i \in \Rad(S)$ for $1 \leqslant i \leqslant m$.  
Since $\N_{S^m}(\widetilde{Q}) = \N_S(Q_1) \times \cdots \times \N_S(Q_m)$, any $\tilde{\theta} \in \dz(\N_{S^m}(\widetilde{Q}) / \widetilde{Q})$ can be written as $\tilde{\theta} = \theta_1 \times \cdots \times \theta_m$ with $\theta_i \in \dz(\N_S(Q_i)/Q_i)$ for all $1 \leqslant i \leqslant m$.  
The surjectivity of $\widetilde{\Omega}$ follows from this observation.

  By \cite[Lemma 10.24]{Nav-Mckaybook}, we have $\Aut(S^m)=\Aut(S)\wr \bS_m.$
  Thus for any $a\in\hH\times\Aut(S^m),$ we can write $a=\sigma a_1\cdots a_m f,$ where $\sigma\in\hH, a_i\in\Aut(S)$, and $f\in\bS_m$. 
  Let $\tilde\eta=\eta_1\times\cdots\times\eta_m\in\IBr(S^m).$ 
   For any $1\leqslant i\leqslant m$, choose $(Q_i, \delta_i) \in \Omega(\eta_i)$. Then
$
(\widetilde Q= Q_1 \times \cdots \times Q_m,  \tilde{\delta}=\delta_1 \times \cdots \times \delta_m)\in\widetilde{\Omega}(\tilde{\eta}).
$
  As introduced in the notation before Lemma \ref{lem:H-triB},
  we have $$\tilde\eta^a=\eta_{f^{-1}(1)}^{\sigma a_{f^{-1}(1)}}\times\cdots\times\eta_{f^{-1}(m)}^{\sigma a_{f^{-1}(m)}}.$$
Since $(Q_{f^{-1}(i)},\delta_{f^{-1}(i)})^{\sigma a_{f^{-1}(i)}}\in\Omega (\eta_{f^{-1}(i)}^{\sigma a_{f^{-1}(i)}})$ by the $\hH\times\Aut(S)$-equivariance of $\Omega,$ we have
\begin{align*}
  \widetilde{\Omega}(\tilde{\eta}^a)&=
  (Q_{f^{-1}(1)}^{a_{f^{-1}(1)}}\times\cdots\times Q_{f^{-1}(m)}^{a_{f^{-1}(m)}}, 
  \delta_{f^{-1}(1)}^{\sigma a_{f^{-1}(1)}} \times\cdots\times \delta_{f^{-1}(m)}^{\sigma a_{f^{-1}(m)}})_{\sim_{S^m}} \\
   & =(Q_1 \times \cdots \times Q_m,   \delta_1 \times \cdots \times \delta_m)^{\sigma a_1\cdots a_mf}_{\sim_{S^m}}\\
   &=(\widetilde{Q},\tilde{\delta})^a_{\sim_{S^m}} =\widetilde{\Omega}(\tilde\eta)^a.
  \end{align*}
This establishes the $\hH \times \Aut(S^m)$-equivariance of $\widetilde{\Omega}$ and completes the goal of the step.

 Let $\widetilde{A}=S^m\semi \Aut(S^m)$. 
Let $\tilde\eta\in\IBr(S^m)$ and $(\widetilde Q,\tilde\delta)\in\widetilde\Omega(\tilde\eta).$ 
 It remains to show that 
 \begin{equation}\label{equ:new3}
  (\widetilde{A}_{\tilde\eta^{\hH}},S^m, \tilde\eta)_{{\hH}}\geqslant_c
        (\N_{\widetilde{A}}(\widetilde Q)_{\tilde\delta^{\hH}}, \N_{S^m}(\widetilde Q),\tilde\delta)_{\hH}.
 \end{equation}
  By Lemma \ref{lem:H-triA}, we can adjust $\tilde\eta$ and $\widetilde Q$ by any automorphism of $\hH\times \widetilde{A}$.
   Note that $\Aut(S^m)=\Aut(S)\wr \bS_m.$
   Thus we may assume that
  \begin{align*}
    \tilde\eta&=\prod_{j=1}^{r}\prod_{k=1}^{s_j}\xi_{jk}^{m_{jk}} \quad {\rm and}\\
    \widetilde Q&=\prod_{j=1}^{r}Q_{j}^{\sum_{k=1}^{s_j} m_{jk}}
  \end{align*}
  such that the following conditions hold:
  \begin{enumerate}
    \item $(j,k)=(j',k')$ if and only if $\xi_{jk}$ and $\xi_{j'k'}$ are $\Aut(S)$-conjugate;
    \item $j = j'$ if and only if $\xi_{jk}$ and $\xi_{j'k'}$ are $\hH \times \Aut(S)$-conjugate and additionally satisfy $m_{jk} = m_{j'k'}$;
    \item $\xi_{jk}$ and $\xi_{j'k'}$ are $\hH$-conjugate if $j=j'$.
  \end{enumerate}
  Note that $m_{jk}\geqslant 1$ is an integer, and different $\xi_{jk}$ in the expression of $\tilde\eta$ are in lexicographic order relative to $(j,k)$ and different $Q_j$ in the expression of $\widetilde Q$ are in order relative to $j.$ 
  Of course $\sum_{j=1}^{r}\sum_{k=1}^{s_j}m_{jk}
  =\sum_{j=1}^{r}s_jm_{j1}=m$.
  Note also that $\tilde{\delta}$ is uniquely determined by $\tilde{\eta}$ and $\widetilde{Q};$ more precisely, \[\tilde\delta=\prod_{j=1}^{r}\prod_{k=1}^{s_j}\lambda _{jk}^{m_{jk}},\] where $(Q_j,\lambda_{jk})\in\Omega(\xi_{jk})$ for each such pair $(j,k).$
 Let $A=S\semi \Aut(S)$, and identify $\widetilde{A}$ with $A\wr \bS_m.$
 By assumption (see \ref{equ:new2}), for any $1\leqslant j\leqslant r,$ we have 
\[
 \big(A_{\xi_{j1}^{\hH}}, S, \xi_{j1}\big)_{\hH} \geqslant_c \big(\N_A(Q_j)_{\lambda_{j1}^{\hH}}, \N_S(Q_j), \lambda_{j1}\big)_{\hH}.
\]
Let $\tilde\eta_j=\prod_{k=1}^{s_j}\xi_{jk}^{m_{j1}}$ and $\tilde\delta_j=\prod_{k=1}^{s_j}\lambda_{jk}^{m_{j1}},1\leqslant j\leqslant r.$  Then applying Lemma \ref{lem:H-triB}, we have
\[
 \big((A\wr \bS_{s_j m_{j1}})_{\tilde{\eta}_j^{\hH}}, S^{s_j m_{j1}},\tilde{\eta}_j\big)_{\hH} \geqslant_c \big((\N_A(Q_j)\wr  \bS_{s_j m_{j1}})_{\tilde{\delta}_j^{\hH}}, \N_S(Q_j)^{s_jm_{j1}}, \tilde{\delta}_j\big)_{\hH}
\]
for every $j.$
Notice that if $a = \sigma a_1 \cdots a_m f$ stabilizes $\tilde{\eta}$, then $f$ cannot permute indices across different $j$; this means that
\[
(A \wr \bS_m)_{\tilde{\eta}^{\hH}} = \prod_{j=1}^{r} (A \wr \bS_{s_j m_{j1}})_{\tilde{\eta}_j^{\hH}},
\text{ and similarly, }
(\N_{A \wr \bS_m}(\widetilde{Q}))_{\tilde{\delta}^{\hH}} = \prod_{j=1}^{r} (\N_A(Q_j) \wr \bS_{s_j m_{j1}})_{\tilde{\delta}_j^{\hH}}.
\]
Applying Lemma \ref{lem:H-triC} then yields (\ref{equ:new3}), as desired. This completes the proof of the lemma. 
\end{proof}

\begin{lem}\label{To_perfectp'exten}
  Let $K$ be a finite perfect group. Suppose that $Z:=\ZZ(K)$ is $p'$-cyclic and $K/Z$ is a direct product of isomorphic non-abelian simple groups for which the inductive NAW condition holds. Then the following conditions hold.
  \begin{enumerate}
   \item There exists an $\hH\times\Aut(K)$-equivariant bijection
    \[\Omega:\IBr(K)\ra\W^{\circ}(K)/K.\]
    \item Let ${A}=K\semi \Aut(K)$. Then for any $\eta\in\IBr(K)$ and $(Q,\delta)\in\Omega(\eta),$ we have $$({A}_{\eta^{\hH}},S,\eta)_{{\hH}}\geqslant_c
        (\N_{{A}}(Q)_{\delta^{\hH}},\N_S(Q),\delta)_{\hH}.$$
  \end{enumerate}
\end{lem}
\begin{proof}
  Let $K/Z\cong L^m$, where $L$ is a non-abelian simple group for which the inductive NAW condition holds by our assumption. Let $S$ be the universal $p'$-covering group of $L.$ Since $K$ is a perfect $p'$-central extension of $L^m$, there exists a group epimorphism $$\epsilon:S^m\ra K$$ such that $(S^m,\epsilon)$ is the universal $p'$-central extension of $K$. 
  Note that $\epsilon^{-1}(Z)=\ZZ(S^m).$
  By the properties of the universal $p'$-central extensions, for every $\phi\in\Aut(K),$ there is a unique $\hat{\phi}\in\Aut(S^m)$ such that $\hat{\phi}\epsilon=\epsilon\phi$ (the maps here are composed from the left).
  And through the map $\phi\mapsto \hat{\phi},$ we identify $\Aut(K)$ with $\Aut(S^m)_{\ker(\epsilon)}.$   
  By Lemma \ref{thm_6.02}, there is a natural bijection from $\Rad(K)$ to $\Rad(S^m)$ given by sending $Q \in \Rad(K)$ to the normal Sylow $p$-subgroup of $\epsilon^{-1}(Q)$. We identify $\Rad(K)$ with $\Rad(S^m)$ via this bijection.
We follow the notation in Lemma \ref{To_directproduct}. Under the above identification, $\W^{\circ}(K)$ can be viewed as the subset of $\W^{\circ}(S^m)$ consisting of elements $(\widetilde{Q}, \tilde{\delta})$ such that $\tilde{\delta}$ lies over the trivial character of $\ker(\epsilon)$.
Via inflation of characters, we may also regard $\IBr(K)$ as a subset of $\IBr(S^m)$.
Then the equivariant bijection $\widetilde{\Omega}: \IBr(S^m) \to \W^{\circ}(S^m)/\sim_{S^m}$  restricts to an $\hH \times \Aut(K)$-equivariant bijection
\[
\Omega: \IBr(K) \longrightarrow \W^{\circ}(K)/\sim_{K}.
\]

Let $\eta \in \IBr(K)$ and $(Q,\delta) \in \Omega(\eta)$. Let $\widetilde{Q} \in \Rad(S^m)$ be the radical subgroup of $S^m$ corresponding to $Q$, and let $\tilde{\eta}$, $\tilde{\delta}$ denote the inflations of $\eta$ and $\delta$ to $S^m$ and $\N_{S^m}(\widetilde{Q})$, respectively. Note that $\epsilon\bigl(\N_{S^m}(\widetilde{Q})\bigr) = \N_K(Q)$.
Since $(\widetilde{Q},\tilde{\delta})\in\widetilde{\Omega}(\tilde\eta)$ by construction, we have \[\big( (S^m\semi \Aut(K))_{\tilde{\eta}^{\hH}}, S^m, \tilde\eta  \big)_{\hH}\geqslant_c \big( \N_{S^m\semi \Aut(K)}(\widetilde Q)_{\tilde{\delta}^{\hH}},\N_{S^m}(\widetilde{Q}),\tilde{\delta}    \big)_{\hH}.\]
By Lemma \ref{lem:H-triUD}, this ordering relation of $\hH$-triples descents to
\[\big( (K\semi \Aut(K))_{{\eta}^{\hH}}, K, \eta  \big)_{\hH}\geqslant_c \big( \N_{K\semi \Aut(K)}( Q)_{{\delta}^{\hH}},\N_{K}({Q}),{\delta}    \big)_{\hH}.\]
This completes the proof of the Lemma.
\end{proof}

In fact, we can omit the assumption that $K$ is perfect in Lemma \ref{To_perfectp'exten}, and we will discuss this in the following theorem.
For the central products of characters we follow the standard notation in \cite[Lemma 2.2]{Nav1}.
Let $Z \leqslant K$ be finite groups and $\chi \in \IBr(K)$. 
We denote by $\IBr(Z \which \chi)$ the set of irreducible Brauer characters of $Z$ over which  $\chi$ lies over.
Let $Z$ be a central subgroup of $K$ and let $\lambda \in \IBr(Z)$.
We denote by $\W^{\circ}(K \which \lambda) \subseteq \W^{\circ}(K)$ the subset consisting of elements $(Q,\delta)$ such that $\delta$ lies over $\lambda$; this subset is closed under $K$-conjugation.

\begin{thm}\label{To_p'exten}
  Let $Z$ be a $p'$-cyclic central subgroup of a finite group $K.$   Suppose that  $K/Z$ is a direct product of isomorphic non-abelian simple groups for which the inductive NAW condition holds, or that $K/Z$ is a $p'$-group. Let $\lambda\in\IBr(Z)$ be faithful. Then we have
  \begin{enumerate}
    \item  There exists an $\big(\hH\times \Aut(K)\big)_{\lambda}$-equivariant bijection \[\Omega:\IBr(K\which\lambda)\ra \W^{\circ}(K\which \lambda)/\sim_K.\]
    \item  Fix any $\eta\in\IBr(K\which \lambda)$ and $(Q,\delta)\in\Omega(\eta).$ Suppose that $K$ is normally embedded into a finite group $T$ and that $\eta$ is $T$-invariant. Then $\delta$ is $\N_T(Q)$-invariant and there exists an $\big(\hH\times\Aut(T)\big)_{\eta,Q}$-equivariant bijection $$\varDelta:\IBr(T\which \eta)\ra \IBr(\N_T(Q)\which \delta).$$
  \end{enumerate}
\end{thm}
\begin{proof}
  If $K/Z$ is a $p'$-group, then $\Rad(K)=\{1\}$ and all the conditions listed in the theorem are trivially satisfied.

  Now suppose that $K/Z$ is a direct product of isomorphic non-abelian simple groups for which the inductive NAW condition holds, thus $Z=\ZZ(K).$
  Let $K_1$ be the commutator subgroup of $K$, and set $Z_1 = Z \cap K_1 = \ZZ(K_1)$. Since $K/Z$ is perfect, we have $K = K_1Z$ and $K_1$ is perfect.
  Let $\lambda_1=\lambda_{Z_1}.$ By \cite[Lemma 2.2]{Nav1}, there is a bijection
\[
   \IBr(K_1\which \lambda_1)\ra\IBr(K\which \lambda),\quad \theta\mapsto\theta\cdot\lambda.
\]
  Since $K_1/Z_1\cong K/Z$ is a direct product of isomorphic non-abelian simple groups for which the inductive NAW condition holds, Lemma \ref{To_perfectp'exten} applies to $K_1$. 
  Consequently, there exists an $\hH\times\Aut(K_1)$-equivariant bijection \[\Omega_1:\IBr(K_1)\ra\W^{\circ}(K_1)/\sim_{K_1},\]
  such that for any $\theta\in\IBr(K_1)$ and $(Q,\varphi)\in\Omega_1(\theta),$ we have
  \begin{equation}\label{equ:new6}
    \big( (K_1\semi \Aut(K_1))_{\theta^{\hH}},K_1,\theta \big)_{\hH} \geqslant_{c}  
  \big( \N_{K_1\semi \Aut(K_1)}(Q)_{\varphi^{\hH}},\N_{K_1}(Q),\varphi \big)_{\hH}. 
  \end{equation} 
 Note that $\Rad(K_1)=\Rad(K).$

  For any $Q \in \Rad(K)$, note that $\N_K(Q) = \N_{K_1}(Q)Z$. Moreover, by \cite[Lemma 2.2]{Nav1}, there is a bijection
\[
\IBr(\N_{K_1}(Q) \which \lambda_1) \ra \IBr(\N_K(Q) \which \lambda), \quad \varphi \mapsto \varphi \cdot \lambda,
\]
and via the natural bijection (\ref{equ:dz-bij-vert.}), one verifies that $(Q,\varphi) \in \W^{\circ}(K_1)$ if and only if $(Q,\varphi \cdot \lambda) \in \W^{\circ}(K)$.
We define
\[
\Omega: \IBr(K \which \lambda) \ra \W^{\circ}(K \mid \lambda), \quad
\theta \cdot \lambda \longmapsto (Q,\varphi \cdot \lambda)_{\sim_K},
\]
where $\theta \in \IBr(K_1 \which \lambda_1)$ and $(Q,\varphi) \in \Omega_1(\theta)$.
It is straightforward to verify that $\Omega$ is an $\big(\hH \times \Aut(K)\big)_{\lambda}$-equivariant bijection.

Following the notation in condition (2), write $\eta = \theta \cdot \lambda$ and $\delta = \varphi \cdot \lambda$, where $(Q,\varphi) \in \Omega_1(\theta)$.  
Set $A = T \rtimes \Aut(T)_K$. Then $K_1 \trianglelefteq A$, and by the equivariance of $\Omega_1$ we have  
\[
A_{\theta^{\hH}} = K_1 \N_A(Q)_{\varphi^{\hH}} \quad \text{and} \quad
\big( \hH \times \N_A(Q) \big)_{\theta} = \big( \hH \times \N_A(Q) \big)_{\varphi},
\]  
as noted in Remark \ref{themethod}. 
Applying Theorem \ref{thm:H-triA} to (\ref{equ:new6}) and the $\hH$-triples $(A_{\theta^{\hH}}, K_1, \theta)_{\hH}$ and $(\N_A(Q)_{\varphi^{\hH}}, \N_{K_1}(Q), \varphi)_{\hH}$, we obtain  
\[
(A_{\theta^{\hH}}, K_1, \theta)_{\hH} \geqslant_c (\N_A(Q)_{\varphi^{\hH}}, \N_{K_1}(Q), \varphi)_{\hH}.
\]  
By Theorem \ref{thm:H-triB}, there exists an $\big(\hH \times \N_A(Q)\big)_{\theta}$-equivariant bijection  
\[
\tilde\Delta : \IBr(T \which \theta) \ra \IBr(\N_T(Q) \which \varphi)
\]  
such that $\IBr(Z \which\chi) = \IBr(Z \which \Delta(\chi))$ for all $\chi \in \IBr(T \which \theta)$. Note that $T\zg A.$ 
Since $\IBr(T \which \eta) = \IBr(T \which \theta) \cap \IBr(T \which \lambda)$ and $\IBr(\N_T(Q) \which \delta) = \IBr(\N_T(Q) \which \varphi) \cap \IBr(T \which \lambda)$, the bijection $\tilde\Delta$ restricts to an $\big(\hH \times \N_A(Q)\big)_{\eta}$-equivariant bijection  
\[
\Delta : \IBr(T \which \eta) \ra \IBr(\N_T(Q) \which \delta).
\]  
As $\big(\hH \times \N_A(Q)\big)_{\eta}$ induces the same automorphisms as $\big(\hH \times \Aut(T)\big)_{\eta, Q}$ on characters of $T$ and $\N_T(Q)$, the proof of the theorem is complete. 
\end{proof}


\section{The reduction}\label{Sec_red}
In this Section we are going to prove Theorem \ref{Redution}, which is a generalization of Theorem~\ref{thm-reduction}. Before  doing this, we give a lemma that is needed there.

\begin{lem}\label{thm_6.03}
  Let $Z$ be a normal subgroup of $G,$ and $\lambda\in\dz(Z)$ be $G$-invariant. If $S/Z$ is a $p$-subgroup of $G/Z$ and $\IBr(\N_G(S)\which\lambda^{\circ},|S/Z|)$ is non-empty, then $S/Z\in\Rad(G/Z).$
\end{lem}
\begin{proof}
By Theorem \ref{thm:centralize} and a standard reduction we may assume that $Z$ is a  $p'$-cyclic central subgroup of $G.$ Let $S_p$ be the unique Sylow $p$-subgroup of $S.$ Then $\N_G(S)=\N_G(S_p)$ and $\IBr(\N_G(S)\which\lambda^{\circ},|S/Z|)=\IBr(\N_G(S_p) \which\lambda^{\circ},|S_p|).$ Since $\IBr(\N_G(S_p)\which\lambda^{\circ},|S_p|)$ is non-empty, we have that $S_p\in\Rad(G)$ and thus $S/Z\in\Rad(G/Z).$
\end{proof}

\begin{thm}\label{Redution}
 Let $Z\zg G$ be finite groups. Let $\Inn(G)\leqslant \mathscr{A} \leqslant\big(\hH\times\Aut(G)\big)_Z,$ and let $\hat\lambda\in\dz(Z)$ be $ \mathscr{A} $-invariant. Set $\lambda=\hat{\lambda}^{\circ}.$ 
 Let $$\rw^{\circ}(G\which \lambda)=\Big\{(S,\varrho)\,\Big|\, S/Z\in\Rad\big(G/Z\big),\varrho\in\IBr(\N_G(S)\which\lambda,|S/Z|) \Big\}.$$
  Assume that the inductive NAW condition holds for every non-abelian simple group involved in $G/Z$ and of order divisible by $p$. Then there is an $ \mathscr{A} $-equivariant bijection $$f:\IBr(G\which\lambda)\ra \rw^{\circ}(G\which \lambda)/\sim_G.$$ 
\end{thm}
\begin{proof}
We will prove this theorem by induction on the order of the group $G/Z.$
By Theorem \ref{thm:centralize} and a standard reduction we may assume that $Z$ is a $p'$-central subgroup of $G$ and $\lambda$ is a faithful linear character.
Since there exists a bijection $\Rad(G)\ra\Rad(G/Z),S\mapsto SZ/Z,$ we can identify $\rw^{\circ}(G\which \lambda)$ with the set $$\W^{\circ}(G\which \lambda)=\{(S,\varrho)\,|\, S\in\Rad\big(G\big),\varrho\in\IBr(\N_G(S)\which\lambda,|S|) \}.$$ 
Throughout the proof, we only need to consider the set $\W^{\circ}(G\which \lambda)$, which means that we need to construct an $ \mathscr{A} $-equivariant bijection  $$f:\IBr(G\which\lambda)\ra \W^{\circ}(G\which \lambda)/\sim_G.$$

  Since $Z$ is $ \mathscr{A} $-invariant, $ \mathscr{A} $ acts on the subgroups of $G/Z.$ Let $K/Z$ be any minimal $ \mathscr{A} $-invariant subgroup of $G/Z$. As $K/Z$ is characteristically simple, it is a direct product of isomorphic simple groups. Note that $K$ is normal in $G,$ because $\Inn(G)\leqslant  \mathscr{A} .$ 

  {\itshape Step 1.} We may assume that $K/Z$ is not a $p$-group. Thus $\bO_p(G)=1.$ 

  Suppose that $K/Z$ is a $p$-group. Then  $K=K_p\times Z,$ where $K_p$ is the unique Sylow $p$-subgroup of $K.$ Note that $K_p$ is an $ \mathscr{A} $-invariant $p$-subgroup of $G.$ 
  Let $\bar{G}=G/K_p,$ and $\bar{Z}=K/K_p.$ Let $\bar{\lambda}=\lambda\times 1_{K_p}\in\IBr(\bar{Z}).$
  Define a group homomorphism $\omega: \mathscr{A} \ra \hH\times\Aut(\bar{G}),(\sigma,\phi)\ra (\sigma,\bar{\phi}),$ where $\bar{\phi}$ is the quotient of $\phi$ to  $\bar{G}.$ 
  Observe that $\omega(\Inn(G))=\Inn(\bar{G})$ and that $\bar{\lambda}$ is $\omega( \mathscr{A} )$-invariant.
  We aim to apply induction to $(\bar{G},\bar{Z},\bar{\lambda},\omega( \mathscr{A} )).$ 
  There is a natural bijection $\varDelta_1:\IBr(G\which\lambda)\ra\IBr(\bar{G}\which \bar{\lambda}),\varphi\mapsto\bar{\varphi},$ and by Lemma \ref{thm_6.02}(2) a natural bijection $\varDelta_2:\W^{\circ}(G\which \lambda)\ra\W^{\circ}(\bar{G}\which\bar{\lambda}), (S,\varrho)\mapsto(\bar{S},\bar{\varrho}).$ 
  Moreover, for $i\in\{1,2\},$ the map $\varDelta_i$ is $ \mathscr{A} $-equivariant, where $ \mathscr{A} $ acts via $\omega$ on the sets $\IBr(\bar G\which\bar\lambda)$ and $\W^{\circ}(\bar G\which
  \bar\lambda).$ 
  Since $|\bar{G}/\bar{Z}|<|G/Z|,$ by induction there exists an $\omega( \mathscr{A} )$-equivariant bijection $\IBr(\bar{G}\which \bar{\lambda})\ra \W^{\circ}(\bar{G}\which \bar{\lambda})/\sim_{\bar G}.$ 
   Composing with the bijections $\varDelta_1$ and $\varDelta_2$ yields the desired bijection $f.$  
If $\bO_p(G) \neq 1$, we may choose $K/Z \subseteq \bO_p(G)Z/Z$ to be a $p$-group. This completes the proof of this step.

Now we have that either $K/Z$ is a direct product of isomorphic non-abelian simple groups for which the inductive NAW condition holds, or $K/Z$ is a $p'$-group. Applying Theorem \ref{To_p'exten}, we obtain an $\mathscr{A}$-equivariant bijection
\[
\Omega: \IBr(K \which \lambda) \ra \W^{\circ}(K \which \lambda)/\sim_K
\]
such that for any $\eta \in \IBr(K \which \lambda)$ and $(Q,\delta) \in \Omega(\eta)$, there exists an $\mathscr{A}_{\eta,Q}$-equivariant bijection
\[
\varDelta^{(\eta,Q)}: \IBr(G_{\eta} \which \eta) \ra \IBr(\N_{G_{\eta}}(Q) \which \delta).
\]
Assume that the maps $\varDelta^{(\eta,Q)}$, as $(\eta,Q)$ varies, commute with the action of $\mathscr{A}$; that is,
\[
\varDelta^{(\eta,Q)}(\eta)^a = \varDelta^{(\eta^a, Q^a)}(\eta^a)
\]
for any pair $(\eta,Q)$ and any $a \in \mathscr{A}$.

  Let $$\textbf{A}=\big\{(\eta,\varphi)\,\big|\, \eta\in\IBr(K\which\lambda), \varphi\in\IBr(G_{\eta}\which\eta)\big\},$$
Then $ \mathscr{A} $ acts on $\textbf{A}$ in a natural way by $(\eta,\varphi)^a=(\eta^a,\varphi^a)$ for $a\in \mathscr{A} .$
We prove that the map
\[
f_1: \IBr(G \which \lambda) \ra \textbf{A}/\sim_G, \quad
\psi \mapsto (\eta,  \varphi)_{\sim_G}
\]
is well-defined and is an $\mathscr{A}$-equivariant bijection, where $\eta \in \IBr(K \which \lambda)$ is an irreducible constituent of the restriction of $\psi$ to $K$, and $\varphi$ is the Clifford correspondent of $\psi$ in $G_{\eta}$ (i.e., $\psi = \varphi^G$).
Suppose that $(\eta_1, \varphi_1) \in \textbf{A}$ also corresponds to $\psi$. Since $\psi$ lies over both $\eta$ and $\eta_1$, we see that $\eta$ is $G$-conjugate to $\eta_1$. 
Replacing $(\eta_1, \varphi_1)$ by a $G$-conjugate if necessary, we may assume that $\eta = \eta_1$.
 Then by the Clifford correspondence, we have $\varphi = \varphi_1$, and hence the map $f_1$ is well-defined. 
The $\mathscr{A}$-equivariance and the bijectivity of $f_1$ both follow from the relation $\psi = \varphi^G$.

Let
\[
\textbf{B} = \big\{ (\eta, Q, \varphi') \;\big|\; \eta \in \IBr(K \which \lambda),\; (Q,\delta) \in \Omega(\eta),\; \varphi' \in \IBr(U_{\eta,Q} \which \delta) \big\},
\]
where we set $U_{\eta,Q}=\N_{G_{\eta}}(Q),$ 
and let $\mathscr{A}$ act on $\textbf{B}$ in a natural way.
We prove that the map
\[f_2:\textbf{A}/\sim_G\ra \textbf{B}/\sim_G,(\eta,\varphi)_{\sim_G} \mapsto\big(\eta,Q,\varDelta^{(\eta,Q)}(\varphi)\big)_{\sim_G}\]
is well-defined and is an $\mathscr{A}$-equivariant bijection, where $Q\in\Rad(K)$ satisfies $(Q,\delta)\in\Omega(\eta).$
Let $(\eta_1,\varphi_1)\in\textbf{A}$ be another element such that $(\eta_1,\varphi_1)=(\eta^g,\varphi^g)$ for some $g\in G$, and let $(Q_1,\delta_1)\in\Omega(\eta_1).$ 
To show  $f_2$ is well-defined, we need to prove that $\big(\eta_1,Q_1,\Delta^{(\eta_1,Q_1)}(\varphi_1)\big)$ is $G$-conjugate to $\big(\eta,Q,\Delta^{(\eta,Q)}(\varphi)\big)$.
Since $\eta_1=\eta^g$, we have $(Q^g,\delta^g)\in\Omega(\eta_1)$ by the equivariance of $\Omega,$ thus $Q_1=Q^{gk}$ for some $k\in K.$
Then $\big(\eta,Q,\Delta^{(\eta,Q)}(\varphi)\big)^{gk}=
\big(\eta_1,Q_1,\Delta^{(\eta_1,Q_1)}(\varphi_1)\big)$.
It follows by direct computation that $f_2$ is an $\mathscr{A}$-equivariant surjection.
Let $(\eta_2,\varphi_2)\in\textbf{A},(Q_2,\delta_2)\in\Omega(\eta_2),$ and suppose that $\big(\eta_2,Q_2,\varDelta^{(\eta_2,Q_2)}(\varphi_2)\big)$ is $G$-conjugate to $\big(\eta,Q,\varDelta^{(\eta,Q)}(\varphi)\big)$.
For the injectivity of $f_2$, we need to show that $(\eta_2, \varphi_2)$ is $G$-conjugate to $(\eta, \varphi)$.  
Since we have already established the equivariance of $f_2$, we may assume that $\eta_2 = \eta$ and $Q_2 = Q$.  
Then $\Delta^{(\eta, Q)}(\varphi_2)$ is $U_{\eta, Q}$-conjugate to $\Delta^{(\eta, Q)}(\varphi)$, and hence they are equal.  
Since $\Delta^{(\eta, Q)}$ is a bijection, it follows that $\varphi_2 = \varphi$.

For any fixed $\eta\in\IBr(K\which\lambda)$ and $(Q,\delta)\in\Omega(\eta),$
since $|U_{\eta,Q}/\N_K(Q)|=|G_{\eta}/K|<|G/Z|$, we can apply induction to $(U_{\eta,Q}/Q,\N_K(Q)/Q, \hat{\bar\delta}, \mathscr{A} _{\eta,Q}).$
Here $\bar\delta\in \IBr(\N_K(Q)/Q)$ is the deflation of $\delta$ and $\hat{\bar\delta}\in\dz(\N_K(Q)/Q)$ is the corresponding character of $\bar{\delta}$ (i.e., $\hat{\bar\delta}^{\circ} = \bar\delta$; see \eqref{equ:dz-bij-vert.}).
By induction there exists an $ \mathscr{A} _{\eta,Q}$-equivariant bijection
\[\bar f_3^{(\eta,Q)}:\IBr(U_{\eta,Q}/Q\which \bar\delta)\ra \rw^{\circ}(U_{\eta,Q}/Q\which \bar\delta)/\sim_{U_{\eta,Q}/Q}.\]
Assume that the maps $\bar f_3^{(\eta,Q)}$, as $(\eta,Q)$ varies, commute with the action of $\mathscr{A}$.
Let
  \begin{align*}
    \widetilde{\rw^{\circ}}(U_{\eta,Q}/Q\which \bar\delta)=&\big\{  (E,\gamma)\,\big|\,E/\N_K(Q)\in\Rad(U_{\eta,Q}/\N_K(Q)),\\
     &\gamma\in\IBr(\N_{U_{\eta,Q}}(E)\which\delta,|E/\N_K(Q)| |Q|) \big\}.
  \end{align*}
By Lemma \ref{thm_6.02}(2), if $(E,\gamma) \in \widetilde{\rw^{\circ}}(U_{\eta,Q}/Q\which \bar\delta)$, then $(E/Q, \bar\gamma) \in \rw^{\circ}(U_{\eta,Q}/Q\which \bar\delta)$, where $\bar\gamma \in \IBr(\N_{U_{\eta,Q}/Q}(E/Q))$ denotes the deflation of $\gamma$. Note that $\N_{U_{\eta,Q}/Q}(E/Q) = \N_{U_{\eta,Q}}(E)/Q$.
We have that the map $(E,\gamma) \mapsto (E/Q, \bar\gamma)$ induces a natural bijection
\[
\widetilde{\rw^{\circ}}(U_{\eta,Q}/Q\which \bar\delta) \ra \rw^{\circ}(U_{\eta,Q}/Q\which \bar\delta).
\]  
Together with the natural bijection $\IBr(U_{\eta,Q} \which \delta) \ra \IBr(U_{\eta,Q}/Q \which \bar\delta)$, the map $\bar f_3^{(\eta,Q)}$ induces an $\mathscr{A}_{\eta,Q}$-equivariant bijection
\[
f_3^{(\eta,Q)}: \IBr(U_{\eta,Q} \which \delta) \ra \widetilde{\rw^{\circ}}(U_{\eta,Q}/Q\which \bar\delta)/\sim_{U_{\eta,Q}}.
\]
Moreover, the maps $ f_3^{(\eta,Q)}$, as $(\eta,Q)$ varies, commute with the action of $\mathscr{A}$.

  Let
  \begin{align*}
      \textbf{C}=\big\{ (\eta,Q,E,\gamma)  \,\big|\, &\eta\in\IBr(K\which\lambda), (Q,\delta)\in\Omega(\eta),\\
      &E/\N_K(Q)\in\Rad(U_{\eta,Q}/\N_K(Q)),\\
      &\gamma\in \IBr(\N_{U_{\eta,Q}}(E)\which\delta,|E/\N_K(Q)||Q|)  \big\},
  \end{align*} 
  and let $\mathscr{A}$ act on $\textbf{C}$ in a natural way.
 We prove that the map \[f_3:\textbf{B}/\sim_G\to\textbf{C}/\sim_G,\quad (\eta,Q,\varphi')_{\sim_G}\mapsto (\eta,Q,E,\gamma)_{\sim_G},\]  where $(E,\gamma)\in  f_3^{(\eta,Q)}(\varphi'),$ is well-defined and is an $\mathscr{A}$-equivariant bijection.
 Let $(\eta_1, Q_1, \varphi'_1) \in \textbf{B}$ be another element such that $(\eta_1, Q_1, \varphi'_1) = (\eta, Q, \varphi')^g$ for some $g \in G$, and let $(E_1, \gamma_1) \in f_3^{(\eta_1, Q_1)}(\varphi'_1)$.
By the equivariance of $f_3^{(\eta,Q)}$, we have $(E^g, \gamma^g) \in f_3^{(\eta_1, Q_1)}(\varphi'_1)$. Hence $(E_1, \gamma_1) = (E^g, \gamma^g)^u$ for some $u \in U_{\eta_1, Q_1}$.
Then $(\eta, Q, E, \gamma)^{gu} = (\eta_1, Q_1, E_1, \gamma_1)$, and therefore $f_3$ is well-defined.
It follows by direct computation that $f_3$ is an $\mathscr{A}$-equivariant surjection.
Let $(\eta_2, Q_2, \varphi'_2) \in \textbf{B}$ and $(E_2, \gamma_2) \in f_3^{(\eta_2, Q_2)}(\varphi'_2)$, and suppose that $(\eta_2, Q_2, E_2, \gamma_2)$ is $G$-conjugate to $(\eta, Q, E, \gamma)$. 
For the injectivity of $f_3$ we need to prove that $(\eta_2,Q_2,\varphi'_2)$ is $G$-conjugate to $(\eta,Q,\varphi').$
Since we have already established the equivariance of $f_3$, we may assume that $\eta_2 = \eta$ and $Q_2 = Q$. Then $(E_2, \gamma_2)$ is $U_{\eta, Q}$-conjugate to $(E, \gamma)$, and hence $\varphi'_2 = \varphi'$ because $f_3^{(\eta, Q)}$ is a bijection.

  Fix any  $\eta\in\IBr(K\which \lambda),(Q,\delta)\in\Omega(\eta)$ and $E/\N_K(Q)\in \Rad(U_{\eta,Q }/\N_K(Q)).$  
  Denote by $\bar\delta \in \IBr(\N_K(Q)/Q)$ the deflation of $\delta$, and by $\hat{\bar\delta} \in \dz(\N_K(Q)/Q)$ the corresponding character of $\bar{\delta}$ (i.e., $\hat{\bar\delta}^{\circ} = \bar\delta$; see \eqref{equ:dz-bij-vert.}).
 Let $F/Q$ be a defect group of the unique block of $E/Q$ covering the defect zero block $\bl(\bar\delta)$ of $\N_K(Q)/Q,$  and  let $C_F/Q=\C_{\N_K(Q)/Q}(F/Q).$
  Let $(\hat{\bar \delta})^{\star_{F/Q}}\in\dz(C_F/Q)$ be the DGN correspondent of $\hat{\bar\delta}$ with respect to $F/Q,$
  and set $\bar\delta^{\star_{F/Q}}=( (\hat{\bar \delta})^{\star_{F/Q}} )^{\circ}.$  
  Since $E=F\N_K(Q),$ we obtain that $\N_{U_{\eta,Q}}(F)\leqslant \N_{U_{\eta,Q}}(E),$ hence $\N_{U_{\eta,Q}}(F)=\N_{\N_{U_{\eta,Q}}(E)}(F).$
  Note that $E/Q\big/\N_K(Q)/Q$ is a normal $p$-subgroup of  $\N_{U_{\eta,Q}}(E)/Q\big/\N_K(Q)/Q$ and $\bar\delta$ is $\N_{U_{\eta,Q}}(E)/Q$-invariant. 
Thus, applying Theorem \ref{Thm:DGN}, with $(G,N,D,\theta)$ in that theorem replaced by $(\N_{U_{\eta,Q}}(E)/Q,\; \N_K(Q)/Q,\; F/Q,\; \hat{\bar\delta})$ in our setting, we obtain an $\mathscr{A}_{\eta,Q,E,F}$-equivariant bijection \[\bar f_4^{(\eta,Q,E,F)}: \IBr\big(\N_{U_{\eta,Q}}(E)/Q\bwhich\bar\delta,|F/Q|\big)\ra \IBr\big(\N_{U_{\eta,Q}}(F)/Q\bwhich \bar\delta^{\star_{F/Q}}, |F/Q|\big) .\]
  Assume that the maps $\bar f_4^{(\eta,Q,E,F)}$, as $(\eta,Q,E,F)$ varies, commute with the action of $\mathscr{A}$.
  Since deflation of characters induces natural bijections
\[
\IBr\big(\N_{U_{\eta,Q}}(E) \which \delta, |F|\big)
\ra \IBr\big(\N_{U_{\eta,Q}}(E)/Q \which \bar\delta, |F/Q|\big)
\]
and
\[
\IBr\big(\N_{U_{\eta,Q}}(F) \which \delta^{\star_{F}}, |F|\big)
\ra \IBr\big(\N_{U_{\eta,Q}}(F)/Q \which \bar\delta^{\star_{F/Q}}, |F/Q|\big)
\]
by Lemma \ref{thm_6.02}(2),
where $\delta^{\star_{F}} \in \IBr(C_F)$ is the inflation of $\bar\delta^{\star_{F/Q}}$,  the bijection $\bar f_4^{(\eta,Q,E,F)}$ induces an $\mathscr{A}_{\eta,Q,E,F}$-equivariant bijection
\[
f_4^{(\eta,Q,E,F)}: \IBr\big(\N_{U_{\eta,Q}}(E) \which \delta, |F|\big)
\ra \IBr\big(\N_{U_{\eta,Q}}(F) \which \delta^{\star_{F}}, |F|\big).
\]
Moreover, the maps $f_4^{(\eta,Q,E,F)}$, as $(\eta,Q,E,F)$ varies, commute with the action of $\mathscr{A}$.

Let
  \begin{align*}
    \textbf{D}=\big\{  (\eta,Q,E,F,\zeta) \,\big|\,&\eta\in\IBr(K\which\lambda), (Q,\delta)\in\Omega(\eta),E/\N_K(Q)\in\Rad(U_{\eta,Q}/\N_K(Q)), \\
     &F/Q  \text{ is a defect group of the unique block of } E/Q  \text{ covering } \bl(\bar\delta), \\
     &\zeta\in\IBr\big(\N_{U_{\eta,Q}}(F) \which \delta^{\star_{F}}, |F|\big)  \big\},
  \end{align*}
  and let $\mathscr{A}$ act on $\textbf{D}$ in a natural way.
 We prove that the map \[f_4:\textbf{C}/\sim_G\to\textbf{D}/\sim_G,\quad (\eta,Q,E,\gamma)_{\sim_G}\mapsto (\eta,Q,E,F,f_4^{(\eta,Q,E,F)}(\gamma))_{\sim_G},\]  where $(Q,\delta)\in \Omega(\eta)$ and $F/Q$  is a defect group of the unique block of $ E/Q$  covering $\bl(\bar\delta)$,  is well-defined and is an $\mathscr{A}$-equivariant bijection.
 Let $(\eta_1, Q_1, E_1,\gamma_1) \in \textbf{C}$ be another element such that $(\eta_1, Q_1, E_1,\gamma_1) = (\eta, Q, E,\gamma)^g$ for some $g \in G$, and let $(Q_1,\delta_1)\in \Omega(\eta_1)$ and $F_1/Q_1$  be a defect group of the unique block of $ E_1/Q_1$ covering $\bl(\bar\delta_1),$ where $\bar\delta_1\in\IBr(\N_K(Q_1)/Q_1)$ is the deflation of $\delta_1.$
Since $F^g/Q_1$ is also a defect group of the unique block of $ E_1/Q_1$ covering $\bl(\bar\delta_1),$ we have $F_1=F^{gx}$ for some $x\in E_1.$
Then $(\eta, Q, E, F, f_4^{(\eta,Q,E,F)}(\gamma))^{gx} =(\eta_1, Q_1, E_1, F_1, f_4^{(\eta_1,Q_1,E_1,F_1)}(\gamma_1))$, and therefore $f_4$ is well-defined.
It follows by direct computation that $f_4$ is an $\mathscr{A}$-equivariant surjection.
Let $(\eta_2, Q_2, E_2,\gamma_2) \in \textbf{C},(Q_2,\delta_2)\in \Omega(\eta_2)$, and $F_2/Q_2$  be a defect group of the unique block of $ E_2/Q_2$ covering $\bl(\bar\delta_2).$ 
Suppose that $(\eta_2, Q_2, E_2, F_2,f_4^{(\eta_2,Q_2,E_2,F_2)}(\gamma_2))$ is $G$-conjugate to $(\eta,Q,E,F,f_4^{(\eta,Q,E,F)}(\gamma)).$
For the injectivity of $f_4$ we need to prove that $(\eta_2, Q_2, E_2,\gamma_2)$ is $G$-conjugate to $(\eta,Q,E,\gamma).$
Since we have already established the equivariance of $f_4$, we may assume that $\eta_2 = \eta, Q_2 = Q$ and $E_2=E$. 
Then $(F_2, f_4^{(\eta_2,Q_2,E_2,F_2)}(\gamma_2))$ is $\N_{U_{\eta, Q}}(E)$-conjugate to $(F,f_4^{(\eta,Q,E,F)}(\gamma))$, and hence $\gamma_2 = \gamma$ because $f_4^{(\eta, Q,E,F)}$ is a bijection.

  Fix any $(\eta,Q,E,F)$ such that $\eta\in\IBr(K\which \lambda), (Q,\delta)\in\Omega(\eta), E/\N_K(Q)\in\Rad(U_{\eta,Q}/\N_K(Q))$ and $F/Q$ a defect group of the unique block of $E/Q$ covering $\bl(\bar\delta).$
  Since $F\cap K=Q,$ we see $\N_G(F)\leqslant \N_G(Q),$ thus $C_F\zg \N_G(F).$ 
  In fact, $\N_{U_{\eta,Q}}(F)$ is the stabilizer of $\delta^{\star_F}$ in $\N_G(F).$ Indeed, $\delta^{\star_F}$ is $\N_{U_{\eta,Q}}(F)$-invariant and if $x\in\N_G(F)\leqslant\N_G(Q)$ stabilizes $\delta^{\star_F},$ then $x$ stabilizes $\eta,$ thus $x\in U_{\eta,Q}.$ By Clifford correspondence, we have that induction defines a bijection
  $$\IBr\big( \N_{U_{\eta,Q}}(F) \,\big|\,\delta^{\star_F},|F| \big)\to \IBr\big( \N_G(F) \,\big|\,\delta^{\star_F},|F| \big)  \big).$$
  
  Let
  \begin{align*}
    \textbf{E}=\big\{  (\eta,Q,E,F,\chi) \,\big|\, & \eta\in\IBr(K\which\lambda),(Q,\delta)\in\Omega(\eta), \\
    &E/\N_K(Q)\in\Rad(U_{\eta,Q}/\N_K(Q)),\\
    &F/Q \text{ is a defect group of the unique block of } E/Q \text{ covering } \bl(\bar\delta), \\
     &\chi\in\IBr\big( \N_G(F) \,\big|\,\delta^{\star_F},|F| \big)   \big\},
  \end{align*}
 and let $ \mathscr{A} $ act on $\textbf{E}$ in a natural way. 
 Then the $ \mathscr{A} $-equivariant bijection 
  \[\tilde f_5:\ \textbf{D}\ \to \ \textbf{E},\quad (\eta,Q,E,F,\zeta)\mapsto(\eta,Q,E,F,\zeta^{\N_G(F)}),\] 
 induces an $\mathscr{A} $-equivariant bijection 
  \[ f_5: \textbf{D}/\sim_G \ra \textbf{E}/\sim_G,\quad (\eta,Q,E,F,\zeta)_{\sim_G} \mapsto(\eta,Q,E,F,\zeta^{\N_G(F)})_{\sim_G}.\]

Recall that
\[
\W^{\circ}(G\which \lambda) = \Big\{ (S,\varrho) \;\Big|\; S \in \Rad(G),\; \varrho \in \IBr(\N_G(S) \which \lambda, |S|) \Big\}.
\]
Let $(\eta, Q, E, F, \chi) \in \textbf{E}$. Since $\IBr\big( \N_G(F) \,\big|\, \delta^{\star_F}, |F| \big)$ is nonempty, we have $F \in \Rad(G)$.
  Thus $(F,\chi)\in \W^{\circ}(G\which \lambda),$ and the map
  $$f_6:\ \textbf{E}/\sim_G\ \to \ \W^{\circ}(G\which \lambda)_{\sim_G},\quad (\eta,Q,E,F,\chi)_{\sim_G}\mapsto (F,\chi)_{\sim_G}$$
  is  well-defined. 
  In the following we will prove that $f_6$ is an  $ \mathscr{A} $-equivariant bijection.
   Consequently, letting
\[
f = f_6 f_5 f_4 f_3 f_2 f_1: \IBr(G \which \lambda) \ra \W^{\circ}(G \which \lambda)_{\sim_G}
\]
be the composition of the equivariant bijections constructed above, we obtain the desired $f$.

  {\itshape Step 2.} The  map $f_6$  defined above is surjective.

  Fix any $S\in\Rad(G)$ and $\varrho \in \IBr(\N_G(S) \which \lambda, |S|)$, and let $Q=S\cap K$. By \cite[Lemma 2.3]{Nav1}, we have $Q\in\Rad(K).$ 
  Since $K$ is a normal subgroup of $G,$ we have $\N_G(S)\leqslant \N_G(Q).$
  Let $C=\N_{\N_K(Q)}(S)$ and $E=S\N_K(Q).$  
  Let $\hat{\varrho}\in\dz(\N_G(S)/S)$ be the character corresponding to $\varrho$; see (\ref{equ:dz-bij-vert.}).
  Let $\mu_1\in\Irr(SC/S)$ be a character over which $\hat\varrho$ lies.
  Since $SC\zg \N_G(S),$ we have $\mu_1\in\dz(SC/S).$ Via the natural isomorphism $SC/S\cong C/Q$, we obtain a defect zero character $\mu\in\dz(C/Q)$.
  Let $\xi\in\dz(\N_K(Q)/Q)$ be the DGN correspondent of $\mu$ with respect to $S/Q,$ and let $\delta\in\IBr(\N_K(Q)\which |Q|)$ correspond to $\xi$ under the bijection (\ref{equ:dz-bij-vert.}).
 Tracing the steps above, we see that $\delta$ lies over $\lambda$; hence there exists $\eta \in \IBr(K \which \lambda)$ such that $(Q,\delta) \in \Omega(\eta)$.
Then it follows that $(\eta, Q, E, S, \varrho) \in \textbf{E}$.
  Note that $\IBr\big(\N_{U_{\eta,Q}}(E)/Q \,\big|\, \xi^{\circ}, |S/Q|\big)$ is nonempty, since by construction it corresponds to the set $\IBr\big( \N_G(S)/Q \which \mu^{\circ}, |S/Q| \big)$, to which $\bar\varrho$ belongs, where $\bar\varrho \in \IBr(\N_G(S)/Q)$ denotes the deflation of $\varrho$.  
Thus, by Lemma \ref{thm_6.03}, we have $E/\N_K(Q) \in \Rad\big(U_{\eta,Q} / \N_K(Q)\big)$. This completes the proof of this step.

By construction, $f_6$ is $\mathscr{A}$-equivariant. Let $(\eta, Q, E, F, \chi), (\eta_1, Q_1, E_1, F, \chi) \in \textbf{E}$.  
We see that $Q = K \cap F = Q_1$ and $E = F \N_K(Q) = E_1$. Set $C = \N_{\N_K(Q)}(F)$.
Let $(Q,\delta)\in\Omega(\eta)$ and $(Q,\delta_1)\in\Omega(\eta_1).$
Since $\delta^{\star_F}, \delta_1^{\star_F} \in \IBr(C)$ both lie under $\chi$, we have $\delta_1^{\star_F} = \big(\delta^{\star_F}\big)^x$ for some $x \in \N_G(F)$.  
It follows that $\delta_1=\delta^x$, and then $\eta_1=\eta^x$ by the equivariance of $\Omega.$
This proves the injectivity of $f_6.$  

This completes the proof of the  theorem.
\end{proof}

Keep the notation in Theorem \ref{Redution}. Set $Z=1$ and $ \mathscr{A} =\hH\times\Aut(G).$ Then we obtain Theorem~\ref{thm-reduction}.


\section{Groups of Lie type in defining characteristic}\label{sec_defining_char}

Let $G$ be a finite group.
Recall that a weight of $G$ can be also defined as a pair $(Q,\varphi)$, where $\varphi\in\Irr(\N_G(Q)/Q)$ with $\varphi(1)_p=|\N_G(Q)/Q|_p$.
In the following, when mentioning a weight, we often mean this version.
Note that $\varphi^{\circ}$, restriction of $\varphi$ to all $p'$-elements of $\N_G(Q)/Q$, is a projective irreducible Brauer character of $\N_G(Q)/Q$. 
In this way, we can identify $(Q,\varphi)$ with $(Q,\varphi^{\circ})$. We also regard $\varphi$ as an irreducible character of $\N_G(Q).$

If $(R,\varphi)$ is a weight of $G$, then we write $\overline{(R,\varphi)}$ for the $G$-conjugacy class containing $(R,\varphi)$.
Let $\cW(G)$ denote the set of $G$-conjugacy classes of weights of $G$.
For $\nu\in\IBr(\ZZ(G))$, we denote by $\cW(S\which \nu)$ the subset of $\cW(G)$ consisting of $\overline{(R,\varphi)}$ with $\nu\in\IBr(\ZZ(G)\which \varphi^{\circ})$, where $\IBr(\ZZ(G)\which \varphi^{\circ})$ is the set of irreducible Brauer characters of $\ZZ(G)$ over which $\varphi^{\circ}$ lies.
Now we adapt \cite[Definition~3.1]{BS22} for the inductive NAW condition.

\begin{prop}\label{def:ind-GAW-cond-2}
  Let $L$ be a finite non-abelian  simple group of order divisible by $p$ and $U$ be the universal covering group of $L.$ Let $U_0$ be a characteristic central $p$-subgroup of $U$ and $S=U/U_0.$ 
 Assume that the following conditions hold.
  \begin{enumerate}
    \item There is an $\hH\times \Aut(S)$-equivariant bijection \[\Omega: \IBr(S)\to\cW(S)\] such that $\Omega(\IBr(S\which \nu)) = \cW(S\which \nu)$ for every $\nu\in\IBr(\ZZ(S))$.
    \item Fix any $\eta\in\IBr(S)$, and let $\overline{(Q,\varphi)}=\Omega(\eta)$. Let $Z$ be an $\Aut(S)_{\eta^{\hH}}$-invariant subgroup of $\ZZ(S)\cap\ker\eta$. 
    Write $\bar S=S/Z$, $\bar Q=QZ/Z,$ and  also regard $\eta$ as a character of $\bar{S}$ and $\varphi$ as a character of $\N_{\bar S }(\bar Q)$.
    Then $\bar S$ can be normally embedded into a finite group $\widetilde{A},$ and there exists a normal subgroup $A$ of $\widetilde{A}$ containing $\bar S\C_{\widetilde{A}}(\bar S)$, and characters $\tilde\eta\in\IBr(A)$ and $\tilde\varphi\in\IBr(\N_A(\bar Q))$ such that  the following conditions hold.
        \begin{enumerate}
          \item $\widetilde{A}/\C_{\widetilde{A}}(\bar S)\cong \Aut(\bar S)_{\eta^{\hH}}$ and $A/\C_{\widetilde{A}}(\bar S)\cong \Aut(\bar S)_{\eta}$ through the natural group homomorphism. Moreover, $\C_{\widetilde{A}}(\bar S)=\ZZ(A)$. 
          \item $\tilde\eta$ is an extension of $\eta$ and $\tilde\varphi$ is an extension of $\varphi^{\circ}$ such that $\IBr(\ZZ(A)\which \tilde\eta)=\IBr(\ZZ(A)\which  \tilde\varphi)$.
          \item For any $a\in(\hH\times \N_{\widetilde{A}}(\bar Q))_\varphi$, let $\tilde\eta^{a}=\mu_a\tilde\eta$ and $\tilde\varphi^a=\mu'_a\tilde\varphi$, where $\mu_a,\mu'_a$ are linear Brauer characters of $A/\bar S\cong \N_A(\bar Q)/\N_{\bar S}(\bar Q)$, then we have $\mu_a=\mu'_a.$
        \end{enumerate}
  \end{enumerate}
Then the inductive NAW condition holds for $L$ at the  prime $p$.
\end{prop}

\begin{proof}
Let $\widetilde{G} = S \rtimes \Aut(S)_{\eta^{\hH}}$. We aim to prove that $(\widetilde{G}, S, \eta)_{\hH} \geqslant_c (\N_{\widetilde{G}}(Q), \N_S(Q), \varphi^{\circ})_{\hH}.$
Since $\Omega$ is $\hH \times \widetilde{G}$-equivariant and $\eta^{\hH}$ is $\widetilde{G}$-invariant, it follows easily that $\widetilde{G} = S \N_{\widetilde{G}}(Q)$ and
$
(\hH \times \N_{\widetilde{G}}(Q))_{\eta} = (\hH \times \N_{\widetilde{G}}(Q))_{\varphi^{\circ}}.
$
Notice that $Q$ is the unique Sylow $p$-subgroup of $QZ$ (see Lemma \ref{thm_6.02}(2)), and $Z$ is normal in $\widetilde{G}$ by assumption. Hence an element of $\widetilde{G}$ normalizes $Q$ if and only if it normalizes $QZ$. Consequently, in $\widetilde{G}/Z$ we have
$
\N_{\widetilde{G}}(Q)/Z = \N_{\widetilde{G}/Z}(\bar{Q})$ and $
\N_S(Q)/Z = \N_{\bar{S}}(\bar{Q}),
$
where $\bar{Q} = QZ/Z$, $\bar{S} = S/Z$.
Moreover, $\Aut(S)_{\eta^{\hH}} = \Aut(\bar{S})_{\eta^{\hH}}$ via the quotient map. This holds because $\Aut(S) = \Aut(L)$ by properties of universal covering groups (see \cite[Corollary B.8]{Nav-Mckaybook}) and the fact that $U_0$ is characteristic in $U$. Thus $\Aut(\bar{S})$ can be identified with the subgroup $\Aut(S)_Z$ of $\Aut(S)$. Since $Z$ is stabilized by $\Aut(S)_{\eta^{\hH}}$, we have $\Aut(\bar{S})_{\eta^{\hH}} = \Aut(S)_{\eta^{\hH}}$.
Condition (2) of the proposition then yields
\[
(\widetilde{A}, \bar{S}, \eta)_{\hH} \geqslant_c (\N_{\widetilde{A}}(\bar{Q}), \N_{\bar{S}}(\bar{Q}), \varphi^{\circ})_{\hH}.
\]
Observe that $\N_{\widetilde{A}}(\bar{Q})$ and $\N_{\widetilde{G}/Z}(\bar{Q})$ induce the same automorphism group $\Aut(\bar{S})_{\eta^{\hH}, \bar{Q}}$ on $\bar{S}$. Hence by Theorem \ref{thm:H-triA}, we have
\[
(\widetilde{G}/Z, \bar{S}, \eta)_{\hH} \geqslant_c (\N_{\widetilde{G}}(Q)/Z, \N_{\bar{S}}(\bar{Q}), \varphi^{\circ})_{\hH}.
\]
Finally, applying Lemma \ref{lem:H-triUD} yields the desired relation
\[
(\widetilde{G}, S, \eta)_{\hH} \geqslant_c (\N_{\widetilde{G}}(Q), \N_S(Q), \varphi^{\circ})_{\hH}.
\]

Let $\ZZ(S)_p$ denote the unique Sylow $p$-subgroup of $\ZZ(S)$.
 Then $S/\ZZ(S)_p$ is the universal $p'$-covering group of $L$, and  quotient of maps induces a bijection $\Aut(S) \to \Aut(S/\ZZ(S)_p)$ (see \cite[Appendix B]{Nav-Mckaybook}). Hence we may identify $\Aut(S) = \Aut(S/\ZZ(S)_p)$.
To verify that the inductive NAW condition holds for $L$, we work with $S/\ZZ(S)_p$. Note that $\IBr(S) = \IBr(S/\ZZ(S)_p)$ via inflation of characters. By Lemma \ref{thm_6.02}(2), the natural map $(Q,\varphi) \mapsto (Q/\ZZ(S)_p, \varphi)$ yields an identification $\cW(S) = \cW(S/\ZZ(S)_p)$, where $(Q,\varphi)$ is a weight of $S$.
Thus there exists an $\hH \times \Aut(S/\ZZ(S)_p)$-equivariant bijection
\[
\Omega: \IBr(S/\ZZ(S)_p) \ra \cW(S/\ZZ(S)_p)
\]
that preserves the central characters on $\ZZ(S)/\ZZ(S)_p$.
Now let $\eta$, $Q$, $\varphi$, and $\widetilde{G}$ be as in the previous paragraph. There we have shown that
$
(\widetilde{G}, S, \eta)_{\hH} \geqslant_c (\N_{\widetilde{G}}(Q), \N_S(Q), \varphi^{\circ})_{\hH}.
$
Applying Lemma \ref{lem:H-triUD} gives
\[
\big( \widetilde{G}/\ZZ(S)_p,\; S/\ZZ(S)_p,\; \eta \big)_{\hH}
\geqslant_c \big( \N_{\widetilde{G}}(Q)/\ZZ(S)_p,\; \N_S(Q)/\ZZ(S)_p,\; \varphi^{\circ} \big)_{\hH}.
\]
Observe that $\N_{\widetilde{G}}(Q)/\ZZ(S)_p = \N_{\widetilde{G}/\ZZ(S)_p}(Q/\ZZ(S)_p)$ and $\N_S(Q)/\ZZ(S)_p = \N_{S/\ZZ(S)_p}(Q/\ZZ(S)_p)$.
In view of the above, we may assume without loss of generality that $\ZZ(S)_p = 1$.

Now $S$ is the universal $p'$-covering group of $L,$ and  conditions (1) and (2) of Definition \ref{Def:ind-GAW} hold in our case. 
\end{proof}  

\begin{remark}\label{rmk:add-condition}
In Proposition~\ref{def:ind-GAW-cond-2}, we assume that conditions (1) and (2.a), (2.b) hold, then the condition (2.c) holds automatically if one of the following is satisfied.
\begin{enumerate}
\item $A/\bar S$ is $p$-group.
\item Both $\tilde\eta$ and $\tilde\varphi$ are $(\hH\times \N_{\widetilde{A}}(\bar Q))_\varphi$-stable.
\end{enumerate}
\end{remark}  

\begin{prop}\label{prop:3-groups}
The inductive NAW condition holds for
  simple groups $\textup{Sp}_4(2)',\textup G_2(2)'$ and $ {}^2\textup F_4(2)'$ at the prime 2.
\end{prop}

\begin{proof}
When $L\in\{\textup{Sp}_4(2)'\cong A_6, \textup G_2(2)'\cong\textup{SU}_3(3), {}^2\textup F_4(2)'\}$, the simple group $L$ has trivial Schur multiplier (i.e., $S=L$ is the universal covering group of $L$), and $\textup{Out}(S)$ is of order 2 or isomorphic to Klein four group. Thanks to Remark~\ref{rmk:add-condition} (1), the condition (2.c) of Proposition~\ref{def:ind-GAW-cond-2} holds. 
By \cite[Proposition~6.1]{Nav1}, the condition (2) of Proposition~\ref{def:ind-GAW-cond-2} is satisfied.
Also, an $\Aut(S)$-equivariant bijection $\Omega: \IBr(S)\to\cW(S)$ is already established in the proof of \cite[Proposition~6.1]{Nav1}.
Note that in the construction of $\Omega$, one has that $\Omega(\eta)=\varphi$, if $\eta$ lies in a block $B$ of defect zero with $\Irr(B)=\{\varphi\}$.
It follows that $\Omega(\eta^a)=\Omega(\eta)^a$ since $\eta=\varphi^{\circ}$.
Thus it suffices to show every element in $\IBr(S)\cup\cW(S)$ lying in a block of positive defect is $\hH$-invariant.

The Brauer character table of the group $S\in\{\textup{Sp}_4(2)', \textup G_2(2)', {}^2\textup F_4(2)'\}$ can be found in \cite{JLPW}. Accordingly, the irreducible Brauer characters of $S$ lying in a block of positive defect take rational values and thus $\hH$-invariant.
Now let $(Q,\varphi)$ be a weight of $S$ lying in a block of positive defect, that is, $Q>1$.
Then $\N_S(Q)/Q\cong 1$, $S_3$ or $C_5\rtimes C_4$ (cf. \cite{CCNPW85}, see also the proof of \cite[Proposition~6.1]{Nav1}). For each situation, we can show that the irreducible characters of $\N_S(Q)/Q$ with defect zero are $\hH$-invariant.
In fact, all irreducible characters of the symmetric group $S_3$ take rational values, and the group $C_5\rtimes C_4$ has a unique irreducible character of defect zero.
Thus we complete the proof.
\end{proof}  

Next, we check the inductive NAW condition  for simple groups of Lie type at the defining characteristic.

\begin{proof}[Proof of Theorem~\ref{thm:def-char}]
  According to Proposition~\ref{prop:3-groups}, we may assume $L\notin\{\textup{Sp}_4(2)', \textup G_2(2)', {}^2\textup F_4(2)'\}$ for $p=2$.
  For $p=3$, we may likewise exclude the group ${}^2G_2(3)'\cong\textup{SL}_2(8)$, as it will be handled Proposition~\ref{prop:sl2} (below).
 Note that we only use the non-defining characteristic part of the proof of Proposition~\ref{prop:sl2}, so the reasoning there is independent of the argument that we need from Proposition~\ref{prop:sl2}.
For every other situation the exceptional part of the Schur multiplier of $L$ is a $p$-group, see \cite[Table 6.1.3]{GLS98}.
Hence there exists some simply-connected simple algebraic group $\mathcal G$ defined over $\overline{\mathbb F}_p$ and some Steinberg map $F:\mathcal G\to\mathcal G$ such that $S=\mathcal G^F$ is a covering group of $L$ such that $S/\ZZ(S)_p$ is a universal $p'$-covering group of $L$.
Write $\overline{\mathbb F}_p$ as $\mathbb F$.

(I) The simple group $S$ is AWC-good for $p$ according to \cite[Theorem~C]{Nav1}; we first recall the $\Aut(S)$-equivariant bijection $\Omega:\IBr(S)\to\cW(S)$ constructed there. 

The group $S$ has a restricted split $(B,N)$-pair in characteristic $p$.
Fix a Borel subgroup of $S$. By a theorem of Curtis \cite{Cu70} there exists a bijection between the set of isomorphism classes $[V]$ of simple $\mathbb FS$-modules $V$, 
\[[V]\longleftrightarrow (P,\lambda)\]
and the set $\mathscr P(S)$ of pairs $(P,\lambda)$, where $P$ is any parabolic subgroup containing $B$ of $S$, and $\lambda$ is a group homomorphism $P\to\mathbb F^\times$.
This bijection is defined as follows. Every simple $\mathbb FS$-module $V$ containing a unique $B$-invariant line $\mathbb Fv$, $P$ is the stabilizer of $\mathbb Fv$ in $S$, and $g.v=\lambda(g)v$ for all $g\in P$.

According to \cite{BT71}, any radical $p$-subgroup $Q$ of $S$ is conjugate to $\bO_p(P)$ for some parabolic subgroup $P$ containing $B$ of $S$.
Moreover, $\N_S(Q)=\N_S(P)=P=Q\rtimes Y$, where $Y$ is a Levi subgroup.
By \cite{Alp-conj} or the proof of \cite[Theorem~C]{Nav1}, if $(Q,\varphi)$ is a $p$-weight of $S$, then $\varphi\in\dz(L)$ which forces that $\varphi=\lambda\cdot\textup{St}_Y$, where $\textup{St}_Y$ is the Steinberg character of $Y$, and $\lambda\in\Irr(Y)$ is a linear character.
Equivalently, $\lambda$ can be regarded as a group homomorphism $P\to\mathbb F^\times$.
In this way, \[(Q,\varphi)\longleftrightarrow (P,\lambda)\] induces a bijection $\cW(S)\to\mathscr P(S)$. This is also the map constructed by Alperin \cite{Alp-conj}.

Therefore, we get a bijection $\Omega:\IBr(S)\to\cW(S)$.
By the proof of \cite[Theorem~C]{Nav1}, $\Omega$ is $\Aut(S)$-equivariant and satisfies that $\Omega(\IBr(S\which \nu)) = \cW(S\which \nu)$ for every $\nu\in\IBr(\ZZ(S))$.

For later use, we recall the description of the behavior of the parameters under the action of $\Aut(S)$ from the proof of \cite[Theorem.~C]{Nav1} as follows. Let $a\in\Aut(G)$. Then there exists $s\in S$ with $U^a=U^s$, where $U=\bO_p(B)$ is a Sylow $p$-subgroup of $S$. 
Then $[V^a]\longleftrightarrow (P^{as},\lambda^{as})$.

(II) Next, we prove that the bijection $\Omega$ is $\hH$-equivariant, which implies the condition (1) of Proposition~\ref{def:ind-GAW-cond-2}. Let $\sigma\in\hH$.

Let $\eta$ be an irreducible Brauer character of $S$ and $V$ be a simple $\mathbb FS$-module affording $\eta$. Suppose that $[V]\leftrightarrow (P,\lambda)$.
Let $n$ denote the dimension of $V$.
Identify $V$ to $\mathbb F^n$, and we let $\rho:S\to\text{GL}(n,\mathbb F)$ denote the corresponding representation of the $\mathbb FS$-module $V$. 
Then we construct the irreducible $\mathbb FS$-module $V^\sigma$ by considering the new action of $S$ on the space $V$ as $g\circ w=\sigma(\rho(g))(w)$ for any $g\in S$ and $w\in V$.
So $V^\sigma$ affords $\eta^\sigma$ by construction.
Let $\mathbb Fv$ be the $B$-invariant line of $V$, with $v\in\mathbb F^n$.
Consider the action of $\hH$ in $\mathbb F^n$, induced by its action on $\mathbb F$.
It can be seen that $\mathbb Fv^\sigma$ is the $B$-invariant line of $V^\sigma$.
Moreover, $P$ is the stabilizer of $\mathbb Fv^\sigma$ in $S$, and $g\circ v=\sigma(\lambda(g)) v^\sigma$ for all $g\in P$.
From this, $[V^\sigma]\leftrightarrow (P,\lambda^\sigma)$.

On the other hand, let $(Q,\varphi)$ be a $p$-weight with $(Q,\varphi)\leftrightarrow (P,\lambda)$.
As $\varphi=\textup{St}_L\lambda$ and the Steinberg character takes rational value, we see that $(Q,\varphi^\sigma)\leftrightarrow (P,\lambda^\sigma)$.
As a consequence, the map $\Omega$ is $\hH$-equivariant.

(III) Now we proceed to verify the condition (2) of Proposition~\ref{def:ind-GAW-cond-2}. Note that $\ZZ(S)$ is of $p'$-order. 
Fix $\eta\in\IBr(S)$ and a simple $\mathbb S$-module $V$ affording $\eta$.
For simplicity, we first assume that $\ker\eta=1$.
Identify $V$ to $\mathbb F^n$, where $n$ is the dimension of $V$.
Then we have a faithful irreducible $\mathbb F$-representation \[\Theta:S\to\GL(n,\mathbb F)\] corresponding to $V$.
Then $\det(\Theta(s))=1$ for all $s\in S$ since $S$ is perfect.
As $\hH$ acts naturally on $\GL(n,\mathbb F)$, we will consider the semidirect product $\GL(n,\mathbb F)\rtimes \hH$ in the following.

For every $\sigma x\in (\hH\times\Aut(S))_\eta$ with $\sigma\in\hH$ and $x\in\Aut(S)$, there exists $\Theta_{\sigma x}\in\GL(n,\mathbb F)$ such that \[\Theta(x(g))^\sigma=\Theta_{\sigma x}\Theta(g)\Theta_{\sigma x}^{-1}\] for all $g\in S$.
Define \[G=\{\alpha\Theta_{\sigma x}\which  \alpha\in\mathbb F^\times, \sigma\in\hH, x\in\Aut(S), \sigma x\in (\hH\times\Aut(S))_\eta, \det(\alpha \Theta_{\sigma x})=1\}.\]
As in \cite[p.554]{Nav1}, one shows that $G$ is a subgroup of $\GL(n,\mathbb F)$.
Furthermore, $G$ is a finite group with $|G|\le n\cdot |(\hH\times\Aut(S))_\eta|$.
Also, $\Theta(S)\subseteq G$. In what follows we identify $\Theta(S)$ with $S$.

Set \[A=\{\alpha\Theta_x\in G\which  \alpha\in\mathbb F^\times,x\in\Aut(S)_\eta\}.\] 
As above, it can be shown that $A$ is a group with $S\unlhd A\le G$.
By the irreducibility of $\Theta$, we know that $\ZZ(S)\le\C_A(S)=\ZZ(A)<\ZZ(\GL(n,\mathbb F))$ which also implies that $\ZZ(A)$ is a $p'$-group.

Set 
\[\tilde A=\{\sigma^{-1}\cdot(\alpha\Theta_{\sigma x})\in\GL(n,\mathbb F)\rtimes\hH\which  \alpha\in\mathbb F^\times, \sigma\in\hH, x\in\Aut(S), \alpha\Theta_{\sigma x}\in G\}.\]
Consider any $\sigma^{-1}\cdot(\alpha\Theta_{\sigma x})$, $\tau^{-1}\cdot(\beta\Theta_{\tau y})\in \tilde A$, where $\alpha,\beta\in\mathbb F^\times$, $\sigma,\tau\in\hH$, $x,y\in\Aut(S)$ such that $\alpha\Theta_{\sigma x},\beta\Theta_{\tau y}\in G$.
Then $\sigma^{-1}\cdot(\alpha\Theta_{\sigma x})\cdot\tau^{-1}\cdot(\beta\Theta_{\tau y})=(\tau\sigma)^{-1}\cdot(\alpha\Theta_{\sigma x})^\tau(\beta\Theta_{\tau y})$ and
\begin{align*}
  &(\Theta_{\sigma x})^\tau\Theta_{\tau y}\Theta(g)\Theta_{\tau y}^{-1}(\Theta_{\sigma x}^{-1})^\tau=(\Theta_{\sigma x})^\tau\Theta(y(g))^\tau(\Theta_{\sigma x}^{-1})^\tau\\
  & =(\Theta(xy(g))^{\sigma})^{\tau}=\Theta_{\tau\sigma xy}\Theta(g)\Theta_{\tau\sigma xy}^{-1}.
\end{align*}
From this 
$(\Theta_{\sigma x})^\tau\Theta_{\tau y}\Theta_{\tau\sigma xy}^{-1}$ centralizes $\Theta(S)$. Schur's Lemma implies that $(\Theta_{\sigma x})^\tau\Theta_{\tau y}=\delta\Theta_{\tau\sigma xy}$ for some $\delta\in\mathbb F^\times$.
As $\det((\alpha\Theta_{\sigma x})^\tau(\beta\Theta_{\tau y}))=1$, one gets $(\alpha\Theta_{\sigma x})^\tau(\beta\Theta_{\tau y})\in G$.
Therefore, $\tilde A$ is a finite subgroup of $\GL(n,\mathbb F)\rtimes \hH$. 

By construction, $S\unlhd\tilde A$, $\widetilde{A}/\C_{\widetilde{A}}(S)\cong \Aut(S)_{\eta^{\hH}}$ and $A/\C_A(S)\cong \Aut(S)_{\eta}$.
Let $\tilde\eta$ be the Brauer character afforded by the inclusion $A\le\GL(n,\mathbb F)$. Then $\tilde\eta$ is an extension of $\eta$.
By the irreducibility of $\eta$ and Schur's lemma, $\ZZ(S)\le\ZZ(A)\le\ZZ(\GL(n,\mathbb F))$ and $\C_A(S)=\ZZ(A)$.

Suppose that $V\leftrightarrow(P,\lambda)$, $Q=\bO_p(P)$ and $\overline{(Q,\varphi)}=\Omega(\eta)$. Write $P=Q\rtimes Y$ with a Levi subgroup $Y$ of $S$ so that $\varphi=\lambda\cdot\textup{St}_Y$.
Let $\sigma a\in(\hH\times\N_{\widetilde{A}}(Q))_\eta$. Write $\tilde\eta^{\sigma a}=\mu\tilde\eta$, where $\mu$ are linear characters of $A/S$ with $p'$-order. 
Then by (I) and (II), we know that $(P,\lambda)=(P^{as},\lambda^{\sigma as})$, where $s\in S$ satisfying that $U^a=U^s$. 
As $a\in\N_{\widetilde{A}}(Q)$, we have $P=P^{as}=P^s$ which forces $s\in P$. This implies that $\lambda^{\sigma a}=\lambda$.
Moreover, $v^\sigma$ is the $B^{a}$-invariant line in $V^{\sigma a}$ by (II) and the proof of \cite[Theorem~C]{Nav1}.

View $\vartheta:=\textup{St}_Y$ as a basic $p$-Steinberg character (see\cite{Fe93} for the definition) of $\N_S(Q)/Q$. As shown in the proof of \cite[Theorem~C]{Nav1} (using \cite[Theorem~B]{Fe93}), $\vartheta$ extends to a rational character $\tilde\vartheta$ of $\N_{\widetilde A}(Q)/Q$.
Let $\mathbb Fv$ be the $B$-invariant line in $V$. By the proof of \cite[Theorem~C]{Nav1}, $\mathbb Fv$ is fixed by $\N_{A}(Q)/Q$, that is, as a linear character of $\N_S(Q)$, $\lambda$ has an extension $\tilde\lambda$ to $\N_{A}(Q)$.
Thus $\tilde\varphi:=\tilde\lambda\tilde\vartheta$ is an extension of $\varphi$ to $\N_A(Q)$.
In particular, $\tilde\varphi^{\sigma a}=\tilde\lambda^{\sigma a}\tilde\vartheta$.
Write $\tilde\varphi^{\sigma a}=\mu'\tilde\varphi$, where $\mu'$ are linear characters of $\N_A(Q)/\N_{S}(Q)$ with $p'$-order. 
Then $\tilde\lambda^{\sigma a}=\mu'\tilde\lambda$.  
By the proof of \cite[Theorem~C]{Nav1}, the condition (2.b) of Proposition~\ref{def:ind-GAW-cond-2} holds.

As $\tilde\eta^{\sigma a}=\mu\tilde\eta$, there exists a matrix $M\in\textup{GL}(n,\mathbb F)$ such that $(\alpha\Theta_x)^{\sigma a}=M(\mu(\alpha\Theta_x)\alpha\Theta_x)M^{-1}$ for every $\alpha\Theta_x\in A$.
Moreover, $v$ is an eigenvector of any matrix $\alpha\Theta_x\in \N_A(Q)$ corresponding to the eigenvalue $\tilde\lambda(\alpha\Theta_x)$ and thus $v$ is also an eigenvector of the matrix $\mu(\alpha\Theta_x)\alpha\Theta_x$ corresponding to the eigenvalue $\mu(\alpha\Theta_x)\tilde\lambda(\alpha\Theta_x)$.
On the other hand, $v^\sigma$ is an eigenvector of any matrix $(\alpha\Theta_x)^{\sigma a}$ corresponding to the eigenvalue $\tilde\lambda((\alpha\Theta_x)^{a\sigma})$.
From this, $\tilde\lambda^{\sigma a}=\mu\tilde\lambda$, and this implies that $\mu=\mu'$. Thus we have proved the condition (2) of Proposition~\ref{def:ind-GAW-cond-2} in the situation $\ker\eta=1$. 

For the case $\ker\eta>1$, we consider the group $S/\ker\eta$ instead of $S$ in the above arguments.
Then the groups $A$ and $\widetilde A$ can be defined analogously.
Moreover, the extendibility of Brauer characters in the condition (2) of Proposition~\ref{def:ind-GAW-cond-2} also holds as in the proof of \cite[Theorem~C]{Nav1}, while the condition (2.b) can also be established as above.
Thus we complete the proof.
\end{proof}


\section{Simple groups with abelian Sylow 2-subgroups}\label{sec_abelian2}

Let $G$ be a finite group and $Q$ be $p$-subgroup of $G$. We denote by $\IBrd(G\which Q)$ (this notation is adopted to distinguish from the notation $\IBr(G\which |Q|)$ used in the previous sections) the set of irreducible Brauer characters of $G$ lying in blocks with defect group $Q$.
If $G$ has a cyclic Sylow $p$-subgroup, then using \cite[Lemma~4.7]{Da96}, Navarro and Tiep \cite[Theorem~7.1]{Nav1} constructed a natural bijection \[f_{G,Q}:\IBrd(G\which Q)\to\IBrd(\N_G(Q)\which Q).\]
We aim to prove that this bijection satisfies several properties, as stated in Proposition \ref{lem:Dade-cyclic}.
Before proceeding, we introduce some lemmas; the main references are \cite{Alp-local-book} and \cite{Nagao}.
 If $B$ is a block of $G,$ then we denote by $\IBr(B)$ the set of irreducible Brauer characters of $G$ lying in the block $B$.

\begin{lem}\label{lem:cyc1}
  Let $G$ be a finite group with a cyclic Sylow $p$-subgroup. Let $Q$ be a non-trivial $p$-subgroup of $G$ and $Q_1$ be the unique subgroup of $Q$ of order $p.$ Let $B$ be a block of $G$ with defect group $Q$ and $b_1$ be the Brauer correspondent of $B$ in $H:=\N_G(Q_1).$ 
  \begin{enumerate}
    \item The simple $b_1$-modules have the same dimension and every indecomposable $b_1$-module is uniserial.
    \item If $V$ is an indecomposable  non-projective $B$-module and $W$ is the ($G,Q,H$-)Green correspondent of $V,$ then $W$ belongs to $b_1$ and as an $H$-module it is the unique direct summand of $V_H$ that is non-projective and belongs to a block with defect group $Q.$
    \item There is a natural bijection $$\mathscr{G}:\IBr(B)\ra\IBr(b_1)$$ such that $\mathscr{G}(\varphi),$ $\varphi\in\IBr(B),$ is the Brauer character of $\soc(W),$  where $W$ is the Green correspondent of the simple $G$-module $V$ that corresponds to $\varphi.$
  \end{enumerate}
\end{lem} 
\begin{proof}
Condition (1) follows from \cite[Corollary 19.8]{Alp-local-book} and the construction of the simple modules in $b_1$. Note that $b_1$ is the unique block of $H$ inducing $B$; therefore, by \cite[Chapter V, Theorem 3.10]{Nagao} and the properties of the Green correspondence \cite[Chapter IV, Theorem 4.3]{Nagao}, we obtain that in condition (2), the module $W$ belongs to the block $b_1$.
If $W'$ is another indecomposable $H$-module which is a direct summand of $V_H$, then $W'$ has a vertex $P \in_H \mathfrak{Y} := \mathfrak{Y}(G, Q, H)$ (see \cite[Chapter IV, Section 4]{Nagao} for the definition). Suppose that $W'$ is non-projective and belongs to a block with defect group $Q$; then we may assume that $P$ is a non-trivial subgroup of $Q$. 
By \cite[Chapter IV, Lemma 4.1]{Nagao}, we have $P \in \mathfrak{Y}$. However, this contradicts the definition of $\mathfrak{Y}$, since no non-trivial subgroup of $Q$ is contained in $\mathfrak{Y}$. This establishes condition (2).
Condition (3) is precisely the statement of \cite[Theorem 21.1]{Alp-local-book}.
\end{proof}

\begin{lem}\label{lem:cyc2}
   Let $G$ be a finite group with a cyclic Sylow $p$-subgroup.
Let $Q$ be a non-trivial $p$-subgroup of $G$ and  $Q_1$ be the unique subgroup of $Q$ of order $p.$
Then we have
\begin{enumerate}
 \item There exists a natural bijection $$\mathscr{G}_{G,Q}:\IBrd(G\which Q)\to\IBrd(\N_G(Q_1)\which Q)$$ that is $\hH\times\Aut(G)_Q$-equivariant and corresponding characters lie in Brauer corresponding blocks.
\item Further assume that $G$ is embedded as a normal subgroup of $\widetilde{G}$ such that $\widetilde{G}/G$ is a $p'$-group, and that $\eta \in \IBrd(G \which Q)$ is a $\widetilde{G}$-invariant irreducible Brauer character of $G$ which extends to an irreducible Brauer character $\tilde{\eta}$ of $\widetilde{G}$. Then
\begin{enumerate} 
\item $\mathscr{G}_{\widetilde G,Q}(\tilde\eta)$ is an extension of $\mathscr{G}_{G,Q}(\eta)$.
\item Let $\lambda$ be a linear character of $\widetilde G$ with $G\subseteq\ker(\lambda)$. Regard $\lambda$ also as a linear character of $\N_{\widetilde G}(Q_1)$ by restriction.
Then $\mathscr{G}_{\widetilde G,Q}(\lambda\tilde\eta)=\lambda \mathscr{G}_{\widetilde G,Q}(\tilde\eta)$.
\end{enumerate}
\end{enumerate}
\end{lem}
\begin{proof}
  The bijection $\mathscr{G}_{G,Q}$ is constructed by putting together the bijections in Lemma \ref{lem:cyc1} (3) for all the blocks of $G$ with defect group $Q.$ To see this bijection is $\hH\times\Aut(G)_Q$-equivariant, we can use the uniqueness properties of the Green correspondence, as stated in Lemma \ref{lem:cyc2} (2). This proves the condition (1) of the lemma.
  
  Now we are in the setting of the condition (2). By Fong's theorem \cite[Chapter V, Theorem 5.16]{Nagao}, we have $\tilde{\eta}\in\IBrd(\widetilde{G}\which Q).$  Let $\widetilde{H}=\N_{\widetilde{G}}(Q_1)$ and $H=\N_G(Q_1).$ As $\eta$ is $\widetilde{G}$-invariant we have $\widetilde{G}=G\widetilde{H}.$ Let $\widetilde{V}$ be an indecomposable $\widetilde{G}$-module affording the irreducible Brauer character $\tilde{\eta}.$  Thus by our assumption $V:=\widetilde{V}_G$ is an irreducible $G$-module affording the Brauer character $\eta.$ 
  Let ${\widetilde{W}}$ be the $(\widetilde{G},Q,\widetilde{H})$-Green correspondent of $\widetilde{V}.$

  By Lemma \ref{lem:cyc1} (2) we have $\widetilde{V}_{\widetilde{H}}\cong \widetilde{W}\oplus \widetilde{W}',$ where $\widetilde{W}'$ is a direct sum of projective $\widetilde{H}$-modules and modules lying in blocks of $\widetilde{H}$ with defect group other than $Q.$ Noticing that $V_H=(\widetilde{V}_{\widetilde{H}})_H\cong \widetilde{W}_H\oplus \widetilde{W}'_H$ and  $\widetilde{W}'_H$ is a direct summand of projective $H$-modules and modules lying in blocks with defect group other than $Q,$ we have that the $(G,Q,H)$-Green correspondent $W$ of $V$ is isomorphic to a direct summand of $\widetilde{W}_H.$
  We want to prove that $W\cong \widetilde{W}_H.$ Otherwise $\widetilde{W}_H\cong W\oplus W'$ for some nonzero $H$-module $W'.$ By Lemma \ref{lem:cyc1} (1) and the fact that $\tilde{\eta}\in\IBrd(\widetilde{G}\which Q),$ we have that $\widetilde{W}$ belongs to a block with defect group $Q.$ Hence the direct summands of $\widetilde{W}_H$ all belong to blocks with defect group $Q.$ By Lemma \ref{lem:cyc1} (2), we have that $W'$ is projective.  On the other hand, as any irreducible $\widetilde{H}$-module restricts to a completely irreducible $H$-module, we have that the radical length of $\widetilde{W}_H$ is no larger than that of $\widetilde{W}.$ 
  Recall that an indecomposable $\widetilde{H}$-module lying in a block with defect group $Q$ is projective if and only if it has the maximal radical length, which is exactly the order of $Q$ (see \cite[Theorem 19.1, Corollary 19.8]{Alp-local-book}). 
  The same happens with indecomposable $H$-modules lying in blocks with defect group $Q.$ Since $\widetilde{W}$ is not projective, its radical length is less than $|Q|.$ By the analysis above, the radical length of $\widetilde{W}_H$ is less than $|Q|,$ which contradicts the fact that $W'$ is nonzero and projective. 
  We have proved that $\widetilde{W}_H\cong W.$ Thus $\widetilde{W}_H$ is indecomposable and uniserial. This leads to the fact that $\soc(\widetilde{W})_H\cong \soc(W).$ We just proved the condition (a). 
  
  If $Z$ is the one-dimensional $\widetilde{G}$-module that correspondents to $\lambda.$ Then $Z\otimes\widetilde{V}$ is an irreducible $\widetilde{G}$-module that correspondents to $\lambda\tilde{\eta}.$ Since $Z_{\widetilde{H}}\otimes\widetilde{W}$ is an indecomposable direct summand of $(Z\otimes \widetilde{V})_{\widetilde{H}}$ that is non-projective and belongs to a block with defect group Q, by the uniqueness properties in Lemma \ref{lem:cyc1} (2) we have that $Z_{\widetilde{H}}\otimes \widetilde{W}$ is the Green correspondent of $Z\otimes\widetilde{V}.$ As $\soc(Z_{\widetilde{H}}\otimes\widetilde{W})=
  Z_{\widetilde{H}}\otimes\soc(\widetilde{W})$ we have $\mathscr{G}_{\widetilde G,Q}(\lambda\tilde\eta)=\lambda \mathscr{G}_{\widetilde G,Q}(\tilde\eta)$ by our definition.
\end{proof}

\begin{prop}\label{lem:Dade-cyclic}
Let $G\zg \widetilde A$ be finite groups, and assume that $G$  has a cyclic Sylow $p$-subgroup. 
Let $Q$ be a $p$-subgroup of $G.$ 
Then for every subgroup $\widetilde{G}$ of $\widetilde{A}$ such that $G\subseteq \widetilde{G}$  and $\widetilde{G}/G$ is a $p'$-group, there exists an $\hH\times\N_{\widetilde{A}}(\widetilde{G},Q)$-equivariant bijection 
$$
   f_{\widetilde{G},Q}:\IBrd(\widetilde{G}\which Q)\to\IBrd(\N_{\widetilde{G}}(Q)\which Q),
$$ where $\N_{\widetilde{A}}(\widetilde{G},Q)= \N_{\widetilde{A}}(\widetilde{G})\cap \N_{\widetilde{A}}(Q),$
 such that corresponding characters lie in Brauer corresponding blocks. 
  Moreover, if $\eta\in\IBrd(G\which Q)$ extends to an irreducible Brauer character $\tilde{\eta}$ of $\widetilde{G}$ for such a $\widetilde{G},$ then 
\begin{enumerate}  
\item $f_{\widetilde G,Q}(\tilde\eta)$ is an extension of $f_{G,Q}(\eta)$.
\item Let $\lambda$ be a linear character of $\widetilde G$ with $G\subseteq\ker(\lambda)$. Regard $\lambda$ also as a linear character of $\N_{\widetilde G}(Q)$ by restriction.
We have $f_{\widetilde G,Q}(\lambda\tilde\eta)=\lambda f_{\widetilde G,Q}(\tilde\eta)$.
\end{enumerate}
\end{prop}

\begin{proof}
We argue by induction on $|\widetilde{A}|.$  Notice that $\widetilde{A}=G\N_{\widetilde{A}}(Q)$ since $G$ has a cyclic Sylow $p$-subgroup.  If $Q$ is trivial, the proposition is trivially satisfied. Thus we may assume that $Q\neq 1.$ Let $Q_1$ be the unique subgroup of $Q$ of order $p.$

We first consider the case that $Q_1$ is normal in $\widetilde{A}.$ 
We apply induction to $\widetilde{A}/Q_1.$ Thus for every subgroup $\widetilde{G}$ of $\widetilde{A}$ such that $G\subseteq \widetilde{G}$  and $\widetilde{G}/G$ is a $p'$-group, we have a bijection
$$
   f_{\widetilde{G}/Q_1,Q/Q_1}:\IBrd(\widetilde{G}/Q_1\which Q/Q_1)\to\IBrd(\N_{\widetilde{G}/Q_1}(Q/Q_1)\which Q/Q_1)
$$ 
that satisfies the desired properties.
By \cite[Lemma 3.3]{Nav-Ann}, a block $B$ of $\widetilde{G}$ with defect group $Q$ contains a unique block $\bar{B}$ of $\widetilde{G}/Q_1$ with defect group $Q/Q_1.$ 
Thus by inflation of characters we have $\IBrd(\widetilde G/Q_1\which Q/Q_1)=\IBrd(\widetilde G\which Q).$  
The same holds for blocks of $\N_{\widetilde{G}}(Q).$ 
We define bijections $f_{\widetilde G,Q}$ naturally from $f_{\widetilde G/Q_1,Q/Q_1}$, respectively. By \cite[Lemma 3.2]{Nav-Ann} and induction, the corresponding characters lie in Brauer corresponding blocks. Other conditions can be deduced by induction.

Let $\widetilde{A}_1:=\N_{\widetilde{A}}(Q_1)<\widetilde{A}.$ Note that $\widetilde{A}_1$ contains $\N_{\widetilde{A}}(Q).$ 
Let $\widetilde{G}$ be a subgroup of $\widetilde{A}$ such that $G\subseteq \widetilde{G}$  and $\widetilde{G}/G$ is a $p'$-group, and let $\widetilde{G}_1=\widetilde{A}_1\cap \widetilde{G}$.
By Lemma \ref{lem:cyc2},  there exists a natural bijection 
$$
  \mathscr{G}_{\widetilde{G},Q}:\IBrd(\widetilde{G}\which Q)\ra\IBrd(\widetilde{G}_1\which Q)
$$
satisfying the conditions listed there. Note that the map $\mathscr{G}_{\widetilde{G},Q}$ is $\hH\times\N_{\widetilde{A}}(\widetilde{G},Q)$-equivariant by Lemma \ref{lem:cyc2} (1).
By induction to $\widetilde{A}_1,$  there exists an $\hH\times\N_{\widetilde{A}}(\widetilde{G},Q)$-equivariant bijection 
$$
   f_{\widetilde{G}_1,Q}:\IBrd(\widetilde{G}_1\which Q)\to\IBrd(\N_{\widetilde{G}}(Q)\which Q)
$$
satisfying the desired properties. 
Define $f_{\widetilde{G},Q}=f_{\widetilde{G}_1,Q}\mathscr{G} _{\widetilde{G},Q}.$ 
By Lemma \ref{lem:cyc2} and induction, we can check that the maps $f_{\widetilde{G},Q},$ as $\widetilde{G}$ varies, have the desired properties.
\end{proof}

\begin{coro}\label{prop:cyclic-sylow}
Let $L$ be a finite non-abelian  simple group of order divisible by $p$ and  $S$ be the universal $p'$-covering group of $L$. Suppose that there exists a finite group $G$ such that $S\zg G$ and $G$ induces all automorphisms on $S$ by conjugation.
Assume that $G/S$ is cyclic, $\C_G(S)=\ZZ(G),$ and $S$ has a cyclic Sylow $p$-subgroup $P$. Then the inductive NAW condition holds for $L$ at the prime $p$.
\end{coro}

\begin{proof}
Let $\eta\in\IBr(S)$.  Assume that  $\eta\in\IBrd(S\which  Q)$ for some $p$-subgroup $Q$ of $P.$
By Proposition \ref{lem:Dade-cyclic}, we have that for any $\widetilde{G}\in\mathcal{S}(G,S)$ such that $\widetilde{G}/S$ is a $p'$-group,  there exists an $(\hH\times\N_G(\widetilde{G},Q))$-equivariant bijection $$f_{\widetilde{G},Q}:\IBrd(\widetilde{G}\which Q)\ra\IBrd(\N_{\widetilde{G}}(Q)\which Q)$$ satisfying the conditions listed there.
Notice that $$f_{S,Q}:\IBrd(S\which Q)\ra\IBrd(\N_S(Q)\which Q)$$ is an $\big(\hH\times\Aut(S)_Q\big)$-equivariant bijection and the set $\IBrd(\N_S(Q)\which Q)$ corresponds naturally to the set $\dz(\N_S(Q)/Q)$ by \cite[Lemma 3.3]{Nav-Ann}. Since $\dz(\N_S(Q)/Q)$ is non-empty (as the correspondent of $f_{S,Q}(\eta)$ belongs to it), we have $Q$ is a $p$-radical subgroup of $S.$ 
Define $\Omega(\eta)=\overline{(Q,f_{S,Q}(\eta))}$ for $\eta\in\IBrd(S\which Q).$
Since the corresponding characters under the map $f_{S,Q}$ lie in Brauer corresponding blocks,  this bijection respects central characters in $\ZZ(S).$ 
As $f_{S,Q}$ is $(\hH\times\Aut(S)_Q)$-equivariant, the condition (1) in Definition \ref{Def:ind-GAW} holds.

Let $\eta$ and $Q$ be as above and $\varphi=f_{S,Q}(\eta).$
Let $\widetilde A=G_{\eta^\hH}$. 
We will prove that $$(\widetilde{A},S,\eta)_{\hH}\geqslant_c (\N_{\widetilde{A}}(Q),\N_S(Q),\varphi)_{\hH}.$$ Let $A=G_{\eta}$ and  $A_1/S$ be the Hall $p'$-subgroup of $A/S.$
Since $A/S$ is cyclic, $\eta$ extends to an irreducible Brauer character $\tilde{\eta}$ of $A.$ 
Now $\tilde{\eta}_{A_1}$ is an extension of $\eta$ to $A_1,$ thus by Proposition \ref{lem:Dade-cyclic} (1), we have $f_{A_1,Q}(\tilde{\eta}_{A_1})$ is an extension of $\varphi=f_{S,Q}(\eta)$ to $\N_{A_1}(Q).$ Let $\tilde{\varphi}$ be the unique extension of $f_{A_1,Q}(\tilde{\eta}_{A_1})$ to $\N_A(Q).$ 
We are going to prove that $(\tilde{\eta},\tilde{\varphi})$ gives the above ordering of $\hH$-triples.
First note that $\C_{\widetilde{A}}(S)=\ZZ(A)$ and $\ZZ(A)_1/\ZZ(S)$ is the Hall $p'$-subgroup of $\ZZ(A)/\ZZ(S),$ where $\ZZ(A)_1:=A_1\cap\ZZ(A).$ 
Since $\ZZ(A)_1$ is a central subgroup of $A_1,$ and $\tilde{\eta}_{A_1},\tilde{\varphi}_{\N_{A_1}(Q)}$ lie in Brauer corresponding blocks, we see $\IBr(\ZZ(A)_1\which \tilde{\eta})=\IBr(\ZZ(A)_1\which \tilde{\varphi}).$ 
Because any irreducible character of $\ZZ(A)_1$ extends to an unique character of $\ZZ(A)$, we have $\IBr(\ZZ(A)\which \tilde{\eta})=\IBr(\ZZ(A)\which \tilde{\varphi}).$ 
Let $a\in(\hH\times \N_{\widetilde{A}}(Q))_\varphi$. Let $\tilde \eta^{a}=\nu\tilde\eta$ and $\tilde \varphi^a=\nu'\hat\varphi$, where $\nu,\nu'$ are linear Brauer characters of $A/S=\N_A(Q)/\N_{S}(Q).$ Then by Proposition \ref{lem:Dade-cyclic} (2) and the fact that $f_{A_1,Q}(\tilde \eta_{A_1}^a)=f_{A_1,Q}(\tilde \eta_{A_1})^a,$  we have   $\nu$ and $\nu'$ agree on $\N_{A_1}(Q),$ and hence agree on $\N_A(Q).$ 
We complete proof of the Lemma.
\end{proof}

\begin{remark}\label{rmk:lie}
Let $L=\textup{PSL}_2(q)$ be a simple group with non-exceptional Schur multiplier, i.e., $\textup{SL}_2(q)$ is the universal covering group of $L$. We claim that if $\textup{Out}(L)$ is cyclic, then the group $E$ in Corollary~\ref{prop:cyclic-sylow} exists.
Let $\mathcal G=\textup{SL}_2(\overline{\mathbb F}_q)$ and $F_0:\mathcal G\to\mathcal G$ be the Frobenius map induced by the map $x\to x^{q_0}$ in $\overline{\mathbb F}_q$. Here $q=q_0^f$ with a prime $q_0$ and a positive integer $f$. Set $F=F_0^f$. Then $\mathcal G^F=\textup{SL}_2(q)$.  
Let $\mathcal L:\mathcal G\to\mathcal G$ be the Lang map defined by $\mathcal L(g)=g^{-1}F(g)$.
Then $\mathcal L^{-1}(\ZZ(\mathcal G))$ is a finite group containing $\mathcal G^F$ as its normal subgroup, and $\textup{Aut}(\mathcal G^F)\cong (\mathcal L^{-1}(\ZZ(\mathcal G))\rtimes \langle F_0\rangle)/\ZZ(\mathcal G^F)$.
Therefore, $\mathcal G^F/\ZZ(\mathcal G^F)_p$ is the universal $p'$-covering group of $L$, and
we can take $E=(\mathcal L^{-1}(\ZZ(\mathcal G))\rtimes \langle F_0\rangle)/\ZZ(\mathcal G^F)_p$.
\end{remark}

\begin{prop}\label{lem:Janko}
The inductive NAW condition holds for the Janko group $J_1$ at every prime.
\end{prop}

\begin{proof}
  Let $S=J_1$.
Note that $|S|=2^3\cdot 3\cdot 5\cdot 7\cdot 11\cdot 19$, thus by Corollary~\ref{prop:cyclic-sylow} we only need to consider the prime 2.
Since the Schur multiplier and outer automorphism group of $J_1$ are trivial (see e.g. \cite[p.36]{CCNPW85}), it suffices to establish a $\hH$-equivariant bijection between $\IBr(S)$ and $\cW(S)$.
By \cite[p.82]{JLPW} and \cite[Table 12]{AD12}, we get $|\IBr(S)|=\cW(S)|=11$, i.e., the original Alperin weight conjecture holds for $S$.

According to the character table of $S$ (cf. \cite[p.36]{CCNPW85}), the group $S$ has five blocks of defect zero, that is, the blocks containing $\chi_2$, $\chi_3$, $\chi_9$, $\chi_{10}$, $\chi_{11}$ (under the notation of \cite[p.36]{CCNPW85}).
Under the notation of \cite[p.82]{JLPW}, the irreducible Brauer characters of $S$ lying in blocks of defect zero are $\varphi_5$, $\varphi_6$, $\varphi_9$, $\varphi_{10}$, $\varphi_{11}$; we denote the set consisting of these five Brauer characters by $I$. 
Let $W$ denote the set of conjugacy classes of weights lying in blocks of defect zero. Then if $\overline{(Q,\varphi)}\notin W$, then $Q$ is cyclic of order 2 or an elementary abelian 2-group of order 8. 
The action of $\hH$ on the irreducible Brauer characters and the conjugacy classes of weights in blocks of defect zero are induced by the action of  $\hH$ on those blocks.
Hence the inductive NAW condition holds if we show that there is a $\hH$-equivariant bijection between $\IBr(S)\setminus I$ and $\cW(S)\setminus W$.

Using the Brauer character table of $S$ in \cite[p.82]{JLPW} we can check directly that $\hH$ fixes every $\varphi_k$, if $k=1,2,7,8$. 
Let $\sigma_2$ denote the $2$-Frobenius automorphism. The Brauer characters $\varphi_3$ and $\varphi_4$ are switched by $\sigma_2$.
Under the notation of \cite[Table 12]{AD12}, the conjugacy class of weights with radical subgroup $C_2$ are fixed by $\hH$.
The radical subgroup $Q=(C_2)^3$ afforded five weights as $N_S(Q)/Q\cong C_7\rtimes C_3$. For convenience we refer the character table of $C_7\rtimes C_3$ to \cite[p.303]{We16} and use the notation therein.
Then the characters $\chi_{1}$, $\chi_{3a}$, $\chi_{3b}$ are fixed by $\hH$, while $\chi_{1a}$ and $\chi_{1b}$ are switched by $\sigma_2$.

Therefore, we can get a $\hH$-equivariant bijection between $\IBr(S)\setminus I$ and $\cW(S)\setminus W$, which completes the proof.
\end{proof}

We will make use of the so-called unitriangular basic set (for its definition, see e.g. \cite[Definition~3.8]{FS23}).

\begin{lem}\label{lem:unitri}
  Let $A$ be a finite group, $G$ a normal subgroup of $A$, and $B$ a union of some blocks of $G$. Suppose that there exists a unitriangular $(A\times \hH)$-stable basic set $X\subseteq\Irr(B)$. Then the following statements hold.
  \begin{enumerate}
  \item There exists an  $(A\times \hH)$-equivariant bijection $f$ between $X$ and $\IBr(B)$.
\item Let $\chi\in X$ and $\psi=f(\chi)$. If $\chi$ has an $(A\times \hH)_\chi$-invariant extension to $A_\chi$, then $\psi$ has an $(A\times \hH)_\psi$-invariant extension to $A_\psi$.
\end{enumerate}
\end{lem}  

\begin{proof}
The assertion (1) follows by \cite[Lemma~2.1]{Na23} immediately.
Write \[\chi^{\circ}=\psi+\sum_{\phi\in\IBr(B)\setminus \{\psi\}}d_\phi\phi\] with $d_\phi\in\mathbb Z_{\ge 1}$.
By (1), $(A\times \hH)_{\chi}=(A\times \hH)_{\psi}$.
Suppose that $\tilde\chi$ is an $(A\times \hH)_\chi$-invariant extension of $\chi$ to $A_\chi$.
As in the proof of \cite[Lemma~2.9]{FLZ19}, there is a unique extension $\tilde\psi\in\IBr(A_\chi\which  \psi)$ of $\psi$ to $A_\chi$ such that $\tilde\psi$ is an irreducible constituent of $\tilde\chi^{\circ}$. 
Now $\chi$, $\psi$ and $\tilde\chi$ are all $(A\times \hH)_\chi$-invariant, so we get that $\tilde\psi$ is also $(A\times \hH)_\psi$-invariant by construction.
\end{proof}  

\begin{prop}\label{prop:sl2}
The inductive NAW condition holds for simple groups $\textup{PSL}_2(q)$ with $2\which  q$ or $q\equiv \pm 3\ \textup{(mod 8)}$ at every prime $p$.
\end{prop}

\begin{proof}
  Let $L=\textup{PSL}_2(q)$ with $2\which  q$ or $q\equiv \pm 3\ \textup{(mod 8)}$.  Let $q_0$ denote the prime such that $q=q_0^f$ with $f\ge 1$. Then the outer automorphism group of $L$ is cyclic, that is, $\Aut(L)\cong C_{2f}$. Note that $q>3$ since $L$ is simple.
According to Theorem \ref{thm:def-char}, we may suppose that $p\nmid q$. 

From now on we assume that $2\nmid q$ and $p=2$, as otherwise the Sylow $p$-subgroups of $L$ are cyclic and the assertion follows by  Corollary~\ref{prop:cyclic-sylow} and Corollary~\ref{rmk:lie}.
Let $S=\textup{SL}_2(q)$. Then $S$ is the universal covering group of $L$ since $q$ is odd and $q\equiv \pm 3\ \textup{(mod 8)}$.
Let $F_0$ denotes the field automorphism of $\textup{GL}_2(\overline{\mathbb F}_q)$ induced by the map $x\mapsto x^{q_0}$ in $\overline{\mathbb F}_q$.
In \cite{FLZ21,FLZ23} or \cite[Proposition~8.5]{Nav1}, an $\Aut(S)$-equivariant bijection $\Omega:\IBr(S)\to\cW(S)$, which preserves blocks, is established. We will prove this bijection $\Omega$ satisfies the conditions of Proposition~\ref{def:ind-GAW-cond-2} in the following.

(I) First we show that $\Omega$ is $\hH$-equivariant which implies that the condition (1) of Proposition~\ref{def:ind-GAW-cond-2} holds. The irreducible Brauer characters and weights are described in the proof of \cite[Proposition~8.5]{Nav1}. 
By the construction there, for any non-principal 2-block $B$ of $S$, one has that $|\IBr(B)|=|\cW(B)|=1$, and thus the actions of $\hH$ on $\IBr(S)\setminus\IBr(B_0)$ and $\cW(S)\setminus\cW(B_0)$ are induced by the action of $\hH$ on blocks.
Here $B_0$ denotes the principal 2-block of $S$.
Thus to prove that $\Omega$ is $\hH$-equivariant, it suffices to show that $\Omega|_{\IBr(B_0)}:\IBr(B_0)\to\cW(B_0)$ is $\hH$-equivariant.

As in the Introduction, we write $\hH=\langle\sigma_2\rangle$. Abbreviate $\sigma_2$ to $\sigma$ in the following.
The principal 2-block $B_0$ consists of three irreducible Brauer characters: the principal Brauer character $\eta_0$, and two Brauer characters $\eta_1$, $\eta_2$ of degree $\frac{q-1}{2}$ which are $\sigma$-conjugate. In fact, this can be checked directly by the character table and decomposition matrix of $S$ determined in \cite{Bu76}. 
Let $P$ denote a Sylow 2-subgroup of $S$, which is isomorphic to the quaternion group of order 8.
Under the notation of \cite{FLZ21}, $P=R_{1,0,1}$, and then $\N_S(P)/P\cong C_3$. From this, $\sigma$ conjugates the two non-principal irreducible characters of $\N_S(P)/P$.
Therefore, the bijection $\Omega|_{\IBr(B_0)}:\IBr(B_0)\to\cW(B_0)$ is $\hH$-equivariant.

(II) We prove that every character $\chi\in\Irr(S)$ has an $(\Aut(G)\times\hH)_\chi$-invariant extension to $S\rtimes\langle F_0\rangle_\chi$.
We can divide the irreducible characters of $S$ into three classes: (i) two unipotent characters, (ii) irreducible characters of degree $q+1$ or $\frac{q+1}{2}$, (iii) irreducible characters of degree $q-1$ or $\frac{q-1}{2}$.

Let $\chi\in\Irr(S)$ be a character in class (i) or (ii). Then $\chi$ lies in a principal series of $S=\textup{SL}_2(q)$. 
Suppose that $(T,\delta)$ is a cuspidal pair such that $\chi$ lies in the Harish-Chandra series $\cE(S,(T,\delta))$.
Then $T\cong C_{q-1}$, $|\N_S(T):T|=2$. Also, $T$ is $\langle F_0\rangle$-stable and $\N_{S\langle F_0\rangle}(T)/\C_{S\langle F_0\rangle}(T)\cong C_{2f}$ which implies that $\delta$ extends to $(\N_{S\langle F_0\rangle}(T))_\delta$.
 Moreover, $\delta$ is $\langle F_0 \rangle_\chi$-invariant.
The linear character $\delta$ has an extension $\hat\delta$ to $T\rtimes\langle F_0\rangle_\delta$ such that $\hat\delta(F_0^k)=1$ if $F_0^k\in\langle F_0\rangle_\delta$.
If $a\iota\in(\Aut(G)\times\hH)_\chi$ and $\delta^{a\iota x}=\delta$ for some $x\in\N_S(T)$, then $\hat\delta^{a\iota x}=\hat\delta$.
Thus by \cite[Proposition~8.7]{Jo22}, a generalization of \cite[Proposition~2.6 and Corollary~2.7]{RS22}, the character $\chi$ has a $(\Aut(G)\times\hH)_\chi$-invariant extension to $S\rtimes\langle F_0\rangle_\chi$.

For the characters in the class (iii), we regard $S$ as $\textup{SU}_2(q)$, which is isomorphic to $\textup{SL}_2(q)$, then every character of $S$ of degree $q-1$ or $\frac{q-1}{2}$ lie in some principal series of $S=\textup{SU}_2(q)$. 
Thus analogously as above we can show that every character $\chi\in\Irr(S)$ lying in the class (iii) has a $(\Aut(G)\times\hH)_\chi$-invariant extension to $S\rtimes\langle F_0\rangle_\chi$.

(III) Let $G=\textup{GL}_2(q)$. By \cite[Theorem~2.5.1]{GLS98}, the group $G\rtimes\langle F_0 \rangle$ induces all automorphisms of $S$, i.e., $\Aut(S)\cong G/\ZZ(G)\rtimes\langle F_0 \rangle$. By \cite[Thmeorem~4.1]{CS17}, one has $(G\rtimes\langle F_0 \rangle)_\chi=G_\chi\rtimes\langle F_0 \rangle_\chi$ for every $\chi\in\Irr(S)$.

Now we prove that every Brauer character $\eta\in\IBr(S/\ZZ(S))$ has a $(G\langle F_0\rangle\times\hH)_\eta$-invariant extension to $(G/\ZZ(G)\rtimes\langle F_0\rangle)_\eta$.
As $S$ has an $\Aut(S)$-stable unitriangular basic set (cf. \cite{Bu76} or \cite{De17}), we may transfer to ordinary characters.

Using the results in \cite[\S4]{De17}, we can check directly that $\cE(S,2')\setminus \Irr(B_0)$ form an $(\Aut(S)\rtimes\hH)$-stable unitriangular basic set of the union of non-principal blocks of $S$.
Here $\cE(S,2')$ is defined as in \cite[p.199]{CE04}.
Let $\chi\in \cE(S,2')\setminus \Irr(B_0)$. Then $\ZZ(S)\subseteq\ker(\chi)$. 
Then as a character of $S/\ZZ(S)$, the character $\chi$ has an extension $\tilde\chi$ to $G_\chi/\ZZ(G)\rtimes D_\chi$ since $\textup{Out}(L)$ is cyclic.
If $\chi$ is not $G$-invariant, then this follows from (II) immediately.
So we assume that $\chi$ is $G$-invariant.
Let $\tau$ be the restriction of $\tilde\chi$ to $G/\ZZ(G)$ and let $\tau'$ be another extension of $\chi$ to $G/\ZZ(G)$.
Also regard $\tau$, $\tau'$ as characters of $G$, then without loss of generality there exists a semisimple $2'$-element $s\in G$ such that $\tau$ lies in the Lusztig series $\cE(G,s)$ and $\tau'$ lies in $\cE(G,-s)$.
By a theorem of Srinivasan--Vinroot \cite{SV20}, $\tau^{\sigma^k}\in\cE(G,s^{2^k})$ for $k\ge 1$.
Since $s^{2^k}$ is of $2'$-order while $-s$ is not of $2'$-order, $\tau^{\sigma^k}\ne\tau'$ for $k\ge 1$.
This implies that $\hH_\chi=\hH_\tau=\hH_{\tau'}$, and thus $((G\rtimes\langle F_0\rangle)\times\hH)_\chi=((G\rtimes\langle F_0\rangle)\times\hH)_\tau$. 
By \cite[Remark~4.1(b)]{Jo24}, $\tau$ has a $(\langle F_0\rangle\times\hH)_\tau$-invariant extension to $G\rtimes\langle F_0\rangle_\tau$, and this gives a $((G\rtimes\langle F_0\rangle)\times\hH)_\chi$-invariant extension of $\chi$ to $(G\rtimes\langle F_0\rangle)_\chi$.
Therefore, by Lemma~\ref{lem:unitri}, every Brauer character $\eta\in\IBr(S/\ZZ(S))$ not lying in the principal block has a $(G\langle F_0\rangle\times\hH)_\eta$-invariant extension to $(G/\ZZ(G)\rtimes\langle F_0\rangle)_\eta$.

Now we consider the principal block $B_0$ of $S$. Then $B_0$ has a unitriangular basic set $\{ 1_S,\gamma_1,\gamma_2 \}$ such that $\gamma_1$ and $\gamma_2$ are both $\langle F_0\rangle$-invariant, $\gamma_1$ is $G$-conjugate to $\gamma_2$ and $\gamma_1$ and $\gamma_2$ are fused by $\sigma$.
By (II), $\gamma_i$ has a $(G\langle F_0\rangle\times\hH)_{\gamma_i}$-invariant extension to $S\rtimes\langle F_0\rangle_{\gamma_i}$ for $i=1,2$.
Thus $\eta_i$ has a $(G\langle F_0\rangle\times\hH)_{\eta_i}$-invariant extension to $S\rtimes\langle F_0\rangle_{\eta_i}$ for $i=1,2$.
As a Brauer character of $S/\ZZ(S)$, $\eta_i$ has a $(G\langle F_0\rangle\times\hH)_{\eta_i}$-invariant extension to $S/\ZZ(S)\rtimes\langle F_0\rangle_{\eta_i}=(G\rtimes\langle F_0\rangle)_{\eta_i}/\ZZ(G)$ for $i=1,2$.

(IV) Next, we show that for every weight $(Q,\varphi)$ of $S$, the Brauer character $\varphi^{\circ}\in\IBr(\N_S(Q)/Q)$ has an $(\N_{G\langle F_0\rangle}(Q)\times\hH)_{\varphi}$-invariant extension to $\N_{G\langle F_0\rangle}(Q)_{\varphi}/Q\ZZ(G)$.

By the classification of weights of $S$ (cf. \cite[\S5]{FLZ21} or the proof of \cite[Proposition~8.5]{Nav1}), $Q$ is $S$-conjugate to one of the three groups: $\ZZ(S)$, $R$, $P$. Here $R$ denotes a defect group of any non-principal block of $S$ which is of positive defect so that $R\cong C_4$.
If $Q=\ZZ(S)$, then this assertion follows by (III). 

Let $Q=R$. First assume that $4\which  (q-1)$. Then by \cite[\S4]{FLZ21}, $\N_S(R)$ is just the normalizer of the maximal split torus $T\cong C_{q-1}$ in $S=\textup{SL}_2(q)$ and $R$ is the Sylow 2-subgroup of $T$.
Let $(R,\varphi)$ be a weight of $S$.
Then $\varphi$ is $\N_G(R)$-invariant. Moreover, $\varphi$ is the induction of some character $\theta\in\Irr(T)$ to $\N_S(R)$. In addition, it can be checked directly $\langle F_0\rangle_{\varphi}=\langle F_0\rangle_{\theta}$.
Note that every character $\theta$ of $T$ has an $(\langle F_0\rangle\rtimes\hH)_\theta$-extension $\hat\theta$ to $T\rtimes \langle F_0\rangle_\theta$ by taking $\hat\theta(F_0^k)=1$ for any $F_0^k\in\langle F_0\rangle_\theta$.
By \cite[Corollary.~4.3]{Is84}, the induction $\hat\varphi$ of $\hat\theta$ to $\N_S(R)\rtimes \langle F_0\rangle_\varphi$ is an extension of $\varphi$.
One can show that $\hat\varphi$ is $(\langle F_0\rangle\rtimes\hH)_\varphi$-invariant.
Suppose that $b$ is a block $S$ such that $(R,\varphi)$ is a $b$-weight, and $B$ is a block of $G$ covering $b$.
By \cite[\S5]{FLZ21}, if $(R',\varphi')$ is a $B$-weight of $G$ covering $(R,\varphi)$, then $R'\cap S=R$, $\N_G(R)=\N_G(R')$ and $\varphi'\in\Irr(\N_G(R)\which \varphi)$. In fact, $\varphi'$ is an extension of $\varphi$.
Since we may regard $(R,\varphi)$ as a weight of $S/\ZZ(S)$, as in the arguments of (III), we may assume that  $(R',\varphi')$ is a weight of $G/\ZZ(G)$, and there is a semisimple $2'$-element $s$ such that $B\subseteq\cE_2(G,s)$ or $B\subseteq\cE_2(G,-s)$, where $\cE_2(G,s)$ is defined as in \cite[p.131]{CE04}.
In particular, $(\N_{G\langle F_0\rangle}(R)\times\hH)_\varphi=(\N_{G\langle F_0\rangle}(R)\times\hH)_{\varphi'}$.
So by \cite[Remark~4.1(b)]{Jo24}, $\varphi'$ has an $(\N_{G\langle F_0\rangle}(R)\times\hH)_{\varphi'}$-invariant extension to $\N_{G}(R)\rtimes\langle F_0\rangle_{\varphi'}$, which implies that the Brauer character $\varphi^{\circ}$ has an $(\N_{G\langle F_0\rangle}(R)\times\hH)_{\varphi}$-invariant extension to $\N_{G\langle F_0\rangle}(R)_{\varphi}$.
Now we assume that $4\which  (q+1)$, in which situation we regard $S$ as $\textup{SU}_2(q)$ and the proof is analogous as above.

Now we consider the Sylow 2-subgroup $P$. 
Then $F_0$ centralizes $P$, $\N_{S}(P)/P\cong C_3$ and $\N_{S\langle F_0\rangle}(P)/P=(\N_{S}(P)/P)\rtimes\langle F_0\rangle$.  The principal block is the unique block of $S$ with maximal defect. 
If $\varphi\in\Irr(\N_{S}(P)/P)$, then $\varphi$ has an extension $\hat\varphi$ to $(\N_{S}(P)/P)\rtimes\langle F_0\rangle$ with $\hat\varphi(F_0)=1$. Then the assertion holds.

(V) Now we define $A$ and $\widetilde A$ for any $\eta\in\IBr(S)$ and $\overline{(Q,\varphi)}=\Omega(\eta)$, and establish the condition (2) of Proposition~\ref{def:ind-GAW-cond-2}.  
Let $A=(G\rtimes\langle F_0\rangle)_\eta/\ZZ(G)$ and $\widetilde{A}=(G\rtimes\langle F_0\rangle)_{\eta^{\hH}}/\ZZ(G)$. Then the condition (2.a) of Proposition~\ref{def:ind-GAW-cond-2} holds.
As $S$ has cyclic outer automorphism group, the extendibility of $\eta$ and $\varphi$ is guaranteed, and the condition (2.b) of Proposition~\ref{def:ind-GAW-cond-2} is satisfied.
Finally, by (III), (IV) and Remark~\ref{rmk:add-condition} (2), the condition (2.c) of Proposition~\ref{def:ind-GAW-cond-2} holds, and this completes the proof.
\end{proof}

\begin{prop}\label{prop:ree}
The inductive condition holds for  the Ree simple groups $^2\textup{G}_2(3^{2n+1})$ (with $n\ge 1$) at every prime $p$.
  \end{prop}
  
  \begin{proof}
   Let $L={}^2\textup{G}_2(q^2)$ (with $n\ge 1$) where $q^2=3^{2n+1}$. Then the outer automorphism group of $L$ is cyclic.
   According to Theorem \ref{thm:def-char}, we may assume that $p\ne3$. Thanks to Corollary~\ref{prop:cyclic-sylow}, it remains to consider the case $p=2$, as the Sylow $p$-subgroups of $L$ are cyclic when $p$ is odd. 
As $L$ has trivial Schur multiple, we may assume $S=L$, which is the universal covering group of $L$.

From now on we assume that $p=2$.
In \cite[Proposition~8.7]{Nav1}, an $\Aut(S)$-equivariant bijection $\Omega:\IBr(S)\to\cW(S)$, which preserves blocks, is established. We will prove this bijection $\Omega$ satisfies the conditions of Proposition~\ref{def:ind-GAW-cond-2} in the following.

(I) We show that $\Omega$ is $\hH$-equivariant, which implies that the condition (1) of Proposition~\ref{def:ind-GAW-cond-2} holds.
Note that $|\IBr(B)|=|\cW(B)|=1$ if $B$ is a non-principal block of $S$ (see the proof of \cite[Proposition~8.4]{Nav1}).
Similar as in the proof of Proposition~\ref{prop:sl2}, we only need to show that $\Omega|_{\IBr(B_0)}:\IBr(B_0)\to\cW(B_0)$ is $\hH$-equivariant for the principal block $B_0$ of $S$.
As in the proof of Proposition~\ref{prop:sl2}, we write $\hH=\langle\sigma_2\rangle$ and abbreviate $\sigma_2$ to $\sigma$.

By the main result of \cite{LM80}, $1_S$, $\zeta_2$, $\zeta_3$, $\zeta_6$, $\zeta_8$ (in the notation if \cite{Wa66}) form a unitriangular basic set of $B_0$. 
Using the charter table of $S$ in the Chapter V of \cite{Wa66}, we can check that $1$, $\zeta_2$, $\zeta_3$ are $\hH$-invariant and $\zeta_6$ and $\zeta_8$ are fused by $\sigma$.
On the other hand, $\N_S(P)/P\cong C_7\rtimes C_3$, where $P\cong (C_2)^3$ is a Sylow 2-subgroup of $S$.
The action of $\hH$ on $\Irr(\N_S(P)/P)$ has been determined in Proposition~\ref{lem:Janko}; that is, three elements are $\hH$-invariant, while the other two are fused by $\sigma$.
So we can check directly that $\Omega|_{\IBr(B_0)}:\IBr(B_0)\to\cW(B_0)$ is $\hH$-equivariant.

(II) Let $\gamma$ be defined as in \cite[p.3]{Jo24}. Then by \cite[Theorem~2.5.1]{GLS98}, the group $S\rtimes\langle\gamma\rangle$ induces all 
automorphisms on $S$.

Let $\eta\in\IBr(S)$ and $\overline{(Q,\varphi)}=\Omega(\eta)$.
We define $A=S\rtimes\langle\gamma\rangle_\eta$ and $\widetilde A=S\rtimes\langle\gamma\rangle_{\eta^{\hH}}$.
Then the condition (2.a) of Proposition~\ref{def:ind-GAW-cond-2} holds.
As the outer automorphism group of $S$ is cyclic, the extendibility of $\eta$ and $\varphi$ is guaranteed, and the condition (2.b) of Proposition~\ref{def:ind-GAW-cond-2} is satisfied.
If $\eta$ lies in a block of $S$ of defect zero, then (2.c) of Proposition~\ref{def:ind-GAW-cond-2} holds automatically.
So it remains to verify the condition (2.c) of Proposition~\ref{def:ind-GAW-cond-2} for the case that $\eta$ lies in a block of $S$ of positive defect, which will be established as soon as we show Remark~\ref{rmk:add-condition} (2).

(III) Now we prove that every Brauer character $\eta\in\IBr(S)$ lying in a block of positive defect has a $(\langle\gamma\rangle\times\hH)_\eta$-invariant extension to $S\rtimes\langle\gamma\rangle_\eta$.
Using \cite{LM80}, \cite{Wa66} and the proof of \cite[Proposition~8.4]{Nav1}, we see that the irreducible characters $1_S$, $\zeta_2$, $\zeta_3$, $\zeta_6$, $\zeta_8$, $\eta_r$, $\eta_t$ (in the notation if \cite{Wa66}), where $1\ne r\in C_{(q^2-1)/2}<\mathbb C^\times$, $\pm1\ne t\in C_{(q^2+1)/2}<\mathbb C^\times$, form a unitriangular basic set of the union of blocks of $S$  of positive defect. 
Thanks to Lemma~\ref{lem:unitri}, we transfer to these ordinary characters.
By \cite{CHEVIE}, we can check that each of these characters is either unipotent, regular, real, or semisimple, and thus the assertion follows \cite[Proposition~3.6]{Jo22b} and \cite[Lemma~4.2 and Proposition 4.3, 4.4]{Jo24}.

(IV) We shall have established this proposition if we prove the following: for every weight $(Q,\varphi)$ of $S$ lying in a block of positive block, the character $\varphi\in\dz(\N_S(Q)/Q)$ has an $(\N_{S\langle\gamma\rangle}(Q)\times\hH)_{\varphi}$-invariant extension to $\N_{S\langle\gamma\rangle}(Q)_{\varphi}/Q$.

The group $S$ has four conjugacy classes of radical 2-subgroups: 1, $C_2$, $(C_2)^2$, $(C_2)^3$ (cf. \cite{Kl86}, \cite{Wa66}, see also the proof of \cite[Proposition~8.4]{Nav1}).
If $Q\cong C_2$, then $\N_S(Q)/Q\cong\textup{PSL}_2(q^2)$ (cf. \cite{Wa66}) and the assertion holds by the proof of Proposition~\ref{prop:sl2}. For the case that $Q\cong (C_2)^3$ is a Sylow 2-subgroup of $S$, this assertion has been established in the proof of \cite[Proposition~5.1]{Jo24}. 
Finally, we let $Q\cong(C_2)^2$ so that $\N_S(Q)/Q\cong(C_{(q^2+1)/4}\rtimes C_2)\rtimes C_3$ (cf. \cite[p.181]{Kl86}). 
Then the proof is similar as the arguments in (IV) of the proof of Proposition~\ref{prop:sl2}. Thus we complete the proof.
  \end{proof}

\begin{coro}\label{cor-abel-Sy}
The inductive NAW condition holds for  every non-abelian simple group with abelian Sylow 2-subgroups at every prime $p$.
\end{coro}  

\begin{proof}
Let $L$ be a non-abelian simple group with abelian Sylow 2-subgroups.
  According to Walter's classification \cite{Wa69}, the group $L$ has to be the first Janko group $J_1$, a Ree group ${}^2\textup{G}_2(3^{2n+1})$ (with $n\ge 1$), or a special linear group $\textup{PSL}_2(q)$ with $2\which  q$ or $q\equiv \pm 3\ \textup{(mod 8)}$. 
  Thus this assertion follows by Propositions~\ref{lem:Janko}, \ref{prop:sl2}, \ref{prop:ree}.
\end{proof}

Finally, we are able to prove the NAW weight conjecture for finite groups with abelian Sylow 2-subgroups.

\begin{proof}[Proof of Theorem~D]
  Since $G$ possesses abelian Sylow 2-subgroups, the same property holds for every non-abelian simple group involved in $G$.
  Thus by Theorem~\ref{thm-reduction}, to prove that the NAW conjecture holds for $G$, it suffices to show that the inductive NAW condition holds for every non-abelian simple group with abelian Sylow 2-subgroups, which follows from Corollary~\ref{cor-abel-Sy}.
\end{proof}

\bigskip
{\noindent\bf\large  Acknowledgement.}
\medskip

The author is grateful to the anonymous referee for carefully reading the manuscript and providing valuable suggestions that significantly enhanced the exposition.




\bibliographystyle{plain}

\end{document}